\documentclass[12pt, reqno]{amsart}
\usepackage{amsmath, amsthm, amscd, amsfonts, amssymb, graphicx, xcolor, mathtools}
\usepackage{subcaption}

\usepackage{tikz}
\usetikzlibrary{patterns, arrows.meta, decorations.markings}

\tikzset{
    blk/.style = {draw, circle, fill=black, text=white},
    wht/.style = {draw, circle, fill=white, text=black},
    mid_auto/.style = {midway, auto}
}

\usepackage{tikz-cd}
\usepackage{bm} %for \bm
\usepackage{comment} %for \begin{comment}
\usepackage{stmaryrd} %for \mapsfrom
\usepackage{extarrows} %for \xlongequal
\usepackage{array} %for notation index
\usepackage{booktabs} %for table

\makeatletter
\def\l@section{\@tocline{1}{0pt}{0pc}{}{}}
\def\l@subsection{\@tocline{2}{0pt}{1.5pc}{}{}}
\makeatother

\usepackage{CJKutf8}

\newtheorem{theorem}{Theorem}[section]
\newtheorem{lemma}[theorem]{Lemma}
\newtheorem{proposition}[theorem]{Proposition}
\newtheorem{corollary}[theorem]{Corollary}
\newtheorem{definition}[theorem]{Definition}
\newtheorem{notation}[theorem]{Notation}
\newtheorem{example}[theorem]{Example}

\theoremstyle{remark}
\newtheorem{remark}[theorem]{Remark}
\numberwithin{equation}{section}

\newcommand{\defsetsmall}[2]{ \{ #1 \ \lvert\ #2 \} }

\newcommand{\NN}{\mathbb{N}}
\newcommand{\ZZ}{\mathbb{Z}}
\newcommand{\Zpos}{\mathbb{Z}_+}

\newcommand{\CP}{\mathbb{CP}}

\newcommand{\Sf}{\mathfrak{S}}
\newcommand{\Nf}{\mathfrak{N}}

\renewcommand{\Mc}{\mathcal{M}}
\newcommand{\Cc}{\mathcal{C}}
\newcommand{\Uc}{\mathcal{U}}

\newcommand{\abs}[1]{ \left| #1 \right| }
\newcommand{\norm}[1]{ \left\| #1 \right\| }
\newcommand{\ol}[1]{ \overline{#1} }
\newcommand{\td}[1]{ \widetilde{#1} }
\newcommand{\ole}{ \overline{e} }
\newcommand{\tde}{ \tilde{e} }

\DeclareMathOperator{\wt}{W}
\newcommand{\bmwt}{ \mathbf{W} }
\newcommand{\bmmu}{ \bm{\mu} }
\newcommand{\bmt}{ \bm{t} }
\DeclareMathOperator{\Lab}{\mathfrak{L}}

\DeclareMathOperator{\Fill}{Fill}
\DeclareMathOperator{\CFF}{CFF}

\DeclareMathOperator{\cyc}{Cyc}
\DeclareMathOperator{\sg}{sg}
\DeclareMathOperator{\fix}{Fix}

\newcommand{\pp}{\Xi}
\newcommand{\ppqu}{\Xi_{QU}}
\newcommand{\ppzl}{\Xi_{0, L}}
\newcommand{\ppl}{\Xi_L}
\newcommand{\ppsi}{\Xi_{\sigma}}

\newcommand{\ppz}{\Xi_0}

\newcommand{\Ptt}{\mathfrak{P}}
\newcommand{\ptt}[1]{\mathfrak{#1}}
\newcommand{\pttp}{\mathfrak{p}}

\newcommand{\one}{\bm{1}}
\newcommand{\ff}{\mathrm{f}}

\DeclareMathOperator{\Aut}{Aut}
\DeclareMathOperator{\id}{id}

\DeclareMathOperator{\Epi}{Epi}
\DeclareMathOperator{\lcm}{lcm}

\DeclareMathOperator{\range}{range}

\makeatletter
\let\qedheremath\displaymath@qed    
\makeatother

\usepackage[bookmarksnumbered, colorlinks, plainpages,
citecolor=teal,
linktoc=all]{hyperref}

\begin{document}
%\centerline{}

\title
[Hurwitz numbers via enumeration of hypermaps]
{Hurwitz numbers of a fixed partition $(m, 1^{n-m})$ \\ via enumeration of unrooted hypermaps}

\author[S. Yi]{Yi Song}

\iffalse
\address{$^{1}$ School of Mathematical Sciences, Soochow University, Suzhou, Jiangsu, People's Republic of China.}
\email{\textcolor[rgb]{0.00,0.00,0.84}{sclu@suda.edu.cn}}

\address{$^{2}$ School of Mathematical Sciences, University of Science and Technology of China, Hefei, Anhui, People's Republic of China.}
\email{\textcolor[rgb]{0.00,0.00,0.84}{sif4delta0@mail.ustc.edu.cn}}
\fi

%\dedicatory{This paper is dedicated to Professor ABCD}

 %{Primary 05C30 05C10; Secondary 05A19 57M15.}
%Ref: 05A19 05C10 05C22 05C30 57M15
%arXiv 2501.00808 / 2503.20661 : Primary 32G15 58E11; Secondary 57M15 30F30

%graph theory

\date{%Received: xxxxxx; Revised: yyyyyy; Accepted: zzzzzz.
%\newline \indent $^{*}$ Corresponding author
}

\begin{abstract} 
This manuscript studies a special case of the Hurwitz enumeration problem: 
for branched covers from genus $g$ compact Riemann surface  to the Riemann sphere, with three branch points, and require the branching data at one of the branch points to be of the partition $(m\ 1^{n-m})$, obtain a formula of Hurwitz number. 
The Hurwitz enumeration problem can be transformed into enumeration of a class of unrooted hypermaps. 
We first provide a enumeration formula for rooted hypermaps, thereby obtaining the weighted Hurwitz numbers. 
Next give the quantitative relationship between the enumeration of unrooted hypermaps and that of rooted hypermaps in the general setting. 
Finally, combining these two results, we obtain a formula for the unrooted hypermaps or the Hurwitz numbers in the special setting. 
\end{abstract}

\maketitle

\tableofcontents

%%%%% section %%%%%
\section{Introduction}

Hurwitz number counts the number of branched coverings with a given branching data.
In the theory of dessins d'enfants, there is a one-to-one correspondence between branched coverings and connected bi-colored graphs embedded on surfaces, called \textbf{hypermaps}.
This allows us to study Hurwitz numbers via the enumeration of hypermaps.

%%%%% subsection %%%%%
\subsection{Branched Coverings and Branching Data}

Let $X,Y$ be two compact Riemann surfaces.
A non-constant holomorphic mapping $f : X \to Y$ is called a \textbf{branched covering}.
If near a point $x \in X$, the map $f$ can be expressed as $z \mapsto z^k$, then $f$ is said to have \textbf{multiplicity} $m$ at $x$.
All points in $X$ with multiplicity greater than $1$ are called \textbf{branch points}, and their images in $Y$ are called \textbf{branch values}.

Now we restrict to the case where $Y = \CP^1$ is the Riemann sphere and there are only three branch values. 
We may assume that the three branch values are $0, 1, \infty$ without loss of generality. 
We describe the \textbf{branching data} of a branched covering $f : X \to \CP^1$ via the following definition.

\begin{definition}
    Let $f : S_g \to \CP^1$ be a branched covering from a compact Riemann surface $S_g$ of genus $g$ to the Riemann sphere $\CP^1$, with three branch values $0, 1, \infty$.
    The \textbf{branching data} $(g, n; \pi_1, \pi_2, \pi_3)$ of $f$ is defined as
    \begin{itemize}
        \item the genus $g$ of the surface;
        \item the degree $n = \deg f$ of the map;
        \item the multiset of multiplicities at all points in the preimage $f^{-1}(0)$ of the branch value $0$, which forms a partition $\pi_1$ of $n$; similarly, the multiplicities at the preimages of $1$ and $\infty$ yield the partitions $\pi_2$ and $\pi_3$, respectively.
    \end{itemize}
\end{definition}

\bigskip

The Riemann-Hurwitz formula tells us (where $|\pi|$ denotes the number of parts in the partition)
\begin{equation}
   |\pi_1| + |\pi_2| + |\pi_3| - n = 2 - 2g .
\end{equation}

Conversely, the Hurwitz existence problem \cite{Hur91} asks:
\begin{center}
\textbf{Given a branching data satisfying the Riemann-Hurwitz formula, \\
does it arise as the branching data of some branched covering?}
\end{center}

%\GREEN{should cite some content}

We are actually more interested in the enumerative version of the Hurwitz problem:
\begin{center}
\textbf{Given a branching data, \\ 
how many branched coverings have the same branching data?}
\end{center}

To count the number of branched coverings, it is necessary to define isomorphisms between branching data.
Because for any automorphism $I$ on the Riemann surface $S_g$, the branched coverings $f$ and $f \circ I$ have the same branching data.

\begin{definition}
    Let $f_1, f_2 : S_g \to \CP^1$ be branched coverings with three branch values $0, 1, \infty$. If there exists a biholomorphic map $I : S_g \to S_g$ such that $f_2 \circ I = f_1$, then the branched coverings $f_1$ and $f_2$ are said to be \textbf{isomorphic}.
\end{definition}

\bigskip

The isomorphism in this definition is also called ``strong equivalence''. In the Hurwitz enumeration problem mentioned above, the phrase ``the number of branched coverings'' actually asks for the number of isomorphism classes of branched coverings corresponding to a given branching data.
This number is called the \textbf{Hurwitz number}.

Moreover, \textbf{weighted Hurwitz numbers} are easier to compute.
The definition above already gives the notion of isomorphism of branched coverings, so one can naturally define the automorphism group $\Aut(f)$ of a branched covering $f$.
The weighted Hurwitz number is defined as $\sum_{f} \frac{1}{|\Aut(f)|}$, with weight $1 / {|\Aut(f)|}$ (the reciprocal of the order of the automorphism group), where $f$ runs over all non-isomorphic branched coverings. 

In the next paragraph, the term Hurwitz number mostly refers to the weighted Hurwitz number.

In the Hurwitz enumeration problem, Mednykh \cite{Med84}, Hao Zheng \cite{Zhe06} and others used representation-theoretic methods to obtain complicated and not easily computable formulas for Hurwitz numbers. 
To find concise formulas, most researchers either impose additional conditions on the branching data or only require constraints on the genus $g$, the degree $n$, and the number of preimages $|\pi_i|$, $i=1,2,3$. 
Zograf \cite{Zog15} used generating function methods to compute Hurwitz numbers for arbitrary genus $g$ and degree $n$. Prior to this, the special cases of genus $g = 0$ and $1$ had been obtained by Walsh \cite{Wal75} and Arqu\`es \cite{Arq87}. 
In the case $\pi_3 = (n)$, Harary and Tutte \cite{HarTut64} gave Hurwitz numbers for genus $g=0$, and Goupil and Schaeffer \cite{GpSch98} gave them for arbitrary genus. 
Chapuy et al.\cite{CFF13} provided a combinatorial proof of the Goupil--Schaeffer result. 
Goulden-Jackson \cite{GlJac92} and Poulalhon-Schaeffer \cite{PouSch02} are respectively extensions of these results to the case of more than three branch points. 
In the case $\pi_3 = (m\, 1^{n-m})$, Kochetkov \cite{YYK15} gave Hurwitz numbers for genus $g=0$. 
Lu and Song \cite{LuSong26} gave a combinatorial proof of Kochetkov's result.

\bigskip

In this paper, we study Hurwitz numbers for a given set of branching data where one partition, $\pi_3 = (m\, 1^{n-m})$, has only one part not equal to $1$. 
We first give a formula for the weighted Hurwitz numbers in this case. 
Then, by studying how to convert weighted Hurwitz numbers to unweighted Hurwitz numbers in the general setting, we finally obtain a formula for the Hurwitz numbers in this case.

%%%%% subsection %%%%%
\subsection{Hurwitz Numbers and Hypermap Enumeration}

The branched coverings mentioned above have a more intuitive model. 
Namely, we can represent a branched covering using graph embedded on a surface, called hypermap. 
Such a correspondence is called the theory of ``dessins d'enfants''. 
Moreover, the genus of the surface, the number of edges, the weight distribution of the hypermap — which we call the \textbf{passport} — correspond exactly to the branching data. 
Thus, the number of branched coverings with given branching data i.e. the Hurwitz number, is transformed into the number of hypermaps with given passport.

We first give the precise definitions of hypermaps, hypermap isomorphisms and passports, etc., and then present the one-to-one correspondence between branched coverings and graphs.

\begin{definition}
    An embedding $i : G \to S_g$ of an abstract graph $G = (V, E)$ into a closed oriented surface $S_g$ of genus $g$ is called a \textbf{map}. The embedding satisfies the following conditions:
    \begin{itemize}
        \item Vertices $v \in V$ are mapped to points on the surface;
        \item Edges $e = \{v_1, v_2\} \in E$ are mapped to curves on the surface connecting $i(v_1)$ and $i(v_2)$. These curves are pairwise disjoint except at their endpoints;
        \item Each connected component of $S_g \setminus i(G)$ is homeomorphic to an open disk; these components are called the \textbf{faces} of the map, and they form the set $F$.
    \end{itemize}
    In addition, if the abstract graph is bicolored, i.e., the vertex set $V = U_1 \sqcup U_2$ can be partitioned into black vertices $U_1$ and white vertices $U_2$ such that every edge connects a black vertex to a white vertex, then the bicolored map is called \textbf{hypermap}. 
    
    We also denote the set of faces by $U_3 := F$, so that a hypermap is denoted by $M = (S_g; E, U_k)$.
\end{definition}

\bigskip

Two hypermaps $M_i = (S_g; E_i, U_{i, k}), i=1,2$ are called isomorphic if there exists an orientation-preserving homeomorphism $I : S_g \to S_g$ of the underlying surface, inducing a bijection of graph elements such that $I(E_1) = E_2$ and $I(U_{1, k}) = U_{2, k}$.

For a hypermap $M = (S_g; E, U_k)$, we define weights $\wt : U_1 \sqcup U_2 \sqcup U_3 \to \Zpos$ on black vertices, white vertices and faces, respectively.
The weight at a vertex is the number of incident edges. The weight at a face is defined as half the number of sides of its topological polygon.
If the total number of edges is $n$, we have
\begin{equation*}
    \sum_{u_k \in U_k} \wt(u_k) \equiv n \ ,\ 
    |U_1| + |U_2| + |U_3| - n = 2 - 2g .
\end{equation*}

Recording the weight distributions on black vertices, white vertices and faces of a hypermap as partitions $\pi_k$ of $n$ ($k=1,2,3$), together with the genus $g$ and the number $n$ of edges, we obtain a \textbf{passport} $\pp = (g, n; \pi_1, \pi_2, \pi_3)$.
We denote by $\Mc(\pp)$ the set of all non-isomorphic labelled hypermaps with passport $\pp$.

\begin{proposition}[\cite{LsZak04}, Section 1.5.1] \label{prop:dessin}
    The isomorphism classes of branched covers $f : S_g \to \CP^1$ with branching data $(g, n; \pi_1, \pi_2, \pi_3)$ are in one-to-one correspondence with isomorphism classes $M \in \Mc(g, n; \pi_1, \pi_2, \pi_3)$ of hypermaps with passport $(g, n; \pi_1, \pi_2, \pi_3)$.
    
    Concretely, for each branched cover $f$, take the preimages $f^{-1}(0)$ and $f^{-1}(1)$ as the black and white vertex sets, respectively. 
    And take the preimage $f^{-1}\big([0,1]\big)$ of the closed interval $[0,1]$ as the edges. 
    This yields the corresponding hypermap. 
    Moreover, each face of the hypermap corresponds to a preimage of $\infty$.
    \qed
\end{proposition}

\bigskip

By the dessins d'enfants theory, Hurwitz numbers count the number of hypermaps $\abs{\Mc(\pp)}$. 

\bigskip

For weighted Hurwitz numbers, we need to define the so-called rooted hypermaps.

\begin{definition}
    A \textbf{rooted hypermap} $(M, e)$ consists of a hypermap $M = (S_g; E, U_k; \Lab_k)$ and a root edge $e \in E$.
    Two rooted hypermaps $(M_i, e_i), i=1,2$ are isomorphic, if an orientation-preserving homeomorphism $I : S_g \to S_g$ gives an isomorphism between $M_1$ and $M_2$, and the root edge is preserved $I(e_1) = e_2$.
    
    The set of all non-isomorphic rooted hypermaps with passport $\pp$ is denoted by $\Mc_R(\pp)$.
    %For fixed genus $g$ and number of edges $n$, we also have the sets $\Mc_R(g, n)$.
\end{definition}

\bigskip

%Analogously, we can define sets of rooted hypermaps $\Mc_R(g, n), \QUc_R(g, m, n), \Uc_{R, W}(g, m, n)$.

%\bigskip

An important property of rooted hypermaps is that their automorphism group is trivial: $\Aut(T, e) \cong \{\id\}$.
Hence, all vertices, edges and faces of a rooted hypermap are distinguishable.

A simple formula in hypermap enumeration gives
\begin{equation} \label{eq:rooted_and_nonrooted}
    \frac1n \abs{\Mc_R(\pp)} 
    = \sum_{M \in \Mc(\pp)} \frac1{\abs{\Aut(M)}} \ . 
\end{equation}
%The following proposition tells us that,
Together with the dessins d'enfants theory in Proposition \ref{prop:dessin}, we obtain directly that the weighted Hurwitz number equals $\frac1n \abs{\Mc_R(\pp)}$.
In other words, the enumeration of rooted hypermaps gives the weighted Hurwitz numbers.

\bigskip
%%%%% subsection %%%%%
\subsection{Labeled Hypermaps and Passports}

For convenience, we assign distinct positive integers as labels to all vertices and faces of an hypermap.  This is \textbf{totally labeled hypermap} $M = (S_g; E, U_k; \Lab_k)$.  
Here the bijective labeling function $\Lab_k : U_k \to [\abs{U_k}]$ maps the vertex or face set $U_k$ to the set of positive integers $[\abs{U_k}] := \{1, \ldots, \abs{U_k}\}$.  
In this case, each weight distribution is represented by a sequence $\bmwt_k = \big( \wt(\Lab_k^{-1}(i)) \big)_{i=1}^{\abs{U_k}}$, recording the weight of the $i$-th vertex or face.

Two totally labeled hypermaps $M_i = (S_g; E, U_{i, k}; \Lab_{i, k}), i=1,2$ are isomorphic if there exists an orientation-preserving homeomorphism $I : S_g \to S_g$ that gives an isomorphism of the unlabeled hypermaps $(S_g; E, U_{i, k})$, and pairs vertices and faces with the same labels $\Lab_{2, k} \circ\ I = \Lab_{1, k}$.

\bigskip

Below we use a broader definition (similar to that in \cite{GorTor1980}) which distinguishes vertices or faces of the same weight only when their labels $\Lab_k(u) \in S_k$ are different; they are not distinguished if the labels are the same.  
Here the index $s \in S_k$ corresponds to vertices of weight $\wt_k(s)$, with total number $\lambda_k(s)$.

\bigskip

\begin{definition}\label{def:passport}
    A \textbf{passport} $\pp = (g, n; \Pi_1, \Pi_2, \Pi_3)$ consists of the genus $g$, the number of edges $n$, and three \textbf{weight distributions} $\Pi_k$ for black vertices, white vertices and faces, respectively.  
    Each weight distribution $\Pi_k = (S_k, \lambda_k, \wt_k)$ consists of a nonempty \textbf{index set} $S_k$, a \textbf{multiplicity function} $\lambda_k : S_k \to \Zpos$, and a \textbf{weight function} $\wt_k : S_k \to \Zpos$;  
    they satisfy the compatibility condition that the sum of weights equals the number of edges and Euler's formula:
    \begin{equation}
        \sum_{s \in S_k} \lambda_k(s) \wt_k(s) = n \ ,\ k = 1,2,3  \quad ,\quad \sum_{k=1}^3 \sum_{s \in S_k} \lambda_k(s) - n = 2 - 2g \ .
    \end{equation}

    The numbers of black vertices, white vertices, faces, and total vertices of the passport are
    \begin{equation}\begin{split}
        u_k(\pp) := \sum_{s \in S_k} \lambda_k(s) \ , \ k=1,2,3 \quad ,\quad 
        v(\pp) &:= u_1(\pp) + u_2(\pp) \ .
    \end{split}\end{equation}
    The Euler's formula can b denoted by $v(\pp) + u_3(\pp) - n = 2 - 2g$.
\end{definition}

\bigskip

\begin{definition} \label{def:labeled_graph}
    A \textbf{labeled hypermap} with passport $\pp$ consists of an hypermap $M = (S_g; E, U_k)$ on a closed oriented surface $S_g$ of genus $g$ and \textbf{labeling functions} $\Lab_k : U_k \to S_k, k=1,2,3$ on vertices and faces, satisfying
    \begin{itemize}
        \item Compatibility between the weight functions  $\wt = \wt_k \circ \Lab_k \ ,\ k=1,2,3$;
        \item The number of vertices with $s$ label is given by the multiplicity function $\abs{\Lab_k^{-1}(s)} = \lambda_k(s), \forall\ s \in S_k \ ,\  k = 1,2,3$.
    \end{itemize}
\end{definition}

\bigskip

Thus the numbers of black vertices, white vertices and faces are exactly the sums of the corresponding multiplicity functions
$$
\abs{U_k} = u_k(\pp) \ ,\ k = 1,2,3 \ .
$$

In this thesis, a labeled hypermap is still denoted $M = (S_g; E, U_k; \Lab_k)$, and all \textbf{maps} mentioned subsequently are labeled hypermaps. 

Isomorphism of labeled hypermaps is defined analogously to isomorphism of totally labeled ones. 
For a map $M$, the set of all isomorphisms from $M$ to itself forms the automorphism group $\Aut(M)$.  
We denote by $\Mc(\pp)$ the set of all non-isomorphic maps with passport $\pp$.

\bigskip

It is necessary to introduce \textbf{notation for passports} to further clarify the meaning of assigning different index labels to vertices of the same weight.

\begin{notation} \label{notat:passport}
    For a weight distribution $\Pi = (S, \lambda, \wt)$, for each admissible weight $w\in \wt(S)$, the indices $\wt^{-1}(w) \subset S$ corresponding to weight $w$ are distinguished by different positive integer subscripts, i.e., $\wt^{-1}(w) = \{ w_1, w_2, \cdots, w_a \}$, where $a=\lvert \wt^{-1}(w) \rvert$.  
    Then we use \textbf{power notation} to represent the entire weight distribution:
    \[ \Pi = \prod_{s\in S} s^{\lambda(s)} \ . \]  
    Finally, a passport is still denoted $\pp = (g, n; \Pi_1, \Pi_2, \Pi_3)$.
\end{notation}

\bigskip

Note that in an unlabeled hypermap, vertices or faces with the same weight are not distinguished; then $\wt_k$ is injective, and $\Pi_k$ can be completely described by a partition $\pi_k = \prod_{s \in S} \wt_k(s)^{\lambda_k(s)}$.

In a totally labeled hypermap, vertices or faces with the same weight are all distinguished because their labels are distinct positive integers.  
In this case $\lambda_k \equiv 1$. Using $S_k = [u_k(\pp)]$, then sequences $\bmwt_k = (\wt_k(i))_{i=1}^{u_k(\pp)}$ completely determine $\Pi_k$.

This shows that our definition of a passport generalizes those of unlabeled and totally labeled hypermaps.

\bigskip

Below we give some examples of labeled hypermaps, their isomorphisms, and their passports.

\begin{example}
Figure \ref{fig:example_map} shows some maps on the sphere.

The upper left map has two black vertices of weight 3, two white vertices of weight 3, and two faces of weight 3, corresponding to the unlabeled passport $(0, 6; 2^3, 2^3, 3^2)$.

The upper right maps have the same vertex weights as the upper left map, but the black vertices are labeled $2_1, 2_2, 2_3$ respectively, and each label corresponds to exactly one vertex.
According to the notation in \ref{notat:passport}, the corresponding passport is $(0, 6; 2_1\ 2_2\ 2_3 \ ,\ 2^3 \ ,\ 3^2)$.

Rotating the upper right map by $2\pi/3$ yields the lower left map, so they are isomorphic.
However, the upper right map requires a horizontal flip to obtain the lower right map; this operation is not orientation-preserving and therefore does not constitute an isomorphism.
In fact, the upper right map and the lower right map are non-isomorphic.
\end{example}

\begin{figure}[ht]
\centering
\begin{tikzpicture}[scale=0.9]

    %%%%% map1 %%%%%
    \begin{scope}
    
    % 绘制一般的黑点白点和连线
    \node[blk] (B1) at (0, 0) {\large $2$};
    \node[wht] (W1) at (2, 0) {\large $2$};
    \node[blk] (B2) at (3, 1.73) {\large $2$};
    \node[wht] (W2) at (2, 3.46) {\large $2$};
    \node[blk] (B3) at (0, 3.46) {\large $2$};
    \node[wht] (W3) at (-1, 1.73) {\large $2$};
    
    \draw (B1) edge[ultra thick] (W1);
    \draw (B1) edge[ultra thick] (W3);
    \draw (B2) edge[ultra thick] (W1);
    \draw (B2) edge[ultra thick] (W2);
    \draw (B3) edge[ultra thick] (W2);
    \draw (B3) edge[ultra thick] (W3);
    
    \end{scope}

    %%%%% map2 %%%%%
    \begin{scope}[shift={(6, 0)}]
    % 绘制一般的黑点白点和连线
    \node[blk] (B1) at (0, 0) {$2_1$};
    \node[wht] (W1) at (2, 0) {\large $2$};
    \node[blk] (B2) at (3, 1.73) {$2_2$};
    \node[wht] (W2) at (2, 3.46) {\large $2$};
    \node[blk] (B3) at (0, 3.46) {$2_3$};
    \node[wht] (W3) at (-1, 1.73) {\large $2$};
    
    \draw (B1) edge[ultra thick] (W1);
    \draw (B1) edge[ultra thick] (W3);
    \draw (B2) edge[ultra thick] (W1);
    \draw (B2) edge[ultra thick] (W2);
    \draw (B3) edge[ultra thick] (W2);
    \draw (B3) edge[ultra thick] (W3);
    \end{scope}

    %%%%% map3 %%%%%
    \begin{scope}[shift={(0, -5)}]
    % 绘制一般的黑点白点和连线
    \node[blk] (B1) at (0, 0) {$2_3$};
    \node[wht] (W1) at (2, 0) {\large $2$};
    \node[blk] (B2) at (3, 1.73) {$2_1$};
    \node[wht] (W2) at (2, 3.46) {\large $2$};
    \node[blk] (B3) at (0, 3.46) {$2_2$};
    \node[wht] (W3) at (-1, 1.73) {\large $2$};
    
    \draw (B1) edge[ultra thick] (W1);
    \draw (B1) edge[ultra thick] (W3);
    \draw (B2) edge[ultra thick] (W1);
    \draw (B2) edge[ultra thick] (W2);
    \draw (B3) edge[ultra thick] (W2);
    \draw (B3) edge[ultra thick] (W3);
    \end{scope}

    %%%%% map4 %%%%%
    \begin{scope}[shift={(6, -5)}]
    % 绘制一般的黑点白点和连线
    \node[blk] (B1) at (0, 0) {$2_3$};
    \node[wht] (W1) at (2, 0) {\large $2$};
    \node[blk] (B2) at (3, 1.73) {$2_2$};
    \node[wht] (W2) at (2, 3.46) {\large $2$};
    \node[blk] (B3) at (0, 3.46) {$2_1$};
    \node[wht] (W3) at (-1, 1.73) {\large $2$};
    
    \draw (B1) edge[ultra thick] (W1);
    \draw (B1) edge[ultra thick] (W3);
    \draw (B2) edge[ultra thick] (W1);
    \draw (B2) edge[ultra thick] (W2);
    \draw (B3) edge[ultra thick] (W2);
    \draw (B3) edge[ultra thick] (W3);
    \end{scope}
\end{tikzpicture}

\caption{example} \label{fig:example_map}

\end{figure}

\begin{remark}
    Passports with identical notation are called \textbf{equivalent}.
    Equivalent passports $\pp_1, \pp_2$ correspond to maps with a canonical bijection between them. Thus we do not distinguish equivalent passports.
    Passport equivalence is equivalent to replacing the index set $S$ by a bijection, so later we may assume that the index set has a given form.
\end{remark}

\bigskip
%%%%% subsection %%%%%
\subsection{Main Results}

We first recall our problem. 
We wish to give the Hurwitz number when one of the three partitions is $\pi_3 =(m\ 1^{n-m})$.
Recall that Hurwitz number is exactly the number of maps $\abs{\Mc(\pp)}$ with passport $\pp = (g, n; \pi_1, \pi_2, \pi_3 = (m\ 1^{n-m}))$. 
The weighted Hurwitz number is $\frac1n \abs{\Mc_R(\pp)}$, which is related to the number of rooted maps.

For a not necessarily unlabeled passport, we still require that one weight distribution has the form $\pi_3 = (m\ 1^{n-m})$. Such a passport is denoted by
$$
\ppqu = (g, m, n; \Pi_1, \Pi_2) := (g, n; \Pi_1, \Pi_2, \pi_3 = (m\ 1^{n-m})) \ .
$$
We call a passport of the above form a \textbf{quasi-one-face passport}, and the corresponding maps are called \textbf{quasi-one-face maps}.
In particular, if $m = n$ i.e. $\pi_3 = (n)$, the corresponding passport and maps are called \textbf{one-face passport} and \textbf{one-face maps}(also called one-face maps).
In this case, a totally labeled quasi-one-face passport only requires that $\Pi_1, \Pi_2$ are totally labeled.

\bigskip

Now we proceed to give the Hurwitz number for $\pi_3 = (m\ 1^{n-m})$, i.e., the number of quasi-one-face maps.

First, we compute the weighted Hurwitz number, or the number of rooted quasi-one-face maps for a given quasi-one-face passport $\ppqu$.

Chapuy et al.\cite{CFF13} pointed out that one-face maps of positive genus can be reduced to genus zero ones by vertex splitting. 
Inspired by this, we reduce the enumeration of rooted quasi-one-face maps of positive genus to the enumeration of those of genus zero in the following theorem. 
Here vertex splitting is reflected in the passport $\ppqu$ decomposed into $\ppzl$ (via the condition $\ppzl \xrightarrow{\Lambda} \Fill(\ppqu)$ described below).
Kochetkov \cite{YYK15} gave an enumeration formula for the genus zero case, so the theorem is ultimately computable. 
We quote that formula in Theorem \ref{thm:YYK_formula}.

\begin{theorem} \label{thm:CFF_general_passport_formula}
    Given a quasi-one-face passport $\ppqu = (g, m, n; \Pi_1, \Pi_2)$, the number of rooted quasi-one-face maps is
    \begin{equation} \label{eq:QUR_ppqu_enum}
        \abs{\Mc_R(\ppqu)} = \frac{n}{\ 2^{2g}\ (\ppqu)!}  \sum_{\Lambda \vdash g} R(\Lambda)^{-1} \sum_{\ppzl \xrightarrow{\Lambda} \Fill(\ppqu)} \abs{\Mc(\ppzl)} \ .
    \end{equation}
\end{theorem}

This involves the following notions:

\textbf{1.} The \textbf{factorial} of a quasi-one-face passport is
\begin{equation}
    (\ppqu)! := \prod_{s \in S_1} \lambda_1(s) \cdot \prod_{s \in S_2} \lambda_2(s) \ .
\end{equation}

\textbf{2.} A \textbf{cycle datum} $\Lambda = (\bmmu_1, \bmmu_2), \bmmu_k = (\mu_{k, j})_{j = 1}^{u_k(\pp)}, k=1,2$ consists of two sequences of nonegetive integers.
The condition $\Lambda \vdash g$ requires $\norm{\bmmu_1} + \norm{\bmmu_2} = g$, where the norm of a sequence is defined as $\norm{\bmmu_k} := \sum_{j=1}^{u_k(\pp)} \mu_{k, j}$.
Its weight $R(\Lambda)$ is defined as
\begin{equation} \label{eq:def_R_Lambda}
    R(\Lambda) := \prod_{j=1}^{u_1(\pp)} (2 \mu_{1, j}+1) \prod_{j=1}^{u_2(\pp)} (2 \mu_{2, j}+1) \ ,
\end{equation}

\textbf{3.} The \textbf{filling} $\Fill(\ppqu) = (g, m, n; \bmwt_1, \bmwt_2)$ of a quasi-one-face passport has the same genus $g$, special face edge count $m$, and total edge count $n$ as the original passport, but replaces each weight distribution $\Pi_k$ with the totally labeled weight distribution represented by the sequence $\bmwt_k$ consisting of $\lambda_k(s)$ copies of the positive integer $\wt_k(s)$.

\textbf{4.} In the condition $\ppzl \xrightarrow{\Lambda} \Fill(\ppqu)$, write
$\Fill(\ppqu) = (g, m, n; \bmwt_1, \bmwt_2)$, and regard each $\bmwt_k$ as a column vector.
A genus zero totally labeled quasi-one-face passport $\ppzl = (0, m, n; \Pi_1 = (S_{T, 1}, \one, \wt_{T, 1}), \Pi_2 = (S_{T, 2}, \one, \wt_{T, 2}))$ has index sets determined by $\Lambda$ of the form
\begin{equation} \label{eq:index_set_T}
    S_{T, k} = \defsetsmall{(i,j)}{1 \leq i \leq u_k(\ppqu), 1 \leq j \leq 2\mu_{k, i}+1} \ ,\ k=1,2 \ .
\end{equation}
Thus the weight distribution $\Pi_k$ can be written as a two-dimensional array $\bmwt_{T, k} = (\wt_{T, k}(i,j))_{i,j}$, where the length of the $i$-th row is $(2\mu_{k, i} + 1)$ (which may vary with $i$).
The condition $\ppzl \xrightarrow{\Lambda} \Fill(\ppqu)$ requires that the sum of each row of the array $\bmwt_{T, k}$ equals the corresponding entry of the column vector $\bmwt_k$, i.e.,
\begin{equation} \label{eq:cond_passport_row_sum}
    \wt_k(i) := \sum_{j=1}^{2\mu_{k,i}+1} \wt_{T, k}(i, j) \ .
\end{equation}

The sum is taken over all \textbf{distinct} passports $\ppzl$ satisfying $\ppzl \xrightarrow{\Lambda} \Fill(\ppqu)$.
If the components $\wt_{T, k}(i,j)$ differ, the corresponding passports are considered different. 

For example, when $\Fill(\ppqu) = (1, 7, 8; (4, 4), (2, 2, 2, 2)), \Lambda = ((1, 0), (0, 0, 0, 0))$, there are three passports $\ppzl$ satisfying the condition
\begin{equation*}
    \pp_1 = (0, 7, 8; 
    \begin{bmatrix}
        2 & 1 & 1 \\
        4 &  & 
    \end{bmatrix}
    ,
    \begin{bmatrix}
        2 \\
        2 \\
        2 \\
        2 \\
    \end{bmatrix}
    ) \ ,\ 
    \pp_2 = (\sim; 
    \begin{bmatrix}
        1 & 2 & 1 \\
        4 &  & 
    \end{bmatrix}
    ,
    \begin{bmatrix}
        2 \\
        2 \\
        2 \\
        2 \\
    \end{bmatrix}
    ) \ ,\ 
    \pp_3 = (\sim; 
    \begin{bmatrix}
        1 & 1 & 2 \\
        4 &  & 
    \end{bmatrix}
    ,
    \begin{bmatrix}
        2 \\
        2 \\
        2 \\
        2 \\
    \end{bmatrix}
    ) \ .
\end{equation*}

\textbf{5.} The total number $\abs{\Mc(\ppzl)}$ of maps corresponding to a genus zero totally labeled quasi-one-face passport $\ppzl$ can be obtained by Kochetkov's formula \ref{thm:YYK_formula}.

\bigskip

As an application of this enumeration theorem, we can also give a simple and direct new proof of the existence problem for Hurwitz numbers with positive genus $g > 0$ and $\pi_3 = (m\ 1^{n-m})$.
Boccara \cite{Boc82} first gave a proof of this existence theorem.
Song et al. \cite{SXY22axv} gave another more constructive proof.

\begin{corollary} \label{cor:exist_quasi-unicellular_map}
    Given a quasi-one-face passport $\ppqu = (g, m, n; \Pi_1, \Pi_2)$.
    If the genus is positive $g > 0$,
    then $\Mc(\ppqu) \neq \varnothing$.
\end{corollary}

\bigskip

Finally, we obtain the unweighted Hurwitz number for $\pi_3 = (m\ 1^{n-m})$.
That is, the number of corresponding unrooted quasi-one-face maps.
The following theorem starts from the enumeration of rooted quasi-one-face maps, and finally obtains the enumeration of unrooted ones.
Note that the enumeration of rooted quasi-one-face maps is already given by Theorem \ref{thm:CFF_general_passport_formula}, so this theorem is also computable.

Formula \eqref{eq:rooted_and_nonrooted} indicates that the number of rooted maps is related not only to the total count of unrooted maps but also to the number of unrooted maps with restricted symmetries.
Conversely, the number of unrooted maps will also be related to the number of rooted maps with symmetry restrictions.
Our method for restricting symmetries is to consider ``quotients'' of maps: quotient the surface $S_g$ by automorphism group of map, obtaining a quotient map on the quotient surface $O(\sigma)$.
This yields the quotient of the passport $\ppqu$, obtaining a passport $\ppsi$ satisfying the condition $\ppqu = \sigma \times_l \ppsi$ described later.

\begin{theorem} \label{thm:unrooted_QU} 
    Given a quasi-one-face passport $\ppqu = (g, m, n; \Pi_1, \Pi_2)$, the number of unrooted quasi-one-face maps is
    \begin{equation}
    \abs{\Mc(\ppqu)} = \frac1n \sum_{l \, |\, n} \sum_{g = \sigma \times l} \sum_{\ppqu = \sigma \times_l \ppsi} 
    \abs{\Mc_R(\ppsi)} \cdot \abs{\Epi_0 (\Gamma(\sigma), \ZZ_l)} \ .
\end{equation}
\end{theorem}

The following content is involved:

\textbf{1.} An \textbf{orbifold symbol} $\sigma = (g_2; \{t_1, \ldots, t_r\})$ consists of the genus $g$ and an multiset of positive integers $\bm{t} = \{t_i\}_{i=1}^r$, where each $t_i$ is called a \textbf{cone point order}.
When $t_i = 1$, the cone point order is degenerate.
Thus removing or adding any number of $1$'s to the orbifold symbol does not change it.
The unique \textbf{orbifold} corresponding to the orbifold symbol $\sigma$ is denoted by $O(\sigma)$.

The condition $g = \sigma \times l$ requires that the orbifold $O(\sigma) = S_g / \ZZ_l$ is the quotient space of a genus $g$ surface by a cyclic group. 
A description of this condition can be found in Theorem \ref{thm:g_l_sigma_admissible}.

\textbf{2.} A \textbf{passport labelled by the orbifold symbol} $\sigma = (g_2; \bm{t})$ is denoted by $\ppsi = (g_2, n; \Pi_1, \Pi_2, \Pi_3)$.
Here the genus $g_2$ of the orbifold symbol is required to be the same as the genus $g_2$ of the passport.
Moreover, the weight distributions $\Pi_i = (S_i, \lambda_i, \wt_i), i=1,2,3$ satisfy the following conditions:
\begin{itemize}
    \item The unordered sequence of cone point orders $\bmt$ should be partitioned into the following form, where $\{t_{ij}^{\lambda_{ij}}\}$ denotes $\lambda_{ij}$ copies of $t_{ij}$
    \begin{equation} \label{eq:sigma_split}
        \bmt = \bigsqcup_{i=1}^3 \bigsqcup_{j=1}^{N_i} \left\{ \ t_{ij}^{\lambda_{ij}} \ \right\} \ ,
    \end{equation}
   This partition can be achieved by adjusting the number of $1$'s.

    \item The weight distribution $\Pi_i$ can be written in the following form
    \begin{equation} \label{eq:sigma_labeled_passport}
        S_i = \defsetsmall{(w_{ij}, t_{ij})}{i=1,2,3\ , \ 1 \leq j \leq T_i} \ ,\ 
        \lambda_i(w_{ij}, t_{ij}) = \lambda_{ij} \ ,\ 
        \wt_i(w_{ij}, t_{ij}) = w_{ij} \ .
    \end{equation}
    That is, each index element can be written as $(w_{ij}, t_{ij})$ and they are pairwise distinct, where the first component $w_{ij}$ is the weight corresponding to the element, and the second component $t_{ij}$ takes each cone point order of the orbifold symbol exactly once.
\end{itemize}

The passport $\ppsi$ will correspond to maps on orbifolds, with $\lambda_{ij}$ vertices or faces of weight $w_{ij}$ and cone point order $t_{ij}$.
This correspondence can be seen in Definition \ref{def:orbifold_map_passport}.

\textbf{3.} \textbf{Multiplication of an orbifold symbol and a passport} yields an unlabelled passport $\pp = \sigma \times_l \ppsi := (g, n; \pi_1, \pi_2, \pi_3)$, where the genus $g = \sigma \times l$ is the genus of the original surface, the number of edges $n := l \cdot m$, and the weight distributions are defined as follows
\begin{equation} \label{eq:sigma_pp_multi}
    \pi_i := \bigsqcup_{j=1}^{N_i} \left\{ (w_{ij} \cdot t_{ij})^{\lambda_{ij} \cdot l\, /\, t_{ij}} \right\} \quad,\quad i=1,2,3 \ .
\end{equation}
That is, after taking the quotient by the cyclic group $\ZZ_l$, vertices with quantities $\lambda_{ij}$,  weights $w_{ij}$ and cone point orders $t_{ij}$ in the quotient map correspond to vertices with quantitative $(\lambda_{ij} \cdot l\  /\  t_{ij})$ and weight $(w_{ij} \cdot t_{ij})$ in the original map.

\textbf{4.} By Lemma \ref{lem:QU_divided_also_QU}, since $\ppqu$ is a quasi-one-face passport, $\ppsi$ is also a quasi-one-face passport, hence $\abs{\Mc_R(\ppsi)}$ can be obtained directly from Theorem \ref{thm:CFF_general_passport_formula}. 
The number $\abs{\Epi_0(\Gamma(\sigma), \ZZ_l)}$ can be obtained from Proposition \ref{prop:Epi_count}.

\bigskip

Examples of the above Theorem \ref{thm:CFF_general_passport_formula} and \ref{thm:unrooted_QU} can be found in Appendix \ref{ch:example}.

\bigskip

In fact, our method can obtain the number of unrooted maps from the number of rooted maps for general passports.
This is equivalent to obtaining unweighted Hurwitz numbers from weighted Hurwitz numbers.
However, it should be noted that weighted Hurwitz numbers for general passports are still not easy to obtain.

This method is inspired by Mednykh-Nedela \cite{MedNed06, MedNed10}.
They only restrict the genus $g$ and the number of edges $n$, while this paper further restricts the passport.
This paper mainly follows the ideas of \cite{MedNed06}.

\begin{theorem} \label{thm:rooted_to_unrooted}
Given an unlabelled passport $\pp$. The number of unrooted maps corresponding to $\pp$ is related to the numbers of rooted maps for all ``divided'' passports $\ppsi$
\begin{equation}
    \abs{\Mc(\pp)} = 
    \frac1n \sum_{l \, |\, n} \sum_{g = \sigma \times l} \sum_{\pp = \sigma \times_l \ppsi} 
    \abs{\Mc_R(\ppsi)} \cdot \abs{\Epi_0 (\Gamma(\sigma), \ZZ_l)} \ .
\end{equation}
\end{theorem}

\begin{remark}
Here the unlabelled passport can be replaced by a general passport $\pp$, 
while the notions of ``passport labelled by the orbifold symbol'' and ``multiplication of an orbifold symbol and passport'' in $\pp = \sigma \times_l \ppsi$ need to be extended to the labelled case.
Then Theorem \ref{thm:rooted_to_unrooted} still holds.

Specifically, in the definition \eqref{eq:sigma_labeled_passport} of a passport labelled by the orbifold symbol, one can replace the weight-order pair $(w_{ij}, t_{ij})$ by an index-order pair $(s_{ij}, t_{ij})$, where different pairs $(s_{ij}, t_{ij})$ may correspond to the same weight and cone point order, but the labels are still distinct.

Moreover, the passport $\pp$ containing only three partitions can be generalized to one with arbitrary number of partitions. 
Here the map is replaced by constellation. 
Theorem \ref{thm:rooted_to_unrooted} still holds in this case, because we do not actually use the condition of ``three'' partitions here.
\end{remark}

\bigskip
%%%%% subsection %%%%%
\subsection{Structure of the Manuscript}

The following is the structure of this manuscript. The structure of each section can also be found at the beginning of each section.

We first prove Theorem \ref{thm:CFF_general_passport_formula}, i.e., the number of rooted quasi-one-face maps or weighted Hurwitz numbers.
The proof route can be seen on the left side of Figure \ref{fig:proof_line}.

In Section \ref{ch:map_and_passport}, 
Proposition \ref{prop:QU_cong_UW} shows that quasi-one-face maps can be reduced to weighted one-face maps. 
Proposition \ref{prop:Fill_rooted_eq_fact_rooted} shows that rooted maps corresponding to a general passport can be reduced to the totally labeled rooted maps.
Thus we can study only totally labeled rooted weighted one-face maps.

In Section \ref{ch:rooted_quasi-unicellular_map_enum}, we first define C-decorated trees.
By Theorem \ref{thm:CFF_ppl}, the number of totally labeled rooted weighted one-face maps is reduced to the number of totally (cycles-) labelled weighted C-decorated trees.
Then by Proposition \ref{prop:passport_tree_bij}, it is further reduced to totally labeled weighted trees.
Finally, by Theorem \ref{thm:YYK_formula}, the number of trees can be computed. 
This completely solves the problem of counting rooted quasi-one-face maps.
As an application of this theorem, we also prove the Hurwitz existence problem for positive genus with $\pi_3 = (m\ 1^{n-m})$ (Corollary \ref{cor:exist_quasi-unicellular_map}).

\bigskip

Then we prove Theorem \ref{thm:unrooted_QU} and Theorem \ref{thm:rooted_to_unrooted}, i.e., the number of unrooted maps or Hurwitz numbers.
The proof route can be seen on the right side of Figure \ref{fig:proof_line}.

Section \ref{ch:algebraic_map} introduces algebraic maps to represent maps on surfaces (Proposition \ref{prop:map_and_algebraic_map}), and transforms them into group action orbits (Lemma \ref{lem:group_action_on_numbered_map}), so that Burnside's lemma can be used to convert the number of unrooted maps into the number of symmetric maps (Lemma \ref{lem:unrooted_from_fix_enum}).
The symmetric maps will be studied via their quotient maps.

Section \ref{ch:map_cover_and_orbifold_map} studies how symmetric maps become quotient maps.
The symmetry gives a regular cover, the quotient map will be the map under the regular cover and lies on the quotient surface, i.e., the orbifold (Definition \ref{def:map_regular_cover}).
Thus it is necessary to study maps on orbifolds.
The cone point orders at the vertices or faces can be transformed into labels, so map on orbifold can actually be regarded as a labelled map (Definition \ref{def:orbifold_map_passport}).
Finally, we obtain the conditions that the quotient map of a symmetric map must satisfy (Lemma \ref{lem:symmetric_map_to_orbifold_map}).

Section \ref{ch:unrooted_map_enum} studies how to construct symmetric maps in reverse given a quotient map.
The specific method is to expand the quotient map to the universal cover map, and then restrict it to the fundamental domain of the original surface. 
It yields a symmetric map on the original surface (Lemma \ref{lem:orbifold_map_to_symmetric_map}).
Finally, Theorem \ref{thm:fix_enum} states that the correspondence between symmetric maps and quotient maps, thus proving that the number of unrooted maps can be computed from the number of rooted maps on orbifolds.

\bigskip

Appendix \ref{ch:labeled_weighted_tree_enum} gives related results for quasi-one-face maps on the sphere.
Appendix \ref{ch:epi_enum} gives related results for the number of order-preserving epimorphisms needed for unrooted map enumeration.
Appendix \ref{ch:example} gives some examples of the computation and application of Theorem \ref{thm:CFF_general_passport_formula} and Theorem \ref{thm:unrooted_QU}.

\vspace*{2em}

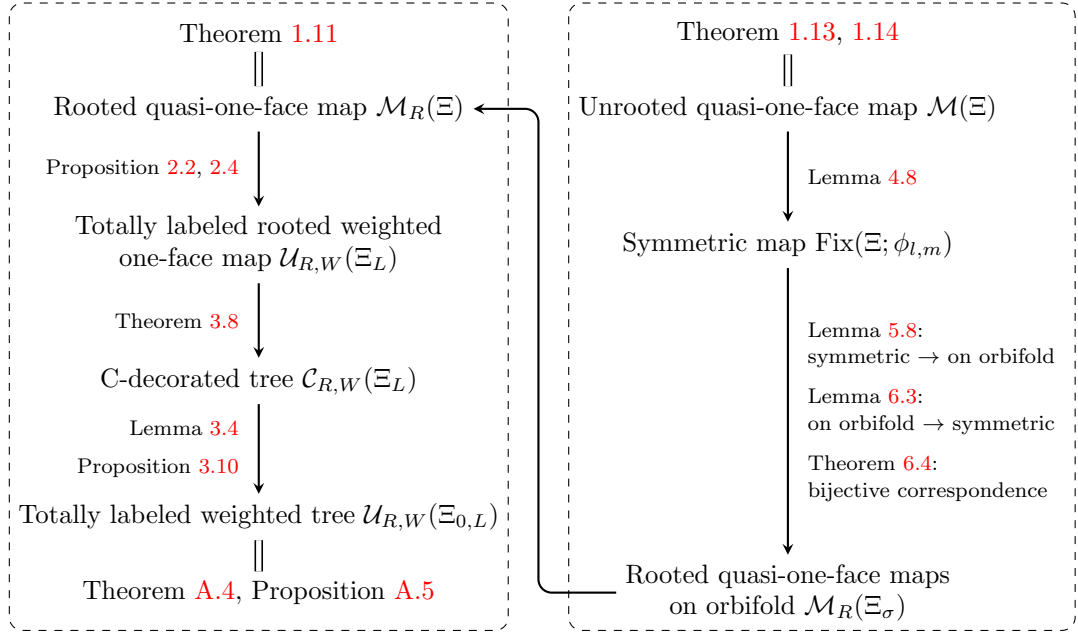
\begin{figure}[ht]
\centering
\begin{tikzpicture}
[
    arrow/.style={-stealth, thick}
]

\footnotesize{
\node (A1) at (0, 0) {Theorem \ref{thm:CFF_general_passport_formula}};
\node (B1) at (0, -1) {Rooted quasi-one-face map $\Mc_R(\pp)$};
\node (U1) at (0, -2.8) [align=center] {Totally labeled rooted weighted \\ one-face map $\Uc_{R, W}(\ppl)$};
\node (C1) at (0, -4.6) {C-decorated tree $\Cc_{R, W}(\ppl)$};
\node (D1) at (0, -6.4) {Totally labeled weighted tree $\Uc_{R, W}(\ppzl)$};
\node (E1) at (0, -7.4) {Theorem \ref{thm:YYK_formula}, Proposition \ref{prop:weighted_tree_no_isom}};

\node (A2) at (7, 0) {Theorem \ref{thm:unrooted_QU}, \ref{thm:rooted_to_unrooted}};
\node (B2) at (7, -1) {Unrooted quasi-one-face map $\Mc(\pp)$};
\node (C2) at (7, -2.8) {Symmetric map $\fix(\pp; \phi_{l, m})$};
\node (D2) at (7, -7.4) [align=center] 
{Rooted quasi-one-face maps \\ on orbifold $\Mc_R(\ppsi)$};
}

\tiny{
\draw[double, double distance=2pt, line width=0.6pt] (A1) -- (B1);

\draw[arrow] (B1) -- node[anchor=east, xshift=-0.5em] {Proposition \ref{prop:QU_cong_UW}, \ref{prop:Fill_rooted_eq_fact_rooted}} (U1);

\draw[arrow] (U1) -- node[anchor=east, xshift=-0.5em] {Theorem \ref{thm:CFF_ppl}} (C1);

\draw[arrow] (C1) -- node[anchor=east, xshift=-0.5em, align=right] 
{
    Lemma \ref{lem:cycle_data_split}\\[0.5em]
    Proposition \ref{prop:passport_tree_bij}
} (D1);

\draw[double, double distance=2pt, line width=0.6pt] (D1) -- (E1);

\draw[double, double distance=2pt, line width=0.6pt] (A2) -- (B2);

\draw[arrow] (B2) -- node[anchor=west, xshift=0.5em] {Lemma \ref{lem:unrooted_from_fix_enum}} (C2);

\draw[arrow] (C2) -- node[anchor=west, xshift=0.5em, align=left] 
{
  Lemma \ref{lem:symmetric_map_to_orbifold_map}:\\
  symmetric $\to$ on orbifold \\[0.6em]
  Lemma \ref{lem:orbifold_map_to_symmetric_map}:\\
  on orbifold $\to$ symmetric \\[0.6em]
  Theorem \ref{thm:fix_enum}:\\
  bijective correspondence
} (D2);

\draw[arrow, rounded corners=5pt] (D2) -- (3.7, -7.4) -- (3.7, -1) -- (B1);

}

\draw[dashed, rounded corners=5pt] (-3.3, 0.4) rectangle (3.3, -7.9);
\draw[dashed, rounded corners=5pt] (4.1, 0.4) rectangle (10.8, -7.9);

\end{tikzpicture}
\caption{Proof Route} \label{fig:proof_line}
\end{figure}

\bigskip
%%%%% section %%%%%
\section{Weighted One-Face Maps and Filling of Passports} \label{ch:map_and_passport}

This section supplements the content of maps and passports not mentioned in the introduction.
Section \ref{sec:QU_and_UW} presents another form of quasi-one-face maps, namely weighted one-face maps (Proposition \ref{prop:QU_cong_UW}).
These two forms play different roles in different contexts.
Section \ref{sec:rooted_map_and_filling} discusses the relation between total labeling and general rooted hypermaps (Proposition \ref{prop:Fill_rooted_eq_fact_rooted}).
%This proposition tells us that in the rooted case, regardless of the labeling, the corresponding graph counts differ only by a multiplicative factor.
Finally, these two propositions will show that the counting of rooted quasi-one-face maps can be reduced to the counting of rooted totally labeled weighted one-face maps.

%%%%% subsection %%%%%
\subsection{Quasi-One-Face Maps and Weighted One-Face Maps} \label{sec:QU_and_UW}

Recall the definition of map $M = (S_g; E, U_k)$ and quasi-one-face passport $\ppqu = (g, m, n; \Pi_1, \Pi_2) \\ := (g, n; \Pi_1, \Pi_2, \pi_3 = (m\ 1^{n-m}))$.
Since $\pi_3 = (m\ 1^{n-m})$, the corresponding \textbf{quasi-one-face map} is a map with only one special face $f_0$ of weight $m$ and all other faces of weight $1$.
A \textbf{one-face map} is a map with only one face $U_3 = \{f_0\}$.

We additionally define a \textbf{totally labeled} quasi-one-face map without restriction on the passport as a map $M = (S_g; E, U_k; \Lab_k)$ with bijective labeling functions $\Lab_k : U_k \to [\abs{U_k}], k=1,2$ assigned to vertices.

Given only the genus $g$ and the number of edges $n$, we denote by $\Mc(g, n)$ and $\Uc(g, n)$ the sets of all maps and all one-face maps, respectively.
The corresponding rooted maps and totally labeled rooted maps are written in subscripts, e.g., $\Uc_L(g, n), \Mc_{R, L}(g, n)$, etc.

\bigskip

Note that in a quasi-one-face map, the faces of weight $1$ are digons whose edges are homotopic. 
The edges forming a digon can be ``glued'' into a weighted edge.
It will eliminate all the digons.
Finally, we obtain a bijection between quasi-one-face maps and so-called \textbf{weighted one-face maps}.

%\bigskip

\begin{definition} \label{def:weighted_unic_map}
    The \textbf{weighted one-face map} corresponding to a quasi-one-face passport $\ppqu = (g, m, n; \Pi_1, \Pi_2)$ consists of a one-face map $M = (S_g; U_1 \sqcup U_2, E, \{f_0\}; \Lab_k)$ of genus $g$ with $m$ edges, and a weight function $\wt_E : E \to \Zpos$ on the edges. 
    The weight function $\wt_E$ on edges naturally induces a weight function on vertices
    $$
        \wt : V = U_1 \sqcup U_2 \to \Zpos \ ,\ v \mapsto \sum_{v \in e} \wt_E(e) \ .
    $$

    Similar to Definition \ref{def:labeled_graph}, the weight function defined in this way must be compatible with the weight function in the passport at vertices (but not necessarily at faces), and the multiplicity functions must also be compatible:
    $$
        \wt = \wt_k \circ \Lab_k \ ,\  \abs{\Lab_k^{-1}(s)} = \lambda_k(s) \ ,\ k=1,2 \ .
    $$

    Finally, we see that the total edge weight equals $n$.
\end{definition}

Denote by $\Uc_W(\ppqu)$ the set of all weighted one-face maps corresponding to the quasi-one-face passport $\ppqu$. 
The set of all weighted one-face maps with genus $g$, $m$ weighted edges and $n$ total edge weight is $\Uc_W(g, m, n)$.
Of course, we can also define set like $\Uc_{R, W}(\ppqu), \Uc_{R, W}(g, m, n)$ and $\Uc_{R, W, L}(g, m, n)$, etc.

%\bigskip

\begin{proposition} \label{prop:QU_cong_UW}
    There is a canonical bijection between quasi-one-face maps and weighted one-face maps.
    That is to say, $\Uc_W(\ppqu) \cong \Mc(\ppqu)$ when $\ppqu$ is a quasi-one-face passport.
\end{proposition}

\begin{proof}
For a quasi-one-face map $M_{QU} = (S_g; E, U_k)$, every face $f \in U_3 \setminus \{f_0\}$ other than $f_0$ is a digon. 
Then the two incident edges $e_1, e_2$ of such a face $f$ are homotopic with fixed endpoints. 
This homotopy equivalence relation gives a set of equivalence classes $\ol{E}$. 
Pick one representative edge in each equivalence class $\ol{e} \in \ol{E}$, still denoted by $\ol{e}$, and take the number of elements in the equivalence class as the edge weight $\wt_{\ol{E}}(\ol{e}) = |\ol{e}|$.
This constructs a weighted one-face map $M_W = (S_g; U_1 \sqcup U_2, \ol{E}, \{f_0\}; \wt_{\ol{E}})$. 

Since the edges in an equivalence class $\ol{e}$ are all homotopic with fixed endpoints, they are also isotopic and preserve the vertex set $U_k$. 
Thus the choice of different representative edges in $\ol{e}$ yields equivalent weighted one-face maps. 

\bigskip

Conversely, given a weighted one-face map $M_W = (S_g; U_1 \sqcup U_2, \ol{E}, \{f_0\}; \wt_{\ol{E}})$, for each edge $\ol{e} \in \ol{E}$ we take its neighborhood and draw $(\wt_{\ol{E}}(\ol{e}) - 1)$ additional edges parallel to it within the neighborhood. 
We still denoting by $\ol{e}$ the set consisting of the original edge and the newly drawn edges. 
Let the set of all edges be $E = \bigsqcup \ol{e}$. 
Edges in $\ol{e}$ pairwise enclose $(\wt_{\ol{E}}(\ol{e}) - 1)$ digons, and all digons together with the face $f_0$ form the face set $U_3$. 
Thus we obtain a quasi-one-face map $M_{QU} = (S_g; E, U_k)$.

\bigskip

These two constructions are inverse to each other. Moreover, equivalent $M_{QU}$ correspond to equivalent $M_W$.
Hence there is a canonical bijection between quasi-one-face maps and weighted one-face maps.
Furthermore, quantities such as the genus $g$, the weight of the face $f_0$ (the number of weighted edges) $m$, the number of edges (total edge weight) $n$, the weights of vertices $\wt$, and the labelings $\Lab_k$ are preserved under the construction, thereby proving the statement when the passport $\ppqu$ is given.
\end{proof}

%\bigskip

Recall the definition of a quasi-one-face passport $\ppqu = (g, m, n; \Pi_1, \Pi_2) := (g, n; \Pi_1, \Pi_2, \pi_3 = (m\ 1^{n-m}))$.
There are some additional points to note.
\begin{itemize}
    \item By Euler's formula $v(\ppqu) + (n-m+1) - n = 2 - 2g$, we obtain $m = v(\ppqu) + 2g - 1$.
    Besides, $1 \leq m \leq n$.

    \item For convenience, when $m = 1$, we require $\Pi_3 = (1_1\ 1_2^{n-1})$ to be a weight distribution containing a $1$ with a distinct label i.e., the label carried by the special face $f_0$.
    Since $m = 1$, we have $g = u_1(\ppqu) = u_2(\ppqu) = 1$, and it is easy to see that $\abs{\Mc(\ppqu)} = 1, \abs{\Mc_R(\ppqu)} = n$. 
    Also, $\abs{\Uc_W(\ppqu)} = \abs{\Uc_{R, W}(\ppqu)} = 1$. 

    \item In genus $0$ case, the spherical weighted one-face map is usually represented by a \textbf{plane weighted tree}. 
\end{itemize}

%%%%% subsection %%%%%
\subsection{Rooted maps and Filling of Passports} \label{sec:rooted_map_and_filling}

Recall that a \textbf{rooted map} is a pair $(M, e)$ consisting of a map and a root edge, whose isomorphism preserves the root edge.
Its automorphism group is trivial $\Aut(M, e) \cong \{\id\}$, so the elements of the vertex, edge, and face sets can all be distinguished.
Moreover, the enumeration of rooted maps is related to the weighted enumeration of unrooted maps
\begin{equation*}
    \frac1n \abs{\Mc_R(\pp)} 
    = \sum_{M \in \Mc(\pp)} \frac1{\abs{\Aut(M)}} \ . 
\end{equation*}

Note that for a quasi-one-face passport $\ppqu = (g, m, n; \Pi_1, \Pi_2)$,
the number of edges in the quasi-one-face maps $\Mc_R(\ppqu)$ is $n$, 
and the number of weighted edges in the weighted one-face maps $\Uc_{R, W}(\ppqu)$ is $m$.
Thus we obtain
\begin{equation} \label{eq:QUpassport_MR_and_URW}
      \frac1n \abs{\Mc_R(\ppqu)} 
    = \sum_{M \in \Mc(\ppqu)} \frac1{\abs{\Aut(M)}} 
    = \sum_{M \in \Uc_W(\ppqu)} \frac1{\abs{\Aut(M)}} 
    = \frac1m \abs{\Uc_{R, W}(\ppqu)} \ .
\end{equation}

We can describe the \textbf{filling} of a quasi-one-face passport in an equivalent way.
Recall that the originally defined filling replaces the weight distribution $\Pi_k$ with a sequence $\bmwt_k$ containing $\lambda_k(s)$ positive integers $\wt_k(s)$. 
Each such sequence represents a totally labeled weight distribution. 

\begin{definition}\label{def:filling_of_passport}
    The \textbf{filling} $\Fill(\Pi) = (S_F, \one, \wt_F)$ of a weight distribution $\Pi = (S, \lambda, \wt)$ is defined as
    \begin{equation}\begin{split}
        &S_F := \defsetsmall{(s,k) \in S \times \Zpos}{s \in S, 1 \leq k \leq \lambda(s)} \ , \\
        &\wt_F(s, k) := \wt(s) \ .
    \end{split}\end{equation}
    
    The filling of a quasi-one-face passport $\ppqu = (g, m, n; \Pi_1, \Pi_2)$ is $\Fill(\ppqu) := (g, m, n; \Fill(\Pi_1), \Fill(\Pi_2))$, where $\pi_3 = (m\ 1^{n-m})$ remains unchanged.
\end{definition}

Recall the factorial of a quasi-one-face passport
\begin{equation*}
    (\ppqu)! := \prod_{s \in S_1} \lambda_1(s) \cdot \prod_{s \in S_2} \lambda_2(s)  \ .
\end{equation*}
The following theorem states that for rooted quasi-one-face maps, the enumeration of maps 
with general passport and its filling, differ only by the passport factorial as a multiplicative factor.

\begin{proposition} \label{prop:Fill_rooted_eq_fact_rooted}
    Given a quasi-one-face passport $\ppqu = (g, m, n; \Pi_1, \Pi_2)$.
    The number of rooted maps corresponding to it and its filling $\Fill(\ppqu)$ are related by
    \begin{equation}
        \abs{\Mc_R(\Fill(\ppqu))} = (\ppqu)! \abs{\Mc_R(\ppqu)} \ .
    \end{equation}

\end{proposition}

\begin{proof}
By the definition of the filling of a weight distribution $\Pi = (S, \lambda, \wt)$, we can naturally define a forgetting mapping $\ff$ that forgets the second component $k$ used for distinguishing in $(s, k)$
\begin{equation*}
    \ff : S_F \to S \ ,\ (s, k) \mapsto s \ .
\end{equation*}

Thus, for the weight distributions in the passports $\Fill(\ppqu)$ and $\ppqu$, we have two forgetting mappings $\ff_k : S_{F, k} \to S_k, k=1,2$, which induce a surjection between maps
\begin{equation*}\begin{split}
    \ff_* \ :\ \Mc_R(\Fill(\ppqu)) \qquad &\twoheadrightarrow \qquad \Mc_R(\ppqu) \\ 
    M_L = (S_g; E, U_k; \Lab_{L, k}) \ &\mapsto \ M = (S_g;  E, U_k; \ff_k \circ \Lab_{L, k})
\end{split}
\ .
\end{equation*}

$\ff_*$ does not change the geometric shape of the map; it only sends the labeling function $\Lab_{L, k}$ to $\ff_k \circ \Lab_{L, k}$.

Given $\Lab_1 : U_1 \to S_1$, it is easy to see that the number of bijections $\Lab_{1, L} : U_1 \to S_{U_3, 1}$ satisfying $\ff_1 \circ \Lab_{1, L} = \Lab_1$ is $\prod_{s \in S_1} \lambda_1(s)$.
Similarly, given $\Lab_2$, the number of bijections $\Lab_{2, L}$ satisfying $\ff_2 \circ \Lab_{2, L} = \Lab_2$ is $\prod_{s \in S_2} \lambda_2(s)$.
Thus, given $M$, there are $(\ppqu)!$ maps $M_L$ corresponding to it.
\end{proof}

\bigskip
%%%%% section %%%%%
\section{Enumeration of Rooted Quasi-One-Face Maps} \label{ch:rooted_quasi-unicellular_map_enum}

The aim of this section is to prove Theorem \ref{thm:CFF_general_passport_formula}, i.e., the enumeration of rooted quasi-one-face maps. 
The previous section indicated that this problem can be reduced to the enumeration of rooted totally labeled weighted one-face maps. 
Thus, we will first prove a version of Theorem \ref{thm:CFF_general_passport_formula} for totally labeled weighted one-face maps: Theorem \ref{thm:CFF_full_passport_formula}. 

In Sections \ref{sec:C-decorated_tree} and \ref{sec:CFF_bij}, we outline the C-decorated trees defined by Chapuy et al.\cite{CFF13} for studying one-face maps, and the combinatorial bijection between them (Theorem \ref{thm:CFF_original_bij}). 
This combinatorial bijection ``unfolds'' a one-face map of positive genus into a plane tree with some structures, called a C-decorated tree. 
We generalize the bijection to the totally labeled weighted case(Theorem \ref{thm:CFF_ppl}). 
And reduce the enumeration of weighted one-face maps to the enumeration of totally labeled weighted C-decorated trees (Theorem \ref{thm:CFF_ppl}).

In Section \ref{sec:C-decorated_tree_and_totally_labeled_tree}, we give the correspondence between totally labeled weighted C-decorated trees and totally labeled weighted trees.
The adjustment here arises because the labeled objects is different:  
cycles for C-decorated tree and vertices for ordinary tree. 
Finally, the enumeration of C-decorated trees can be reduced to the enumeration of totally labeled trees (Proposition \ref{prop:passport_tree_bij}). 
The latter one can be obtained via the formula \ref{thm:YYK_formula}. 

In Section \ref{sec:CFF_passport_formula}, based on the former transformations (weighted one-face map $\to$ C-decorated tree $\to$ totally labeled tree), we present the proof of Theorem \ref{thm:CFF_full_passport_formula} and then Theorem \ref{thm:CFF_general_passport_formula}. 
Moreover, using the enumeration Theorem \ref{thm:CFF_general_passport_formula}, we give a new simple proof of the Hurwitz existence problem in the case of positive genus and $\pi_3 = (m\ 1^{n-m})$ (Corollary \ref{cor:exist_quasi-unicellular_map}).

%%%%% subsection %%%%%
\subsection{C-Decorated Trees} \label{sec:C-decorated_tree}

C-decorated trees are plane trees decorated at vertices by C-permutations. Their definition is as follows.

\begin{definition} \label{def:C-permutation}
    A \textbf{C-permutation} $\tau^* := (\tau, \sg)$ of genus $g$ and vertex number $m+1$ is defined as follows:
    \begin{itemize}
        \item Take a permutation $\tau \in \Sf_{m+1}$ in the symmetric group of order $m+1$, requiring that the number of its cycles is $c(\tau) = m+1 - 2g$, and that each cycle has odd length;

        \item In the cycle decomposition $\tau = \alpha_1 \cdots \alpha_{c(\tau)}$, each cycle is assigned a sign $+$ or $-$.
        Concretely, let the \textbf{cycle set} be $\cyc(\tau) := \{\alpha_1, \ldots, \alpha_{c(\tau)}\}$, then a sign function $\sg : \cyc(\tau) \to \{+, -\}$ should be given.
    \end{itemize}
\end{definition}

An example of a C-permutation is $\tau^*_1 = \ ^+(1) \ ^-(2,5,6) \ ^-(3) \ ^+(4)$.
However, $\tau^*_2 = \ ^+(1,2,3) \  ^-(4,5) \  ^-(6)$ is not a C-permutation because not all cycles have odd length.

We use the function $\mu : \cyc(\tau) \to \NN$ to record the length of each cycle in the C-permutation $\ell(\alpha) =2 \mu(\alpha) + 1$, so that the cycle length is automatically odd.

%\bigskip

\begin{definition} \label{def:C-decorated_tree}
    A \textbf{C-decorated tree} $(T, e_T, \tau^*)$ of genus $g$ and edge number $m$ consists of an unlabeled rooted tree $(T, e_T) \in \Uc_R(0, m)$ with $m$ edges and a C-permutation $\tau^*$ of genus $g$ on $m+1$ vertices. 
    Since C-decorated trees are required to be rooted, their vertices are distinguishable. 
    The C-permutation $\tau^*$ on it acts as a permutation $\tau \in \Sf_{V}$ on the vertex set, and it is required that vertices within the same cycle have the same color. 
\end{definition}

The color of the vertices in a cycle also gives the color (not the sign) of the cycle $\cyc(\tau) = \cyc_1(\tau) \sqcup \cyc_2(\tau)$. The number of cycles of each color is denoted by $c_k(\tau) = \abs{\cyc_k(\tau)}, k=1,2$.

The set of all C-decorated trees of genus $g$ and edge number $m$ forms the set $\Cc_{R}(g, m)$.

\bigskip

We wish to consider fully (cycle-)labeled C-decorated trees. Here, labeling means giving a bijection $\Lab_k : \cyc_k(\tau) \to [c_k(\tau)], k=1,2$ on the cycle set.
The set of such \textbf{totally labeled C-decorated trees} $(T, e_T, \tau^*; \Lab_k)$ is denoted by $\Cc_{R, L}(g, m)$.

Given the labeling function $\Lab_k$, we can arrange all cycles, i.e., define the cycle $\alpha^{(k)}_i := \Lab_k^{-1}(i)$.
Thus, the $i$-th cycle in $\cyc_1(\tau) = \{\alpha^{(1)}_1, \ldots, \alpha^{(1)}_{c_1(\tau)}\}$ is $\alpha^{(1)}_i$, and similarly for white cycles.

\bigskip

Next, we also consider totally labeled weighted C-decorated trees, i.e., assigning a positive integer weight $\wt_E : E \to \Zpos$ to each edge as in Definition \ref{def:weighted_unic_map}.
Equivalently, the $(T, e_T) \in \Uc_R(0, m)$ in Definition \ref{def:C-decorated_tree} can be replaced by taking a weighted tree $((T, \wt_E), e_T) \in \Uc_{R, W}(0, m, n)$.
The set of \textbf{totally labeled weighted C-decorated trees} $(T, e_T, \tau^*; \Lab_k, \wt_E)$ of genus $g$, weighted edge number $m$, and total weight $n$ is denoted by $\Cc_{R,W, L}(g, m,n)$.

Similarly, at the vertices of a totally labeled weighted C-decorated tree, there is naturally a weight function $\wt_{T} : V \to \Zpos, v \mapsto \sum_{v \in e} \wt_E(e)$. Now we define the \textbf{weight on a cycle} as the sum of the weights at the vertices it acts on: let the set of vertices acted on by the cycle $\alpha$ be $V(\alpha)$, then 
$$
    \wt_C : \cyc(\tau) \to \Zpos, \alpha \mapsto \sum_{v \in V(\alpha)} \wt_T(v) \ .
$$

%\bigskip

Finally, we can define the \textbf{totally labeled passport} $\ppl = (g, m, n; \bmwt_{C, 1}, \bmwt_{C, 2})$ of a totally labeled weighted C-decorated tree to be the requirement that the black cycle $\alpha^{(1)}_i$ labeled $i$ has weight $\wt_C(\alpha^{(1)}_i) = \wt_{C, 1}(i)$, where the sequence $\bmwt_{C, 1} = (\wt_{C, 1}(i))_{i=1}^{c_1(\tau)}$.
Similarly, the weights of the white cycles are required to be compatible with the weights in the passport $\wt_C(\alpha^{(2)}_i) = \wt_{C, 2}(i)$.
The set of totally labeled weighted C-decorated trees with passport $\ppl$ is denoted by $\Cc_{R, W}(\ppl)$.

%\bigskip

\begin{definition} \label{def:cycle_data}
    The \textbf{cycle data} $\Lambda = (\bmmu_1, \bmmu_2)$ of a totally labeled weighted C-decorated tree $(T, e_T, \tau^*; \Lab_k, \wt_E) \in \Cc_{R,W, L}(g, m,n)$ consists two sequences.
    The sequences
    \begin{equation}
        \bmmu_k = (\mu_{k, j})_{j=1}^{c_k(\tau)} := (\mu(\alpha^{(k)}_j))_{j=1}^{c_k(\tau)} \ ,\ 
        k = 1,2 
    \end{equation}
    record the lengths of the black and white cycles respectively. 
    Recall that the $\mu$ function determines the cycle length $\ell(\alpha) = 2\mu(\alpha) + 1$.
\end{definition}

Let $\Cc_{R, W, L}(g, m, n; \Lambda)$ and $\Cc_{R, W}(\ppl; \Lambda)$ denote corresponding the totally labeled weighted C-decorated trees with cycle data $\Lambda$.

Since the C-permutation is taken in the symmetric group of order $m+1$,
\begin{equation*}\begin{split}
    m+1 
    &= \sum_{i=1}^{c_1(\tau)} (2 \mu_{1,i} + 1) + \sum_{i=1}^{c_2(\tau)} (2 \mu_{2,i} + 1)  \\
    &= 2 \norm{\bmmu_1}  + 2 \norm{\bmmu_2} + c(\tau) 
    = 2 \norm{\bmmu_1}  + 2 \norm{\bmmu_2} + m+1-2g \ ,
\end{split}\end{equation*}
thus the cycle data must satisfy the condition $\norm{\bmmu_1} + \norm{\bmmu_2} = g$, recall that this condition is $\Lambda \vdash g$.
Another condition is that the lengths of the two sequences in the cycle data equal the number of vertices $\abs{\bmmu_k} = u_k(\pp), k=1,2$, and $\abs{\bmmu_1} + \abs{\bmmu_2} = m+1 - 2g$. 

\begin{lemma} \label{lem:cycle_data_split}
    Cycle data can be used to finely classify C-decorated trees
    \begin{equation}
        \Cc_{R, W, L}(g, m, n) 
        = \bigsqcup_{\Lambda \vdash g} 
        \Cc_{R, W, L}(g, m, n; \Lambda) \ ,\ 
        \Cc_{R, W}(\ppl) 
        = \bigsqcup_{\Lambda \vdash g} 
        \Cc_{R, W}(\ppl; \Lambda) \ .
    \end{equation}
    where $\Lambda = (\bmmu_1, \bmmu_2)$ ranges over all pairs of sequences of nonnegative integers satisfying $\Lambda \vdash g$ and the conditions $\abs{\bmmu_1} + \abs{\bmmu_2} = m+1 - 2g$ or$\abs{\bmmu_k} = u_k(\pp), k=1,2$.
    \qed
\end{lemma}

\bigskip
%%%%% subsection %%%%%
\subsection{Bijections for C-Decorated Trees} \label{sec:CFF_bij}

Chapuy et al.\cite{CFF13} described a combinatorial bijection that maps one-face maps of positive genus to C-decorated trees.
Informally, this bijection splits some vertices of a one-face map into several vertices, without altering other graph objects, until the genus of the map decreases to $0$, yielding a plane tree.
The splitting information is recorded by the C-permutation on the vertices, which can be used to recover the original one-face map.

\begin{definition}
    The disjoint union of $k \in \Zpos$ copies of a set $A$ is denoted by $kA$. Explicitly, we may set $kA := \defsetsmall{(i, x)}{1 \leq i \leq k, x \in A}$.
\end{definition}

\begin{theorem}[Chapuy et al.\cite{CFF13}] \label{thm:CFF_original_bij}
    Given genus $g$ and number of edges $m$, there exists a bijection between C-decorated trees and rooted one-face maps
    \begin{equation}
        \CFF : \Cc_{R}(g, m) \cong 2^{m+1} \Uc_{R}(g, m) \ .
    \end{equation}
    For the corresponding rooted one-face map $\CFF(T, e_T, \tau^*) = (i, (M, e_M))$, when forgetting the embedding structure and considering only the bicolored abstract graph structure, they are denoted respectively by $T' = (U_{T, 1} \sqcup U_{T, 2}, E_T)$ and $M' = (U_{M, 1} \sqcup U_{M, 2}, E_M)$. 
    If one glues together the vertices $V(\alpha)$ that are acted by cycles $\alpha$ in $T'$, one obtains $M'$. Moreover, the root edges correspond to each other $\CFF(e_T) = e_M$.
    \qed
\end{theorem}

%\bigskip

Specifically, there is a color-preserving bijection between cycles and vertices $\CFF : \cyc_k(\tau) \cong U_{M, k}, k=1,2$, and a bijection between edge sets $\CFF : E_T \cong E_M$. The incidence relations between vertices and edges are also preserved: if we denote the set of edges incident to a vertex $v$ by $E(v)$, then for the corresponding cycle and vertex $\CFF(\alpha) = v_M$, we have
$\bigsqcup_{v_T \in V(\alpha)} E(v_T) = E(v_M)$.

\begin{remark}
    The original bijection of Chapuy et al.\cite{CFF13} is a correspondence between \textbf{monochrome} C-decorated trees and rooted one-face maps. 
    Theorem \ref{thm:CFF_original_bij} states a correspondence between \textbf{bicolored} graphs, where vertices on both sides are assigned colors. The latter can be easily deduced from the former.
\end{remark}

\bigskip

Since $\CFF$ gives a bijection between edge sets $E_T \cong E_M$, we can simultaneously assign weight functions $\wt_{T, E} : E_T \to \Zpos$ and $\wt_{M, E} : E_M \to \Zpos$ as in Definition \ref{def:weighted_unic_map} on the two edge sets, requiring that corresponding edges have the same weight $\wt_{M, E} \circ \CFF= \wt_{T, E}$, and that the total edge weight is $\sum_{e_T' \in E_T} \wt_{T, E}(e_T') = \sum_{e_M' \in E_M} \wt_{M, E}(e_M') = n$. This yields a weighted version of the $\CFF$ bijection:
\begin{equation}
    \CFF_W : \Cc_{R, W}(g, m, n) \cong 2^{m+1} \Uc_{R, W}(g, m, n)  \ .
\end{equation}

Furthermore, $\CFF$ gives a bijection between cycles and vertices $\cyc_k(\tau) \cong U_{M, k}, k=1,2$. Thus, we define labeling functions $\Lab_{C, k} : \cyc_k(\tau) \to [c_k(\tau)], k=1,2$ and $\Lab_{M, k} : U_{M, k} \to [c_k(\tau)], k=1,2$, requiring that corresponding cycles and vertices have the same label $\Lab_{M, k} \circ \CFF= \Lab_{T, k}$. This yields a totally labeled and weighted version of the $\CFF$ bijection:
\begin{equation} \label{eq:CFF_WL}
    \CFF_{W, L} : \Cc_{R, W, L}(g, m, n) \cong 2^{m+1} \Uc_{R, W, L}(g, m, n) \ .
\end{equation}

Since the incidence relations between vertices and edges are preserved, for $\CFF(\alpha) = v_M$ we have $\bigsqcup_{v_T \in V(\alpha)} E(v_T) = E(v_M)$. Hence the weight functions $\wt_C, \wt_M$ on the C-decorated tree and the one-face map satisfy
\begin{equation*}
    \wt_C(\alpha) 
    = \sum_{v_T \in V(\alpha)} \wt_T(v_T)
    = \sum_{v_T \in V(\alpha)} \sum_{e \in E(v)} \wt_{T, E}(e)
    = \sum_{e \in E(v_M)} \wt_{M, E}(e)
    = \wt_M(v_M) \ .
\end{equation*}
Therefore, the passports of the C-decorated tree and the weighted one-face map coincide. Finally, we obtain the $\CFF$ bijection restricted by the passport.

\begin{theorem} \label{thm:CFF_ppl}
    Given a totally labeled quasi-one-face passport $\ppl$, there is a bijection between totally labeled weighted C-decorated trees and totally labeled rooted weighted one-face maps
    \begin{equation} \label{eq:CFF_ppl}
        \CFF_{\ppl} : \Cc_{R, W}(\ppl) \cong 2^{m+1} \Uc_{R, W}(\ppl) \ . 
    \end{equation}
    \qed
\end{theorem}

%%%%% subsection %%%%%
\subsection{C-Decorated Trees and Totally Labeled Trees} \label{sec:C-decorated_tree_and_totally_labeled_tree}

Now take a totally labeled weighted C-decorated tree $(T, e_T, \tau^*; \Lab_{C, k}, \wt_E) \in \Cc_{R, W, L}(g, m, n; \Lambda)$.
Here the labeling $\Lab_{C, k} : \cyc_k(\tau) \to [c_k(\tau)]$ assigns labels to cycles.

There is an unlabeled rooted weighted tree $(T = (S^2; E, U_k), \wt_E, e_T) \in \Uc_{R,W}(0, m, n)$ under the C-decorated tree. 
We wish to further assign a total labeling $\Lab_{T, k} : U_k \to S_{T, k}, k=1,2$ to its vertices. 

Since the index set in a total labeling can be chosen arbitrarily as long as its size matches, for convenience we use the index set completely determined by the cycle data $\Lambda$ (where $c_k(\tau) = u_k(\pp)$ are the numbers of cycles) as in \eqref{eq:index_set_T}
\begin{equation*}
    S_{T, k} = \defsetsmall{(i,j)}{1 \leq i \leq c_k(\tau), 1 \leq j \leq 2\mu_{k, i}+1} \ ,\ k=1,2 \ .
\end{equation*}
Denote the corresponding vertices by $v^{(k)}_{i,j} := \Lab_{T, k}^{-1}(i, j)$.

We naturally want the first component $i$ of a vertex $v^{(k)}_{i,j}$ to indicate that $v^{(k)}_{i,j} \in V(\alpha^{(k)}_i)$, since the latter has exactly $\big| V(\alpha^{(k)}_i) \big| = 2\mu_{k, i}+1$ elements. 
It remains to determine the second component for the vertices in $V(\alpha^{(k)}_i)$. 
If we have chosen the vertex with second component $1$ to be $v^{(k)}_{i,1} \in V(\alpha^{(k)}_i)$, then, following the cyclic order, we can successively assign labels to the vertices in $V(\alpha^{(k)}_i)$
\begin{equation*}
    v^{(k)}_{i,j} := \big(\alpha^{(k)}_i\big)^{j-1} \cdot v^{(k)}_{i,1} \ ,\ k=1,2 \ ,\ 1 \leq i \leq c_k(\tau) \ ,\ 1 \leq j \leq 2\mu_{k, i} +1 \ .
\end{equation*}
For each cycle $\alpha^{(k)}_i$, this choice of a distinguished vertex admits $\big| V(\alpha^{(k)}_i) \big| = 2\mu_{1,i} +1$ possibilities. 
If we make such a choice for all black cycles and all white cycles, the total number of choices is exactly $R(\Lambda)$ as in \eqref{eq:def_R_Lambda}
\begin{equation*}
    R(\Lambda) := \prod_{i=1}^{c_1(\tau)} (2 \mu_{1,i} + 1) \cdot \prod_{i=1}^{c_2(\tau)} (2 \mu_{2,i} + 1) \ .
\end{equation*}
Through these choices, all vertices receive labels $\Lab_{T, k}, k=1,2$. This yields the following proposition.

\begin{proposition} \label{prop:gmn_tree_bij}
    Given $g, m, n$ and admissible cycle data $\Lambda$, there is a bijection between totally labeled weighted C-decorated trees and totally labeled weighted trees
    \begin{equation}
        R(\Lambda) \ \Cc_{R, W, L}(g, m, n; \Lambda) \cong 2^{m+1-2g} \ \Uc_{R, W, L}(0, m, n) \ .
    \end{equation}
\end{proposition}

\begin{proof}
From the preceding discussion, given a totally labeled weighted C-decorated tree $(T, \wt_E, e_T, \tau^*; \Lab_{C, k}) \in \Cc_{R, W, L}(g, m, n; \Lambda)$ and making a choice among $R(\Lambda)$ possibilities, we obtain a totally labeled rooted weighted tree $(T, \wt_E, e_T; \Lab_{T, k}) \in \Uc_{R, W, L}(0, m, n)$. 
Moreover, the total labeling on trees does not encode the signs $\sg : \cyc(\tau) \to \{+, -\}$ of the cycles in the C-permutation $\tau^*$; there are $2^{c(\tau)} = 2^{m+1-2g}$ choices for these signs. 
This gives a map
\begin{equation*}
    R(\Lambda) \ \Cc_{R, W, L}(g, m, n; \Lambda) \to 2^{m+1-2g} \ \Uc_{R, W, L}(0, m, n) \ .
\end{equation*}

Conversely, given a totally labeled tree $(T, \wt_E, e_T; \Lab_{T, k})$ and a sign function $\sg : \cyc(\tau) \to \{+, -\}$, denote the vertex with label $(i, j)$ by $v^{(k)}_{i, j} := \Lab_{T, k}^{-1}(i, j)$. 
Then the cycles can be defined as $\alpha^{(k)}_i := (v^{(k)}_{i, 1}, v^{(k)}_{i, 2}, \ldots, v^{(k)}_{i, 2\mu_{1, i} + 1})$, with labeling $\Lab_{C, k}(\alpha^{(k)}_i) = i$. 
This defines the labeling functions $\Lab_{C, k}, k=1,2$. 
Multiplying the cycles yields the permutation $\tau$, and hence the C-permutation $\tau^* = (\tau, \sg)$ and the C-decorated tree.
\end{proof}

\bigskip

Now given a totally labeled quasi-one-face passport $\ppl = (g, m, n; \bmwt_{C, 1}, \bmwt_{C, 2})$. 
For a C-decorated tree $(T, \wt_E, e_T, \tau^*; \Lab_{C, k}) \in \Cc_{R, W}(\ppl; \Lambda)$ with this passport, its underlying unlabeled rooted weighted tree $(T, \wt_E , e_T)$ has a total labeling passport $\ppzl = (0, m, n; \Pi_{T, 1} = (S_{T,1}, \one, \wt_{T, 1}), \Pi_{T, 2} = (S_{T, 2}, \one, \wt_{T, 2}))$. 
The passport $\ppzl$ uses the index sets $S_{T, 1}, S_{T,2}$ determined by the cycle data $\Lambda$, and its weight distribution can be written as two-dimensional arrays $\bmwt_{T, 1} = (\wt_{T, 1}(i, j))_{i, j}, \bmwt_{T, 2} = (\wt_{T, 2}(i, j))_{i, j}$ (where the rows may have different lengths).

We finally denote $\ppzl = (0, m, n; \bmwt_{T, 1}, \bmwt_{T, 2})$. 
Now determine the conditions that the passport $\ppzl$ must satisfy.

By definition, the weight of a cycle is the sum of the weights of the vertices it acts on, which gives
\begin{equation*}
    \wt_{C, k}(i) = \wt_C \big(\alpha^{(k)}_i \big) = \sum_{v \in V \big( \alpha^{(k)}_i \big)} \wt_T(v) = \sum_{j=1}^{2\mu_{k, i}+1} \wt_T \big(v^{(k)}_{i, j} \big) = \sum_{j=1}^{2\mu_{k, i}+1} \wt_{T, k}(i, j) \ .
\end{equation*}
This condition is exactly $\ppzl \xrightarrow{\Lambda} \ppl$ (see \eqref{eq:cond_passport_row_sum}).

\begin{proposition} \label{prop:passport_tree_bij}
    Let $\ppl = (g, m, n; \bmwt_{C, 1}, \bmwt_{C, 2})$ be a totally labeled one-face passport and $\Lambda$ an admissible cycle data, 
    there is a bijection between totally labeled weighted C-decorated trees and totally labeled weighted trees
    \begin{equation}
        R(\Lambda) \ \Cc_{R, W}(\ppl; \Lambda) \cong 
        2^{m+1-2g} 
        \bigsqcup_{\ppzl \xrightarrow{\Lambda} \ppl} 
        \Uc_{R, W}(\ppzl) \ .
    \end{equation}
    Here the disjoint union runs over all \textbf{distinct} passports $\ppzl$ satisfying $\ppzl \xrightarrow{\Lambda} \ppl$. 
    Recall that passports are considered distinct if any component $\wt_{T, k}(i, j)$ differs. 
\end{proposition}

\begin{proof}
The bijection in this proposition is the restriction of the bijection from Proposition \ref{prop:gmn_tree_bij}. 
%Moreover, by the preceding argument, a C-decorated tree with passport $\ppl$ corresponds to a totally labeled tree with passport $\ppzl$ satisfying $\ppzl \xrightarrow{\Lambda} \ppl$.

Note that the weights $\wt_{T, k}(i, j)$ in the passport $\ppzl$ represent the vertex weights $\wt_T(v^{(k)}_{i, j})$. 
If two distinct passports satisfy $\ppzl \xrightarrow{\Lambda} \ppl$ and differ in some component $\wt_{T, k}(i, j)$, then the corresponding vertex weights differ, hence the corresponding totally labeled trees are certainly different. 
Therefore the right-hand side of the bijection should take the disjoint union over all distinct passports $\ppzl$ satisfying $\ppzl \xrightarrow{\Lambda} \ppl$.
\end{proof}

%%%%% subsection %%%%%
\subsection{Rooted Quasi-One-Face Map Enumeration Formula} \label{sec:CFF_passport_formula}

\begin{theorem} \label{thm:CFF_full_passport_formula}
    Let $\ppl = (g, m,  n; \bmwt_{C, 1},\bmwt_{C, 2})$ be a totally labeled quasi-one-face passport. 
    The number of rooted weighted one-face maps is
    \begin{equation}
        \abs{\Uc_{R, W}(\ppl)} = 2^{-2g} \sum_{\Lambda \vdash g} R(\Lambda)^{-1} \sum_{\ppzl \xrightarrow{\Lambda} \ppl} \abs{\Uc_{R, W}(\ppzl)} \ .
    \end{equation}
\end{theorem}

\begin{proof}
By Theorem \ref{thm:CFF_ppl}, Lemma \ref{lem:cycle_data_split}, Proposition \ref{prop:passport_tree_bij}
\begin{equation*}\begin{split}
    \abs{\Uc_{R, W}(\ppl)} 
    &= 2^{-m-1} \abs{\Cc_{R, W}(\ppl)} 
    = 2^{-m-1} \sum_{\Lambda \vdash g} \abs{\Cc_{R, W}(\ppl; \Lambda)} \\
    &= 2^{-m-1} \sum_{\Lambda \vdash g} R(\Lambda)^{-1} \cdot 2^{m+1-2g} \sum_{\ppzl \xrightarrow{\Lambda} \ppl} \abs{\Uc_{R, W}(\ppzl)} \ .
    \qedhere
\end{split}\end{equation*}
\end{proof}

\begin{proof}[Proof of Theorem \ref{thm:CFF_general_passport_formula}]
Given a quasi-one-face passport $\ppqu = (g, m, n; \Pi_1, \Pi_2)$.
By Proposition \ref{prop:Fill_rooted_eq_fact_rooted}, Formula \eqref{eq:QUpassport_MR_and_URW}, we obtain
\begin{equation*}
    \abs{\Mc_R(\ppqu)} = \frac1{(\ppqu)!} \abs{\Mc_R(\Fill(\ppqu))} = \frac{n}{m\ (\ppqu)!} \abs{\Uc_{R, W}(\Fill(\ppqu))} \ .
\end{equation*}
Using Theorem \ref{thm:CFF_full_passport_formula}, we get
\begin{equation*}
    \abs{\Uc_{R, W}(\Fill(\ppqu))} = 2^{-2g} \sum_{\Lambda \vdash g} R(\Lambda)^{-1} \sum_{\ppzl \xrightarrow{\Lambda} \Fill(\ppqu)} \abs{\Uc_{R, W}(\ppzl)} \ .
\end{equation*}
Finally we need to compute $\abs{\Uc_{R, W}(\ppzl)} = m \abs{\Uc_{W}(\ppzl)} = m \abs{\Mc(\ppzl)}$ (by Proposition \ref{prop:weighted_tree_no_isom}),
and the latter $\abs{\Mc(\ppzl)}$ can be computed by \ref{thm:YYK_formula}.
Thus we obtain
\begin{equation*}
    \abs{\Mc_R(\ppqu)} = \frac{n}{\ 2^{2g}\ (\ppqu)!}  \sum_{\Lambda \vdash g} R(\Lambda)^{-1} \sum_{\ppzl \xrightarrow{\Lambda} \Fill(\ppqu)} \abs{\Mc(\ppzl)} \ . \qedhere
\end{equation*}
\end{proof}

\bigskip

For applications of the main theorem \ref{thm:CFF_general_passport_formula}, see Appendix \ref{ch:example}.

As another application of this theorem, we prove the Hurwitz existence problem for positive genus $g > 0$ and $\pi_3 = (m\ 1^{n-m})$. 

\begin{proof}[Proof of Corollary \ref{cor:exist_quasi-unicellular_map}]
By Theorem \ref{thm:CFF_general_passport_formula},
we only need to show that there exist $\Lambda$ and $\ppzl \xrightarrow{\Lambda} \ppl$ such that $\Mc_{R}(\ppzl) \neq \varnothing$.

First, construct the cycle data $\Lambda$.
Define two sequences of nonnegative integers $\bm{\xi}_k = (\xi_{k, i})_{i=1}^{u_k(\pp)}$ satisfying
\begin{equation*}
    \xi_{k, i} := \lfloor  (\wt_{C, k}(i) - 1) / 2  \rfloor \ ,
\end{equation*}
then we have $\xi_{k, i} \geq (\wt_{C, k}(i) - 1) / 2 - 1/2 = \wt_{C, k}(i) / 2 - 1$. Since $m = 2g - 1 + v(\ppl) \leq n$, we obtain
\begin{equation*}
    \norm{\bm{\xi}_1} + \norm{\bm{\xi}_2} 
    \geq \frac{\norm{\bmwt_1}}{2} - u_1(\ppl) + \frac{\norm{\bmwt_2}}{2} - u_2(\ppl)
    = n - v(\ppl)
    \geq 2g - 1 
    \geq g \ .
\end{equation*}
Thus we can choose nonnegative integer sequences $\bmmu_k, k=1,2$ of appropriate lengths such that $0 \leq \mu_{k, i} \leq \xi_{k, i}$ and $\norm{\bmmu_1} + \norm{\bmmu_2} = g$.
Denote the cycle data by $\Lambda = (\bmmu_1, \bmmu_2)$.

Next, construct the passport $\ppzl = (0, m, n; \bmwt_{T, 1}, \bmwt_{T, 2}) \xrightarrow{\Lambda} \ppl$.
Since $2\mu_{k, i} + 1 \leq 2 \xi_{k, i} +1 \leq \wt_{C, k}(i)$, the number $\wt_{C, k}(i)$ can always be partitioned into $2\mu_{k, i} + 1$ positive integers $\wt_{T, k}(i, j), 1 \leq j \leq 2\mu_{k, i} + 1$.
Since $g > 0$, there exists some $\mu_{k, i} > 0$, hence some $\wt_{C, k}(i)$ is partitioned into at least $3$ parts, so we may additionally require that there exists a part $\wt_{T, k}(i, j) = 1$.
Thus all $\wt_{T, k}(i, j)$ are defined, the passport $\ppzl$ is given.

Finally, verify that $\Mc_R(\ppzl) \neq \varnothing$.
By Theorem \ref{thm:exist_tree}, it suffices to show that $(v(\ppzl) - 1) \cdot \gcd(\ppzl) \leq n$.
Since there exists a part $\wt_{T, k}(i, j) = 1$, we have $\gcd(\ppzl) = 1$.
And $v(\ppzl) - 1 = m \leq n$.
\end{proof}

\bigskip
%%%%% section %%%%%
\section{Algebraic Maps and Orbit Enumeration} \label{ch:algebraic_map}

The goal of this section is to prove Lemma \ref{lem:unrooted_from_fix_enum}, which shows that the enumeration of unrooted maps can be derived from the enumeration of symmetric maps. 
This will ultimately lead to the proof of Theorem \ref{thm:rooted_to_unrooted} on the enumeration of unrooted maps. 

To obtain this lemma, we introduce the concept of algebraic maps in Section \ref{sec:algebraic_map}. Algebraic maps represent maps through a simpler combinatorial structure: the conjugacy classes of two elements in permutation group. 
We also introduce the regular cover of an algebraic map, which is analogous to taking the ``quotient'' of a map.

In Section \ref{sec:rooted_and_numbered_map}, by rewriting the definition of isomorphisms for ordinary algebraic maps, we obtain rooted algebraic maps and totally numbered algebraic maps.
A totally numbered algebraic map is completely determined by the elements of two permutation groups, with no isomorphism relations involved, making it the simplest type of algebraic map.
Rooted maps and ordinary algebraic maps can each be viewed as sets of orbits of totally numbered maps under the conjugacy action of different groups (Lemma \ref{lem:group_action_on_numbered_map}).

Finally, in Section \ref{sec:rooted_and_numbered_map}, Burnside's Lemma shows that the enumeration of unrooted algebraic maps, i.e., the enumeration of orbits, can be obtained from the enumeration of symmetric maps. 
This proves Lemma \ref{lem:unrooted_from_fix_enum}.
Subsequently, we will take the ``quotient'' of symmetric maps, i.e., consider the quotient maps of their regular covers, to obtain their enumeration.

%%%%% subsection %%%%%
\subsection{Algebraic Maps} \label{sec:algebraic_map}

We first define algebraic maps. As we will see later, there is a bijective correspondence between algebraic maps and maps embedded on surfaces, so we may treat them as identical objects without distinction.

\begin{definition} \label{def:comb_map}
    An \textbf{algebraic map} $M = (E, x, y)$ consists of an edge set $E$ and two permutations $x, y \in \Sf_E$ on the edge set, such that the group generated by them $G = \langle x, y \rangle$ acts transitively on $E$.
    The group $G$ is also called the \textbf{cartographic group} (or monodromy group) of $M$.

    Define the set of ``black vertices'' as the set of cycles of $x$, denoted $U_1 := \cyc(x)$, the set of ``white vertices'' as the set of cycles of $y$, denoted $U_2 := \cyc(y)$, and the set of ``faces'' as the set of cycles of $yx$, denoted $U_3 := \cyc(yx)$.
    
    The genus of an algebraic map is defined by Euler's formula: $2 - 2g = \abs{U_1} + \abs{U_2} + \abs{U_3} - \abs{E}$.
    
    The weight of a ``vertex'' or ``face'' is defined as the length of the corresponding cycle $\wt : U_1 \sqcup U_2 \sqcup U_3 \to \Zpos, \alpha \mapsto \ell(\alpha)$.
\end{definition}

For isomorphisms of algebraic maps, there is a more general definition of morphisms, namely cover of the algebraic map. 

\begin{definition}\label{def:comb_map_cover}
    A \textbf{cover} $\phi : M_1 \to M_2$ between two algebraic maps $M_i = (E_i, x_i, y_i), i=1,2$ is defined as a surjection on the edge sets $\phi : E_1 \to E_2$ that commutes with the permutations on the edge sets: $x_2 \circ \phi = \phi \circ x_1, y_2 \circ \phi = \phi \circ y_1$.

    If $\phi : M_1 \to M_2$ is also a bijection, then $\phi$ is called an \textbf{isomorphism}.
\end{definition}

\begin{proposition} [\cite{LsZak04}, Section 1.5.1] \label{prop:map_and_algebraic_map}
    There is a one-to-one correspondence between the sets of non-isomorphic (unlabeled) maps (embedded on surfaces) and non-isomorphic algebraic maps.

    Let $M_G = (S_{g_G}; E_G, U_{G, k})$ be a map with weight function $\wt_G$. Its corresponding algebraic map is $M_C = (E_C, x, y)$, where the sets of ``vertices'' and ``faces'' of the algebraic map are $U_{C,k}$, the genus $g_C$, and the weight function $\wt_C$ given by Definition \ref{def:comb_map}. 
    The correspondence provides the following correspondences between elements
    \begin{equation*}
        E_G \cong E_C \ ,\ U_{G, k} \cong U_{C, k} \ ,\ g_G = g_C \ ,\ \wt_G \equiv \wt_C \ .
        \qedheremath
    \end{equation*}
    
\end{proposition}

We still denote the set of all non-isomorphic algebraic maps of genus $g$ with edge number $\abs{E} = n$ by $\Mc(g, n)$.
Similarly, we can attach a labeling function $\Lab_k : U_k \to [|U_k|]$ to algebraic maps to obtain labeled algebraic maps, and define the passport of a \textbf{labeled algebraic map}.
The set of all labeled algebraic maps with passport $\pp$ is denoted by $\Mc(\pp)$.

\bigskip

We continue to study covers between algebraic maps, which will be of great use in subsequent studies.

An algebraic map cover $\phi : M_1 \to M_2$ also induces a \textbf{epimorphism between cartographic groups} $\phi_G : G_1 \to G_2$, satisfying $\phi_G(h) \circ \phi = \phi \circ h$.
Here $\phi_G$ can be constructed by setting $\phi_G(x_1) = x_2, \phi_G(y_1) = y_2$ and using the property that $G_1, G_2$ are generated by $x_1, y_1$ or $x_2, y_2$, respectively.
An isomorphism $\phi : M_1 \to M_2$ induces an isomorphism of cartographic groups $\phi_G : G_1 \to G_2$.
Moreover, if the edge sets of $M_1$ and $M_2$ are the same set $E$, then $\phi \in \Sf_E$ and the group isomorphism $\phi_G(h) = \phi \circ h \circ \phi^{-1} := h^{\phi}$ is a conjugacy action.

If $\phi \in \Aut(M)$ is an automorphism of an algebraic map, then it must satisfy $\phi_G(h) = h^{\phi} = h$, i.e., $\phi$ commutes with every $h \in G$, which means $\phi \in C_{\Sf_E}(G)$, the centralizer of the cartographic group $G$ in $\Sf_E$.
Thus, the \textbf{automorphism group} $\Aut(M) \cong C_{\Sf_E}(G)$.

Similarly, we can define the \textbf{covering transformation group} of an algebraic map cover $\phi$ as $\Aut(\phi) := \defsetsmall{I \in \Aut(M_1)}{\phi = \phi \circ I}$.
Cover $\phi$ is \textbf{regular cover} if $\Aut(\phi)$ acts transitively on the fiber $\phi^{-1}(e)$.
Every regular cover can be obtained by taking a subgroup $A \leq \Aut(M_1)$ of the automorphism group and defining $M_2 := M_1 / A$ to get $\phi_A : M_1 \to M_2$, where the edge set of $M_2$ is the orbit set $E_2 := E_1 / A$, and $\phi_A : E_1 \to E_2, e \mapsto [e]$ is the canonical surjection.

\bigskip
%%%%% subsection %%%%%
\subsection{Rooted Algebraic Maps and Totally Numbered Maps} \label{sec:rooted_and_numbered_map}

By pairing an algebraic map with a root edge $(M, e)$ and requiring that the isomorphism mapping $I : M_1 \to M_2$ also preserves the root edge $I(e_1) = e_2$, we can define rooted algebraic maps.

However, we can also fix an edge set $E = \{e_1, \ldots, e_n\}$, set the root edge to be $e_1$, and modify the definition of the isomorphism of maps. 
This gives a better definition of rooted algebraic maps. 

\begin{definition}
    A \textbf{rooted algebraic map} with $n$ edges is denoted by $M_R = (E, x, y)$, the same as the algebraic map.
    
    Fix the edge set $E = \{e_1, \ldots, e_n\}$.
    Given two rooted algebraic maps $M_{R, i} = (E, x_i, y_i), i=1,2$, we say that $\phi \in \Sf_E$ is an isomorphism between rooted algebraic maps $M_{R,1}$ and $M_{R, 2}$ if
    $x_2 \circ \phi = \phi \circ x_1, y_2 \circ \phi = \phi \circ y_1$, and $\phi(e_1) = e_1$.
\end{definition}

Rooted algebraic maps correspond one-to-one with rooted maps. The set of all non-isomorphic rooted algebraic maps of genus $g$ with $n$ edges is still denoted by $\Mc_R(g, n)$, and by $\Mc_R(\pp)$ when a passport $\pp$ is given.

\bigskip

We can also consider \textbf{totally (edge-)numbered} maps. 
They still take the form $M_N = (E, x, y)$, but contain no isomorphisms, making the identification of such algebraic maps very easy.
The corresponding totally numbered maps embedded on surfaces are described as follows.

Notice that for vertices and cycles we use ``labeled'', now for edges we use ``numbered''. 

\begin{definition}
    A \textbf{totally (edge-)numbered algebraic map} with $n$ edges is still denoted by $M_N = (E, x, y)$.
    When the same edge set $E = \{e_1, \ldots, e_n\}$ is fixed, two totally numbered algebraic maps are considered identical only if both $x$ and $y$ are exactly the same.

    A \textbf{totally numbered map} $(M, (e_i)_{i=1}^n)$ with $n$ edges consists of a map $M = (S_g; E, U_k)$ and a permutation $(e_i)_{i=1}^n$ of the edge set $E = \{e_1, \ldots, e_n\}$.
    An isomorphism $I : S_g \to S_g$ of totally numbered maps $(M_i, (e_{i, j})_{j=1}^n), i=1,2$ must satisfies $I(M_1) = M_2$ and also preserves the edge permutation $I(e_{1, j}) = e_{2, j}, 1 \leq j \leq n$.

    Totally numbered algebraic maps correspond one-to-one with totally numbered maps. The sets they form are denoted by $\Mc_N(g, n)$ or $\Mc_N(\pp)$. 
\end{definition}

%\bigskip

We can see that, fixing $E = \{e_1, \ldots, e_n\}$,
a bijection $\phi \in \Sf_E$ can be viewed as acting on a totally numbered algebraic map $(E, x, y) \in \Mc_N(g, n)$, yielding a change of the map
\begin{equation} \label{eq:SE_action_on_numbered_map}
    \phi(E, x, y) = (E, x^\phi, y^\phi) \ .
\end{equation}

In the sense of totally numbered maps, when $x \neq x^\phi$ or $y \neq y^\phi$, the maps $(E, x, y)$ and $(E, x^\phi, y^\phi)$ are of course different.
However, when considered as ordinary algebraic maps, the maps $(E, x, y)$ and $(E, x^\phi, y^\phi)$ are isomorphic. 
It is easy to verify that the orbits of this conjugacy action of $\Sf_E$ on $\Mc_N(g, n)$ are exactly the equivalence classes of algebraic maps $\Mc(g, n) = \Mc_N(g, n) / \Sf_E$.

Similarly, the orbits of the subgroup action $\Sf_{E \setminus \{e_1\}} \curvearrowright \Mc_N(g, n)$ are exactly the rooted algebraic maps $\Mc_R(g, n) = \Mc_N(g, n) / \Sf_{E \setminus \{e_1\}}$.
Since this group action is faithful ($\phi \in \Sf_{E \setminus \{e_1\}}, \phi(M_N) = M_N \implies \phi = \id$, i.e., rooted maps have no automorphisms), we have the simple numerical relation $\big| \Sf_{E \setminus \{e_1\}} \big| \cdot \abs{\Mc_R(g, n)} = \abs{\Mc_N(g, n)}$.

\begin{lemma} \label{lem:group_action_on_numbered_map}
    Given genus $g$ and number of edges $n$, fix $E = \{e_1, \ldots, e_n\}$. 
    Formula \eqref{eq:SE_action_on_numbered_map} defines the group action $\Sf_E \curvearrowright \Mc_N(g, n)$. 
    The subgroup $\Sf_{E \setminus \{e_1\}} \leq \Sf_E$ provides the faithful subgroup action $\Sf_{E \setminus \{e_1\}} \curvearrowright \Mc_N(g, n)$. 
    The orbits of these two group actions are the algebraic maps and the rooted algebraic maps
    \begin{equation} \label{eq:unrooted_and_rooted_is_numbered_quotient}
        \Mc(g, n) = \Mc_N(g, n) / \Sf_E \ ,\ \Mc_R(g, n) = \Mc_N(g, n) / \Sf_{E \setminus \{e_1\}} \ .
    \end{equation}
    And due to the faithfulness of $\Sf_{E \setminus \{e_1\}} \curvearrowright \Mc_N(g, n)$,
    \begin{equation} \label{eq:numbered_to_rooted}
        \abs{\Mc_N(g, n)} = (n-1)! \abs{\Mc_R(g, n)} \ .
    \end{equation}

    Replacing the genus and edge number requirement with a passport $\pp$ yields similar formulas.
    \qed
\end{lemma}

%\bigskip

\begin{definition}
    The unnumbered version of a totally numbered map $M_N \in \Mc_N(\pp)$ is denoted by $M \in \Mc(\pp)$, which is the orbit of $M_N$ under the action of the group $\Sf_E$.
\end{definition}

\bigskip
%%%%% subsection %%%%%
\subsection{Enumeration between Unrooted Maps and Symmetric Maps} \label{sec:l_fold_map_enum}

Given an unlabeled passport $\pp$.
For $\phi \in \Sf_E$, let 
$$
    \fix(\pp; \phi) := \defsetsmall{M_N \in \Mc_N(\pp)}{\phi(M_N) = M_N}
$$
be the elements that are invariant under the action of $\phi$.
That is, when $M_N = (E, x, y)$, we have $x = x^\phi, y = y^\phi$.

Since an unrooted map is an orbit under the group action $\Mc(\pp) = \Mc_N(\pp) / \Sf_E$, by Burnside's Lemma in group actions
\begin{equation}
    \abs{\Mc(\pp)} = \sum_{\phi \in \Sf_E} \frac1{n!} \abs{\fix(\pp; \phi)}  \ .
\end{equation}
We only need to find $\abs{\fix(\pp; \phi)}$.

Now take $\phi$ such that $\fix(\pp; \phi) \neq \varnothing$.
Since there exist $x, y$ such that $G = \langle x, y \rangle$ acts transitively on $E$, and $x = x^\phi, y = y^\phi$, by a theorem in algebraic map theory, $\phi$ is a regular permutation, i.e., it can be decomposed into $m$ cycles of length $l$
$$
\phi = \alpha_1 \cdots \alpha_m \ ,\ \ell(\alpha_i) = l \ ,\ lm = n \ . 
$$
Clearly, given another permutation $\phi'$ of the same cycle type $(l^m)$, $\phi$ and $\phi'$ are conjugate, hence there is a bijection $\fix(\pp; \phi) \cong \fix(\pp; \phi')$.
Thus we only need to pick one $\phi_{l, m}$ for each cycle type $(l^m)$ and compute $\abs{\fix(\pp; \phi_{l, m})}$.

Since the number of permutations within cycle type $(l^m)$ is $n! / (m! \ l^m)$, we have
\begin{lemma} \label{lem:unrooted_from_fix_enum}
    Given an unlabeled passport $\pp$, the number of  unrooted maps $\Mc(\pp)$ can be obtained by the following formula
    \begin{equation} \label{eq:unrooted_from_sum_fix}
    \abs{\Mc(\pp)} = \sum_{lm = n} \frac1{m! \ l^m} \abs{\fix(\pp; \phi_{l, m})} \ . 
    \end{equation}

    Here the cycle type of $\phi_{l, m}$ is $(l^m)$.
    \qed 
\end{lemma}

\bigskip

Finally, the enumeration of unrooted maps $\Mc(\pp)$ is reduced to the enumeration of symmetric maps in $\fix(\pp; \phi_{l, m})$.
The maps $M_N \in \fix(\pp; \phi_{l, m})$ are all totally numbered.
Lemma \ref{lem:group_action_on_numbered_map} indicates that the counting of totally numbered maps is related to the enumeration of rooted maps, this enables us to obtain the enumeration of unrooted maps via the simpler enumeration of rooted maps.

Moreover, the unlabeled version $M$ of the map $M_N$ contains an automorphism $\phi_{l, m} \in \Aut(M)$.
This automorphism generates a cyclic group $\langle \phi_{l, m} \rangle \cong \ZZ_l \leq \Aut(M)$.
The map $M$ admits a regular cyclic cover $N := M / \langle \phi_{l, m} \rangle$ under this subgroup.
This suggests that we can indirectly obtain the enumeration of $\fix(\pp; \phi_{l, m})$ by studying the enumeration of the quotient map $N$.

\bigskip
%%%%% section %%%%%
\section{Map Regular Covers and Maps on Orbifolds} \label{ch:map_cover_and_orbifold_map}

The goal of this section is to state and prove Lemma \ref{lem:symmetric_map_to_orbifold_map}.
It shows that the symmetric map in Lemma \ref{lem:unrooted_from_fix_enum} can be transformed into a quotient map on orbifold.

In Section \ref{sec:map_cover}, we describe the definition of regular cover of a map, and the related definitions are parallel to those of algebraic maps.

Taking the quotient of the underlying surface of a map by a subgroup of the automorphism group, the quotient surface is generally an orbifold, and thus the corresponding quotient map is naturally on orbifold.
The definition of a map on orbifold and its passport are described in Section \ref{sec:sigma_and_orbifold_map}.

For symmetric map in Lemma \ref{lem:unrooted_from_fix_enum}, it has a cyclic group as a subgroup of the automorphism group, so a quotient map under a regular cyclic cover can be defined.
Section \ref{sec:regular_cyclic_cover} explores the relationship between the orbifold symbol of the corresponding orbifold and the passport of the map under a regular cyclic cover (Lemma \ref{lem:passport_multi_when_cyclic_regular_cover}). 
We finally obtain the transformation from a symmetric map to a quotient map (Lemma \ref{lem:symmetric_map_to_orbifold_map}).

\bigskip
%%%%% subsection %%%%%
\subsection{Map Regular Covers} \label{sec:map_cover}

\begin{definition}
    Given two (unlabeled) bicolored maps $M_i = (S_{g_i}; E_i, U_{i, k}), i = 1, 2$.
    Assume they are given by dessins d'enfants in Proposition \ref{prop:dessin}, so that $S_{g_i}, i=1,2$ are compact Riemann surfaces of genus $g_i$.
    
    A \textbf{map cover} $\phi : M_1 \to M_2$ is defined as a surjective holomorphic mapping $\phi : S_{g_1} \to S_{g_2}$, inducing a surjective correspondence of map elements $\phi(E_1) = E_2, \phi(U_{1, k}) = U_{2, k}, k=1,2,3$.
\end{definition}

%\bigskip

The map cover defined above is the map version of that of algebraic map defined previously.
Moreover, if there exist inverse map covers $\phi : M_1 \to M_2, \phi^{-1} : M_2 \to M_1$, then $\phi$ is a map isomorphism as defined previously.

\bigskip

The \textbf{covering transformation group} of $\phi : M_1 \to M_2$ is the set of automorphisms of $M_1$ preserving $\phi$, $\Aut(\phi) := \defsetsmall{I \in \Aut(M_1)}{\phi = \phi \circ I}$.
$\phi$ is called a \textbf{regular cover} if the action of the covering transformation group $\Aut(\phi)$ on the fiber $\phi^{-1}(e)$ of each edge $e \in E_2$ is transitive.
Given a subgroup $A \leq \Aut(M_1) \leq \Aut(S_{g_1})$ of the automorphism group of a map $M_1$, one can define a regular cover $\phi_A : M_1 \to M_1 / A$, where the holomorphic mapping is defined as $\phi_A : S_{g_1} \to S_{g_1} / A$. Naturally, $A = \Aut(\phi_A)$ is the covering transformation group.

Now fix a regular cover $\phi : M_1 \to M_2 := M_1 / A$. 
Since a non-constant holomorphic mapping is a branched cover, we can prove
\begin{enumerate}
    \item On edges, $\phi$ is unramified.
    That is, $E_2 = E_1 / A$ is the set of orbits of a faithful action.
    \item On black vertices, white vertices, and faces, $\phi$ may be ramified, and there is at most one ramification point in each face.
    Thus, on the set of orbits $U_{2, k} = U_{1, k} / A$, one can define the \textbf{cone point order function}
    \begin{equation}
        \nu_{M_2} : U_{2, 1} \sqcup U_{2, 2} \sqcup U_{2, 3} \to \Zpos \ ,\ u \mapsto |A|\ \big/\ \abs{\phi^{-1}(u)} \ .
    \end{equation}
    This function is also the common multiplicity of the branch points $\td{u} \in \phi^{-1}(u)$ of the holomorphic mapping $\phi$.
    Moreover, it is easy to prove that $\nu_{M_2}(u) = \wt_{M_1}(\td{u}) \ /\ \wt_{M_2}(u), \td{u} \in \phi^{-1}(u)$, determined by the properties of the map.
\end{enumerate}

%\bigskip

If we collect the cone point orders $\nu_{M_2}$ at vertices and faces into an unordered sequence $\range \nu_{M_2}$, and add the surface genus $g_2$, we obtain an \textbf{orbifold symbol} $\sigma = (g_2; \range \nu_{M_2})$.
Then $\phi : S_g \to O(\sigma) = S_g / A$ is a regular orbifold cover.
Moreover, the map $M_2$ can be regarded as a map on the orbifold $O(\sigma)$.

Thus, a regular map cover $\phi : M_1 \to M_2$ can be redefined as
\begin{definition} \label{def:map_regular_cover}
    A \textbf{regular map cover} $\phi : M_1 \to M_2$ from a map $M_1$ on $S_{g_1}$ to a map $M_2$ on $O(\sigma)$ is defined as a regular orbifold cover $\phi : S_{g_1} \to O(\sigma)$, together with a surjective correspondence of map elements $\phi(E_1) = E_2, \phi(U_{1, k}) = U_{2, k}, k=1,2,3$,
    such that the cone point orders of $M_2$ defined by $\nu_{M_2}(u) = \wt_{M_1}(\td{u}) \ /\ \wt_{M_2}(u), \td{u} \in \phi^{-1}(u)$, are exactly the cone point orders of  orbifold $O(\sigma)$.
\end{definition}

The definitions of the orbifold $O(\sigma)$ and maps on orbifold are given in the next subsection.

\bigskip
%%%%% subsection %%%%%
\subsection{Orbifold Symbols and Maps on Orbifolds} \label{sec:sigma_and_orbifold_map}

Recall that an \textbf{orbifold symbol} $\sigma = (g; \{t_1, \ldots, t_r\})$ consists of a genus $g$ and an unordered sequence of cone point orders. 
The \textbf{orbifold} $O(\sigma)$ can be constructed as follows: start with a compact oriented surface $S_g$ of genus $g$, take $r$ points $R = \{p_1, \ldots, p_r\}$ on it; endow $S_g$ with a cone point order function $\nu_O : S_g \to \Zpos$, which takes possibly non-one values only at the $r$ orbifold points $R$, with $\nu_O(p_i) = t_i$, and at all other points $\nu_O(S_g \setminus R) = \{1\}$.
The orbifold $(S_g, \nu_O)$ defined in this way is unique up to diffeomorphism. 

\bigskip

Next, we define a map on orbifold. It is similar to a map on a surface.

\begin{definition}
    Given an orbifold symbol $\sigma$ and an orbifold $O(\sigma) = (S_g, \nu_O)$.
    Similar to a map on surface, an (unlabeled) \textbf{maps on orbifold} consists of an abstract graph $G = (V, E)$ and an embedding $i : G \to S_g$. 
    There are no cone points on edges, cone points may lie on vertices and in faces, and there is at most one cone point in each face.
\end{definition}

%\bigskip

A map on orbifold is denoted by $M = (O(\sigma); E, U_k) = (S_g, \nu_O; E, U_k)$.
A regular cover of maps on orbifolds is given by Definition \ref{def:map_regular_cover}.
Isomorphisms of maps on orbifolds can also be defined similarly to those of maps on surfaces.

The orbifold cone point order function $\nu_O$ can be restricted to the vertex set $\nu_M : V \to \Zpos$.
For a face $f$, if it contains a cone point of order $t$, then its cone point order function $\nu_M(f) = t$; otherwise define $\nu_M(f) = 1$.
This defines the \textbf{cone point order function at vertices and faces} $\nu_M : V \sqcup U_3 \to \Zpos$.
Conversely, given only the cone point order function $\nu_M$ at vertices and faces,
one can also determine the orbifold symbol $\sigma = (g; \range \nu_M)$, where $\range \nu_M$ denotes the unordered sequence formed by the image of $\nu_M$.
Thus a bicolored maps on orbifold can also be denoted by $M = (S_g; E, U_k; \nu_M)$.

\bigskip

In a map on orbifold, if the weight or cone point order differs, the vertices are distinct.
Thus, the weight and the cone point order together play the role of the labeling function.
Therefore, the passport of a map on orbifold can be defined as a \textbf{passport labeled by an orbifold symbol}.

Recall that a passport $\ppsi = (g, n; \Pi_1, \Pi_2, \Pi_3)$ labeled by an orbifold symbol $\sigma = (g; \bmt)$ has the same genus $g$.
And it is required that the unordered sequence of cone point orders $\bmt$ can be partitioned into unordered sequences, as in \eqref{eq:sigma_split}
\begin{equation*}
    \bmt = \bigsqcup_{i=1}^3 \bigsqcup_{j=1}^{N_i} \left\{ \ t_{ij}^{\lambda_{ij}} \ \right\} \ ,
\end{equation*}
and the weight distribution $\Pi_i = (S_i, \lambda_i, \wt_i)$ should be writable in the form of \eqref{eq:sigma_labeled_passport} as
\begin{equation*}
    S_i = \defsetsmall{(w_{ij}, t_{ij})}{i=1,2,3\ , \ 1 \leq j \leq T_i} \ ,\ 
    \lambda_i(w_{ij}, t_{ij}) = \lambda_{ij} \ ,\ 
    \wt_i(w_{ij}, t_{ij}) = w_{ij} \ .
\end{equation*}

%\bigskip

We use the following notation for passports labeled by orbifold symbols.

\begin{notation} \label{notat:sigma_labeled_passport}
Assume a passport labeled by an orbifold symbol having the form in \eqref{eq:sigma_split} and \eqref{eq:sigma_labeled_passport}. 
Record the cone point order $t_{ij}$ in its weight distribution as a left subscript of the weight $w_{ij}$, and the multiplicity $\lambda_{ij}$ as a right superscript, i.e.,
\begin{equation}
    \Pi_i = \prod_{j = 1}^{N_i} {_{t_{ij}} (w_{ij})^{\lambda_{ij}}} \ .
\end{equation}

For example, if $\sigma = (0; 2^4) = (0; 2^4\ 1^2)$, then $\ppsi = (0, 4; {_2 4} \ ,\ {_2 2^2} \ , \ {_2 2} \ {_1 1^2})$ is a passport labeled by an orbifold symbol.
\end{notation}

\bigskip

\begin{definition} \label{def:orbifold_map_passport}
    Given an orbifold symbol $\sigma = (g; \bmt)$ and a passport $\ppsi = (g, n; \Pi_1, \Pi_2, \Pi_3)$ labeled by $\sigma$,
    and assume that the weight distribution $\Pi_i$ has the form of \eqref{eq:sigma_labeled_passport}.
    
    An maps on orbifold $M = (S_g; E, U_k; \nu_M)$ corresponding to the passport $\ppsi$ should satisfy that the genus $g$ and the number of edges is the same; 
    and there are exactly $\lambda_{1j}$ black vertices with weight $w_{1j}$ and cone point order $t_{1j}$; for white vertices and faces, the requirement is the same.
\end{definition}

%\bigskip

Note that if a passport $\ppsi$ labeled by an orbifold symbol is equivalent to a general passport $\pp$, then also $\Mc(\ppsi) \cong \Mc(\pp)$.
In this sense, these two types of passports are indistinguishable.
That is, a map on orbifold can be viewed as a labeled map.

\bigskip
%%%%% subsection %%%%%
\subsection{Regular Cyclic Covers of Maps} \label{sec:regular_cyclic_cover}

In Section \ref{sec:map_cover}, we already know that a regular cover of a map $\phi : M \to N$ can be defined by a subgroup $A \leq \Aut(M)$ of map automorphisms, so that $N = M / A$.
Moreover, the regular cover gives the cone point orders $\nu_{N}$ at the vertices and faces of the map $N$.

Now we require that the subgroup $A \cong \ZZ_l$ is a cyclic group, and the corresponding cover is a regular cyclic cover. Also, given an unlabeled passport $\pp = (g_1, n; \pi_1, \pi_2, \pi_3)$, let $M \in \Mc(\pp)$.
Let $N = M / \ZZ_l$ be a map embedded in the orbifold $O(\sigma = (g_2; \bmt))$, with passport $\ppsi = (g_2, m; \Pi_1, \Pi_2, \Pi_3)$.
We wish to find the relationship between $g_1$ and $\sigma$, and between $\pp$ and $\ppsi$.
We already have the relation for the number of edges $n = lm$.

Since the regular cover of a map is a regular cover of its underlying surface, $\phi : S_{g_1} \to O(\sigma)$ is also a regular cyclic cover $O(\sigma) = S_{g_1} / \ZZ_l$.
Recall that this condition is $g = \sigma \times l$.
The following fact explains the conditions $\sigma$ should satisfy.

\begin{theorem}[Harvey \cite{Har66}] \label{thm:g_l_sigma_admissible}
    Given a genus $g_1$, a positive integer $l \in \Zpos$, and an orbifold symbol $\sigma = (g_2; \{t_1, \ldots, t_r\})$.
    Then the orbifold $O(\sigma) = S_{g_1} / \ZZ_l$ is a regular cyclic cover of $S_{g_1}$ of order $l$ if and only if all the following conditions hold
    \begin{enumerate}
    \item Let $T = \lcm(t_1, \ldots, t_r)$, then $T \ \big|\ l$, and if $g = 0$ then we must have $T = l$;
    \item Riemann-Hurwitz formula:
    $2 - 2g_1 = l \left(2 - 2g_2 - \sum_{i=1}^r \left(1 - \frac{1}{t_i}\right)\right)$;
    \item Let $T_i := \lcm(t_1, \cdots, t_{i-1}, t_{i+1}, \cdots, t_r)$; then for each $i$, $T_i = T$;
    \item If $T$ is even, let $2^e \ \big\|\ T$ be the highest power of $2$ dividing $T$; then the number of $t_i$ divisible by $2^e$ is even.
    \qed
    \end{enumerate}
\end{theorem}

On the other hand, the relationship between $\pp$ and $\ppsi$ is also easy to obtain.
Assume the passport $\ppsi$ has the form of \eqref{eq:sigma_labeled_passport}.
Since the cone point order function of map $N$ satisfies $\nu_N(u) = \wt_M(\td{u}) / \wt(u), \td{u} \in \phi^{-1}(u)$, and $\abs{\phi^{-1}(u)} = l\ /\ \nu_N(u)$,
whenever $N$ has $\lambda_{ij}$ vertices or faces with weight $w_{ij}$ and cone point order $t_{ij}$, $M$ should have $\lambda_{ij} \cdot l\ /\ t_{ij}$ vertices or faces with weight $w_{ij} \cdot t_{ij}$.
Thus the weight distribution of $\pp$ should have the form
\begin{equation*}
    \pi_i := \bigsqcup_{j=1}^{N_i} \left\{ (w_{ij} \cdot t_{ij})^{\lambda_{ij} \cdot l / t_{ij}} \right\} \quad,\quad i=1,2,3 \ .
\end{equation*}
Recall that this condition is precisely the definition of \textbf{multiplication of an orbifold symbol and a passport} in \eqref{eq:sigma_pp_multi}.
Thus the relationship between $\pp$ and $\ppsi$ is $\pp = \sigma \times_l \ppsi$.

\begin{lemma} \label{lem:passport_multi_when_cyclic_regular_cover}
    Given an unlabeled passport $\pp = (g_1, n; \pi_1, \pi_2, \pi_3)$ and a map $M = \Mc(\pp)$.
    If the map $N = M / \ZZ_l$ is the quotient map of $M$ under a regular cyclic cover, and $N \in \Mc(\ppsi)$ is a map on the orbifold $O(\sigma)$,
    then the orbifold symbol $\sigma$ satisfies the condition $g = \sigma \times l$,
    and the passports satisfy the relation $\pp = \sigma \times_l \ppsi$.
    \qed
\end{lemma}

%\bigskip

\begin{lemma} \label{lem:symmetric_map_to_orbifold_map}
    Given an unlabeled passport $\pp = (g, n; \pi_1, \pi_2, \pi_3)$, positive integers $l, m$ satisfying $lm = n$, let $\phi_{l, m}$ be a permutation with cycle type $(l^m)$.
    
    Take any totally numbered map $M_N \in \fix(\pp; \phi_{l, m})$, i.e., the conjugation action of $\phi_{l, m}$ leaves the map unchanged $\phi_{l, m}(M_N) = M_N$.
    One can always obtain an orbifold symbol $\sigma$, a passport $\ppsi$, and a totally numbered map $N_N \in \Mc_N(\ppsi)$, such that $g = \sigma \times l$, $\pp = \sigma \times_l \ppsi$.
    After removing the numbering, $N = M / \langle \phi_{l, m} \rangle$ is actually the quotient map. 
\end{lemma}

\begin{proof}
According to Lemma \ref{lem:passport_multi_when_cyclic_regular_cover} and the fact that $\langle \phi_{l, m} \rangle \cong \ZZ_l \leq \Aut(M)$ is a cyclic subgroup of the automorphism group of $M$, we can directly provide the orbifold symbol $\sigma$, the passport $\ppsi$, and the quotient map $N = M / \langle \phi_{l, m} \rangle$.
It only remains to give $N$ a total numbering.

We may assume that the edge set $E$ of $M_N = (E, x, y)$ and $\phi_{l, m}$ have the form
\begin{equation}
\begin{split}
    E &= \defsetsmall{e_{i, a}}{1 \leq i \leq m, a \in \ZZ_l} \ , \\ 
    \phi_{l, m} &= (e_{1, 0}, e_{1, 1}, \ldots, e_{1, l-1}) \cdots (e_{m, 0}, e_{m, 1}, \ldots, e_{m, l-1}) \ .
\end{split} 
\end{equation}
Here we denote $\ZZ_l = \{0, 1, \ldots, l-1\}$.
Then the orbits in the edge set $\ol{E} = E / \langle \phi_{l, m} \rangle$ of map $N$ are of the form $\ole_i = \{\ole_{i, 0}, \ole_{1, 1}, \ldots, \ole_{i, l-1}\}$.
This naturally gives each edge $\ole_i$ in $\ol{E}$ the label $i$.
Finally, the permutations $\ol{x}, \ol{y}$ in $N_N = (\ol{E}, \ol{x}, \ol{y})$ can be obtained via the induced homomorphism between the cartographic groups.
This completes the determination of the totally numbered map $N_N$.
\end{proof}

\bigskip
%%%%% section %%%%%
\section{Enumeration of Unrooted Maps} \label{ch:unrooted_map_enum}

This section will combine the previous two sections to finally prove the main theorems \ref{thm:CFF_general_passport_formula} and \ref{thm:unrooted_QU}.

Lemma \ref{lem:unrooted_from_fix_enum} shows that the enumeration of unrooted maps can be transformed into the enumeration of a class of symmetric maps.
Lemma \ref{lem:symmetric_map_to_orbifold_map} in Section \ref{ch:map_cover_and_orbifold_map} shows that this class of symmetric maps is related to maps on orbifolds.
This section provides Lemma \ref{lem:orbifold_map_to_symmetric_map}, showing that given a map on an orbifold and some supplementary data, one can conversely determine the symmetric map.
Finally, we obtain the conversion from the enumeration of symmetric maps to the enumeration of maps on orbifolds (Theorem \ref{thm:fix_enum}).
This proves that the enumeration of unrooted maps can be converted into the enumeration of rooted maps on orbifolds, and the latter can be fully obtained by the main theorem \ref{thm:CFF_general_passport_formula}.

Section \ref{sec:map_universal_cover} of this section first provides preliminary knowledge on universal covers of maps.
Then we prove Lemma \ref{lem:orbifold_map_to_symmetric_map} and Theorem \ref{thm:fix_enum} in order.

%\bigskip
%%%%% subsection %%%%%
\subsection{Universal Cover of Maps} \label{sec:map_universal_cover}

First, we state the following fact: almost all orbifolds have a universal cover.
In particular, the orbifolds we study are of the form $O(\sigma) = S_g / A$ as a quotient space, and $S_g$ has a regular cover, hence $O(\sigma)$ also has a regular cover.

\begin{proposition}[Koebe \cite{Won71}] \label{prop:orbifold_universal_cover}
    Any 2-dimensional orbifold $O(\sigma)$ has a universal covering space $X$, except for $O(0; t)$ and $O(0; t_1, t_2), t_1 \neq t_2$.
    Moreover, the sign of the \textbf{Euler characteristic} of the orbifold signature
    $$
    \chi(\sigma) := 2 - 2g - \sum_{i=1}^r \left( 1 - \frac1{t_i} \right)
    $$
    determines the form of its universal covering space $X$:
    if $\chi(\sigma) > 0$ then $X$ is sphere, otherwise $X$ is plane.
    \qed
\end{proposition}

The fundamental group $\pi_1(O(\sigma))$ of the orbifold can be defined as the group of deck transformations of the universal cover $X \to O(\sigma)$.
It is known that the group presentation of $\pi_1(O(\sigma))$ is the same as that of the Fuchsian group $\Gamma(\sigma)$, i.e., $\pi_1(O(\sigma)) \cong \Gamma(\sigma)$, where
\begin{equation}
    \Gamma(\sigma) = \left\langle
    \{a_i, b_i\}_{i=1}^g, \{x_i\}_{i=1}^r \ \left|\ 
    \prod_{i=1}^g [a_i, b_i] \prod_{i=1}^r x_i = x_1^{t_1} = \cdots = x_r^{t_r} = \one \right.\right\rangle \ .
\end{equation}

%\bigskip
%\bigskip

First, given an unlabeled map $M \in \Mc(\pp)$, $\pp = (g_1, n; \pi_1, \pi_2, \pi_3)$.
Also given a subgroup $A \leq \Aut(M)$ of the automorphism group and a regular cover $\psi : M \to N := M / A$, this induces the cone point orders $\nu_N$ of the map $N$, and yields the orbifold signature $\sigma := (g_2; \range \nu_N)$ of the orbifold on which $N$ lies.

By Proposition \ref{prop:orbifold_universal_cover}, we can provide the universal covering space $X$ common to the orbifold $O(\sigma)$ and $S_g$,
as the left of the commutative diagram \eqref{eq:cd_universal_cover}.

\begin{equation} \label{eq:cd_universal_cover}
\begin{tikzcd}
    S_g \arrow[d, "\psi"'] & X \arrow[l, "\zeta"'] \arrow[dl, "\eta"] \\
    O(\sigma) &
\end{tikzcd}
\qquad\qquad
\begin{tikzcd}
    M \arrow[d, "\psi"'] & \td{N} \arrow[l, "\zeta"'] \arrow[dl, "\eta"] \\
    N &
\end{tikzcd}
\end{equation}

The regular cover $\psi : S_g \to O(\sigma) = S_g / A$ gives
a normal group embedding $\psi_* : \pi_1(S_g) \hookrightarrow \pi_1(O(\sigma))$ with $\psi_* \pi_1(S_g) \trianglelefteq \pi_1(O(\sigma))$, and
$A \cong \pi_1(O(\sigma)) \ /\ \psi_* \pi_1(S_g)$, i.e., $A \cong \Gamma(\sigma) \ /\ \pi_1(S_g)$.
Equivalently, this yields an \textbf{order-preserving} epimorphism $\zeta_A : \Gamma(\sigma) \to A$, where order-preserving means mapping finite-order elements to elements of the same order, i.e., requiring that $\ker \zeta_A$ is torsion-free.
Here $\ker \zeta_A = \pi_1(S_g)$ is torsion-free. 
Thus $\zeta_A$ is order-preserving.

Let $\Epi_0(G, H)$ denote the set of all order-preserving epimorphisms from group $G$ to group $H$. In the preceding discussion we obtained $\zeta_A \in \Epi_0(\Gamma(\sigma), A)$.

\bigskip

We pull back the map $N$ to the universal cover $X$, and the resulting map $\td{N} := \eta^{-1}(N)$ is called the \textbf{universal covering map} of map $N$.
Since $\eta = \psi \circ \zeta$, the universal covering map of $M$ is also $\td{N} = \zeta^{-1}(M)$, as the right of the commutative diagram \eqref{eq:cd_universal_cover}.
Thus $N = \td{N} / \Gamma(\sigma)$ and $M = \td{N} / \pi_1(S_g)$.

Here the order-preserving epimorphism $\zeta_A : \Gamma(\sigma) \to A$ yields the action of $A \leq \Aut(M)$ on the edge set $E$.
That is, for $e \in E$ and $\tde \in \zeta^{-1}(e)$, and any $b \in \Gamma(\sigma)$,
$\zeta(b \cdot \tde) = \zeta_A(b) \cdot e$.

\begin{proposition}[Mednykh-Nedela \cite{MedNed06}] \label{prop:epi_0_zetaA}
    Given an unlabeled map $M \in \Mc(\pp)$ and a subgroup $A \leq \Aut(M)$ of the automorphism group, one can define the regular covering map $N = M / A$ on orbifold $O(\sigma)$, and give the cone point orders $\nu_N$ at vertices and faces.
    
    Moreover, via the universal cover $\eta : X \to O(\sigma)$ of the orbifold $O(\sigma)$, one can define the universal covering map $\td{N} := \eta^{-1}(N)$ of the map $N$.
    Furthermore, $\td{N}$ is also the universal covering map of $M$ via $\zeta : \td{N} \to M$.
    There is also an order-preserving epimorphism $\zeta_A \in \Epi_0(\Gamma(\sigma), A)$.
    For $b \in \Gamma(\sigma)$ and an edge $\tde$ of the universal covering map $\td{N}$, $\zeta(b \cdot \tde) = \zeta_A(b) \cdot \zeta(\tde)$.
    \qed
\end{proposition}

\bigskip
%%%%% subsection %%%%%
\subsection{From maps on orbifolds to symmetric maps} \label{sec:from_orbicyclic_map_to_symmetric_map}

In this section we prove the following lemma that a map on orbifold $N$ and additional data yields a corresponding symmetric map $M$.
The proof is mainly algebraic, but its motivation comes from the following geometric intuition.

The orbifold $O(\sigma)$ embedding map $N$ can be lifted to the universal covering map $\td{N}$ on the universal cover $X$.
Conversely, taking the fundamental domain $\ol{D}$ of $O(\sigma)$ in $X$, the restriction of $\td{N}$ to $\ol{D}$ gives the map $N$.
If we take a larger fundamental domain $D$, which is the fundamental domain of $S_g$ in $X$, the restriction of $\td{N}$ to $D$ gives the map $M$.
Thus the additional data mainly serve to determine the shape of the fundamental domain $D \supset \ol{D}$ for $M$ and the numbering of the edges of $M$.

\begin{lemma} \label{lem:orbifold_map_to_symmetric_map}
Given an unlabeled passport $\pp = (g_1, n; \pi_1, \pi_2, \pi_3)$, positive integers $l, m$ satisfying $lm = n$, and a permutation $\phi_{l, m}$ of cycle type $(l^m)$.
Then every choice of an orbifold symbol $\sigma$ satisfying $g = \sigma \times l$, a passport $\ppsi$ satisfying $\pp = \sigma \times_l \ppsi$, a totally numbered map on orbifold $N_N \in \Mc_N(\ppsi)$, an order-preserving epimorphism $\zeta_A \in \Epi_0(\Gamma(\sigma), \ZZ_l)$, and $(m-1)$ cyclic group elements $(a(i))_{i=2}^m \in \ZZ_l^{m-1}$, determines a totally numbered map $M_N \in \fix(\pp; \phi_{l, m})$, satisfying their unnumbered versions $M / \ZZ_l = N$.

In the language of mappings, there exist
\begin{equation} \label{eq:to_fix}
    \bigsqcup_{\sigma} \bigsqcup_{\pp = \sigma \times_l \ppsi} \Mc_N(\ppsi) \times \Epi_0(\Gamma(\sigma), \ZZ_l) \times \ZZ_l^{m-1}
    \quad \to \quad
    \fix(\pp; \phi_{l, m}) \ . 
\end{equation}
\end{lemma}

\begin{proof}
According to the conditions of the lemma, we are given $\sigma, \ppsi, N_N \in \Mc_N(\ppsi), \zeta_A \in \Epi_0(\Gamma(\sigma), \ZZ_l)$ and $(a(i))_{i=2}^m \in \ZZ_l^{m-1}$.
We now need to obtain $M_N \in \fix(\pp; \phi_{l, m})$ such that the unnumbered version satisfies $M / \ZZ_l = N$.
If we write $M_N = (E, x, y)$ as a totally numbered algebraic map, we ultimately need to determine $x, y$.
According to the discussion preceding Proposition \ref{prop:epi_0_zetaA}, we can obtain the universal cover $\td{N}$ of $N$, which is also the universal cover of $M$, specifically described by the commutative diagram \eqref{eq:cd_universal_cover}
\begin{equation*}
    \psi : M \to N \ ,\ \zeta : \td{N} \to M \ ,\ \eta : \td{N} \to N \ ,\ \eta = \psi \circ \zeta \ .
\end{equation*}

We now consider these three maps $N, M, \td{N}$ from algebraic maps perspective.

Let $N_N = (\ol{E}, \ol{x}, \ol{y})$, where $\ol{E} = \{\ole_1, \ldots, \ole_m\}$.
Since the pullback map $\td{N} = \eta^{-1}(N)$ is known, we can write $\td{N} = (\td{E}, \td{x}, \td{y})$, where $\td{x}, \td{y}$ are not unique but chosen from the equivalence classes.
Because the $\ole_i$ are pairwise distinct, the fibres $\eta^{-1}(\ole_i)$ are identifiable, and $\Gamma(\sigma)$ acts regularly on the fibres, yielding a $\Gamma(\sigma)$-equivariant encoding on $\td{E}$:
\begin{equation*}
    \td{E} = \defsetsmall{\tde_{i, b}}{1 \leq i \leq m, b \in \Gamma(\sigma)} \ ,\ 
    \eta^{-1}(\ole_i) = \{\tde_{i, b}\}_{b \in \Gamma(\sigma)} \ ,\ 
    b' \cdot \tde_{i, b} = \tde_{i, b' \cdot b} \ .
\end{equation*}

Similarly, we may assume that the edge set of $M = (E, x, y)$ has the form
\begin{equation*}
    E = \defsetsmall{e_{i, a}}{1 \leq i \leq m, a \in \ZZ_l} \ ,\ 
    \psi^{-1}(\ole_i) = \{e_{i, a}\}_{a \in \ZZ_l} \ ,\ 
    a' \cdot e_{i, a} = e_{i, a' \cdot a} \ .
\end{equation*}
We always assume $\phi_{l, m} = (e_{1, 0}, e_{1, 1}, \ldots, e_{1, l-1}) \cdots (e_{m, 0}, e_{m, 1}, \ldots, e_{m, l-1})$. 

\bigskip

Furthermore, we can choose the base point of $O(\sigma)$ on the edge $\ole_1$ and require that the base point of its regular cover $S_g$ lies on the edge $e_{1, \id}$, and the base point of the universal cover $X$ lies on the edge $\tde_{1, \id}$.
This effectively fixes a root edge $e_{1, \id}$ for $\td{N}$.
We now know that $\zeta(\tde_{1, \id}) = e_{1, \id}$.
However, for other indices $i$, it is not generally true that $\zeta(\tde_{i, \id}) = e_{i, \id}$.
Thus we set
\begin{equation*}
    \zeta(\tde_{i, \id}) = e_{i, a(i)} \ ,\ a(1) = \id \ .
\end{equation*}
Here for $2 \leq i \leq m$, $a(i)$ is the sequence of elements chosen in the lemma condition.

\bigskip

By Proposition \ref{prop:epi_0_zetaA}, for any $b \in \Gamma(\sigma)$ we have
\begin{equation*}
      \zeta(\tde_{i, b})
    = \zeta(b \cdot \tde_{i, \id})
    = \zeta_A(b) \cdot e_{i, a(i)}
    = e_{i, \zeta_A(b) \cdot a(i)} \ .
\end{equation*}

Write $\td{x} \cdot \tde_{i, \id} = \tde_{\ol{x}(i), b_1(i)}$, where $b_1(i)$ is a function determined by $\td{x}$.
Similarly, write $\td{y} \cdot \tde_{i, \id} = \tde_{\ol{y}(i), b_2(i)}$.
Here $b_k(i), k = 1,2$ are determined by $\td{x}, \td{y}$.

Thus, for any $1 \leq i \leq m$ and $a \in \ZZ_l$, we have
\begin{equation*}\begin{split}
      x \cdot e_{i, a}
    &= x (a a(i)^{-1}) \cdot e_{i, a(i)}
    = (a a(i)^{-1}) x \cdot e_{i, a(i)} \\
    &= (a a(i)^{-1}) \cdot \zeta\left(\td{x} \cdot \tde_{i, \id}\right) \\
    &= a a(i)^{-1} \cdot \zeta(\tde_{\ol{x}(i), b_1(i)}) \\
    &= a a(i)^{-1} \zeta_A(b_1(i)) \cdot \zeta(\tde_{\ol{x}(i), \id}) \\
    &= a a(i)^{-1} \zeta_A(b_1(i)) \cdot e_{\ol{x}(i), a(\ol{x}(i))} \\
    &= e_{\ol{x}(i), a a(i)^{-1} \zeta_A(b_1(i)) a(\ol{x}(i))} \ .
\end{split}\end{equation*}

Similarly, on the other side we have
$$
y \cdot e_{i, a} = e_{\ol{y}(i), a a(i)^{-1} \zeta_A(b_2(i)) a(\ol{y}(i))} \ .
$$

This determines $x$ and $y$. Thus we completely obtain the totally numbered map $M_N = (E, x, y) \in \fix(\pp; \phi_{l, m})$.
\end{proof}

\bigskip
%%%%% subsection %%%%%
\subsection{Bijective between Maps on Orbifolds and Symmetric Maps} \label{sec:orbifold_map_cong_symmetric_map}

In this section we combine Lemma \ref{lem:symmetric_map_to_orbifold_map} and Lemma \ref{lem:orbifold_map_to_symmetric_map} to establish a bijective correspondence between symmetric maps and orbifold embedding maps.
And finally we can enumerate the symmetric maps.

\begin{theorem} \label{thm:fix_enum}
    Given an unlabeled passport $\pp = (g_1, n; \pi_1, \pi_2, \pi_3)$, $l$ such that $lm = n$, and a permutation $\phi_{l, m}$ of cycle type $(l^m)$, one can prove that the mapping \eqref{eq:to_fix} in Lemma \ref{lem:orbifold_map_to_symmetric_map} is a bijection
    \begin{equation} \label{eq:bij_fix}
        \bigsqcup_{\sigma} \bigsqcup_{\pp = \sigma \times_l \ppsi} \Mc_N(\ppsi) \times \Epi_0(\Gamma(\sigma), \ZZ_l) \times \ZZ_l^{m-1}
    \quad \cong \quad
    \fix(\pp; \phi_{l, m}) \ . 
    \end{equation}
    Consequently we obtain the counting formula for $\fix(\pp; \phi)$:
    \begin{equation}
        \abs{\fix(\pp; \phi)} = \sum_{g = \sigma \times l} \sum_{\pp = \sigma \times_l \ppsi} l^{m-1} \cdot \abs{\Mc_N(\ppsi)} \cdot \abs{\Epi_0(\Gamma(\sigma), \ZZ_l)} \ .
    \end{equation}
\end{theorem}

\begin{proof}
We adopt the notation from Lemma \ref{lem:orbifold_map_to_symmetric_map}.
We first prove that the map is injective.

Suppose, for contradiction, that at least one of the parameters $\sigma, \ppsi, N_N, \zeta_A, \\ (a(i))_{i=2}^m$ differs but yields the same $M_N \in \fix(\pp; \phi_{l, m})$.
By Lemma \ref{lem:symmetric_map_to_orbifold_map} and Proposition \ref{prop:epi_0_zetaA}, the parameters $\sigma, \ppsi, N_N, \zeta_A$ have already been determined by $M_N$, so only two different sets of group elements $(a(i))_{i=2}^m, (a'(i))_{i=2}^m$ could possibly differ.

From the proof of Lemma \ref{lem:orbifold_map_to_symmetric_map}, the expressions for $x, y$ are as follows:
$$
x \cdot e_{i, a} = e_{\ol{x}(i), a a(i)^{-1} \zeta_A(b_2(i)) a(\ol{x}(i))} \ ,\ 
y \cdot e_{i, a} = e_{\ol{y}(i), a a(i)^{-1} \zeta_A(b_2(i)) a(\ol{y}(i))} \ .
$$
Hence
\begin{equation*}\begin{split}
    &a a(i)^{-1} \zeta_A(b_1(i)) a(\ol{x}(i)) = a {a'(i)}^{-1} \zeta_A(b_1(i)) a'(\ol{x}(i)) \ , \\
    &a a(i)^{-1} \zeta_A(b_2(i)) a(\ol{y}(i)) = a {a'(i)}^{-1} \zeta_A(b_2(i)) a'(\ol{y}(i)) \ .
\end{split}\end{equation*}

Note that our elements lie in the abelian group $\ZZ_l$; the above equations are equivalent to
\begin{equation*}
    a(i)^{-1} a(\ol{x}(i)) = {a'(i)}^{-1} a'(\ol{x}(i)) \ ,\ 
    a(i)^{-1} a(\ol{y}(i)) = {a'(i)}^{-1} a'(\ol{y}(i)) \ .
\end{equation*}
These are in turn equivalent to
\begin{equation*}
      a'(i) a(i)^{-1} 
    = a'(\ol{x}(i)) a(\ol{x}(i))^{-1}
    = a'(\ol{y}(i)) a(\ol{y}(i))^{-1} \ .
\end{equation*}
Since the subgroup $\ol{G}$ generated by $\ol{x}, \ol{y} \in \Sf_{\ol{E}} \cong \Sf_m$ acts transitively on the set $\ol{E} \cong [m]$, the above equality implies that for all $i \in [m]$, the difference between $a(i)$ and $a'(i)$ is a constant:
\begin{equation*}
      a'(i) a(i)^{-1} = a'(1)a(1)^{-1} = \id \ .
\end{equation*}
Thus $a(i) = a'(i)$ for all $2 \leq i \leq m$, i.e., the two sets of group elements $(a(i))_{i=2}^m$ and $(a'(i))_{i=2}^m$ are identical, a contradiction.

\bigskip

We now prove surjectivity.
Given $M_N \in \fix(\pp; \phi_{l, m})$,
by Lemma \ref{lem:symmetric_map_to_orbifold_map} and Proposition \ref{prop:epi_0_zetaA}, the parameters $\sigma, \ppsi, N_N, \zeta_A$ can be determined.
Moreover, the universal covering map $\td{N} = (\td{E}, \td{x}, \td{y})$ is also determined.
Then clearly $(a(i))_{i=2}^m$ is determined by $\zeta(\tde_{i, \id}) = e_{i, a(i)}$.
\end{proof}

\begin{remark} 
In the construction of Lemma \ref{lem:orbifold_map_to_symmetric_map}, choosing different $\td{x}, \td{y}$ may yield different $b_k(i), k=1,2$, and consequently possibly different $x, y$.
Thus the maps in Lemma \ref{lem:orbifold_map_to_symmetric_map} and Theorem \ref{thm:fix_enum} depend on the choice of $\td{x}, \td{y}$.
However, this does not affect the bijectivity of the mapping.
\end{remark}

\bigskip
%%%%% subsection %%%%%
\subsection{Unrooted Map Enumeration Formula} \label{sec:proof_of_unrooted_enum}

According to Lemma \ref{lem:unrooted_from_fix_enum}, the enumeration of unrooted maps can be reduced to the enumeration of symmetric maps.
Furthermore, by Theorem \ref{thm:fix_enum}, the enumeration of symmetric maps can be reduced to the enumeration of rooted orbifold embedding maps. This proves Theorem \ref{thm:rooted_to_unrooted} and Theorem \ref{thm:unrooted_QU}.

\begin{proof}[Proof of Theorem \ref{thm:rooted_to_unrooted}]
From Lemma \ref{lem:unrooted_from_fix_enum}, Theorem \ref{thm:fix_enum} and formula \eqref{eq:numbered_to_rooted}, we obtain
\begin{equation*}\begin{split}
    \abs{\Mc(\pp)} 
    =& \sum_{lm = n} \frac1{m! \ l^m} \abs{\fix(\pp; \phi_{l, m})} \\
    =& \sum_{lm = n} \frac1{m! \ l^m} \sum_{g = \sigma \times l} \sum_{\pp = \sigma \times_l \ppsi} l^{m-1} \cdot \abs{\Mc_N(\ppsi)} \cdot \abs{\Epi_0(\Gamma(\sigma), \ZZ_l)} \\
    =& \sum_{lm = n} \frac1{m! \ l^m} \sum_{g = \sigma \times l} \sum_{\pp = \sigma \times_l \ppsi} l^{m-1} \cdot \left((m-1)! \abs{\Mc_R(\ppsi)}\right) \cdot \abs{\Epi_0(\Gamma(\sigma), \ZZ_l)} \ .
\end{split}\end{equation*}

Here
\begin{equation*}
    \frac1{m! \ l^m} \cdot l^{m-1} \cdot (m-1)! = \frac1{m \cdot l} = \frac1n \ .
\end{equation*}

Thus we have proved
\begin{equation*}
    \abs{\Mc(\pp)} = 
    \frac1n \sum_{l \, |\, n} \sum_{g = \sigma \times l} \sum_{\pp = \sigma \times_l \ppsi} 
    \abs{\Mc_R(\ppsi)} \cdot \abs{\Epi_0 (\Gamma(\sigma), \ZZ_l)} \ . \qedhere
\end{equation*}
\end{proof}

\begin{proof}[Proof of Theorem \ref{thm:unrooted_QU}]
Restrict the passport in Theorem \ref{thm:rooted_to_unrooted} to the quasi-one-face passport $\ppqu$ directly.
\end{proof}

\bigskip

By the definition of multiplication of orbifold symbol and passport, only quasi-one-face passports can multiply to a quasi-one-face passport.
This completes the issue of how to compute $\abs{\Mc_R(\ppsi)}$ in Theorem \ref{thm:unrooted_QU}:
since $\ppsi$ will be quasi-one-face, it can ultimately be obtained from Theorem \ref{thm:CFF_general_passport_formula}.
Recall that it is not easy to compute the number of maps for arbitrary passports.

\begin{lemma} \label{lem:QU_divided_also_QU}
    If $\ppqu$ is an unlabeled quasi-one-face passport, then any passport $\ppsi$ satisfying $\ppqu = \sigma \times_l \ppsi$ is also a quasi-one-face passport, and $l \mid m$.
\end{lemma}

\begin{proof}
Since $\ppqu = (g, m, n; \pi_1, \pi_2) = (g, n; \pi_1, \pi_2, \pi_3)$ is a quasi-one-face passport, we have $\pi_3 = (m\ 1^{n-m})$.
By the definition of multiplication \eqref{eq:sigma_pp_multi},
the parameters $w_{3i}, t_{3i}, \lambda_{3i}$ must satisfy
\begin{equation*}
    w_{31} = m\ /\ l \ ,\  t_{31} = l \ ,\ \lambda_{31} = 1 \quad,\quad 
    w_{32} = 1 \ ,\ t_{32} = 1 \ , \ \lambda_{32} = (n-m)\ /\ l \ .
\end{equation*}
Hence $\Pi_3$ of $\ppsi$ equals $\left( {_l (m/l)}\ {_1 1^{(n - m) / l}} \right)$, which proves that $\ppsi$ is a quasi-one-face passport.
\end{proof}

\bigskip\noindent
{\bf Acknowledgment. }
The author would like to express their sincere thanks to Bin Xu and Sicheng Lu for providing influential support on this work. The author is supported by the National Natural Science Foundation of China (Grant No. 125B1019).

\newpage

%%%%% bibliogy %%%%%%
\bibliographystyle{plain}
\bibliography{ref_Tree}  % 参考文献使用

@book {LsZak04,
    AUTHOR = {Lando, Sergei K. and Zvonkin, Alexander K.},
     TITLE = {Graphs on surfaces and their applications},
    SERIES = {Encyclopaedia of Mathematical Sciences},
    VOLUME = {141},
      NOTE = {With an appendix by Don B. Zagier, Low-Dimensional Topology, II},
 PUBLISHER = {Springer-Verlag, Berlin},
      YEAR = {2004},
     PAGES = {xvi+455},
      ISBN = {3-540-00203-0},
   MRCLASS = {14H55 (05-02 05C10 05C30 05C50 14H10 14H30 32G15)},
  MRNUMBER = {2036721},
MRREVIEWER = {Athanase\ Papadopoulos},
       DOI = {10.1007/978-3-540-38361-1},
       URL = {https://doi.org/10.1007/978-3-540-38361-1},
}

@article{YYK15,
    AUTHOR = {Kochetkov, Yu. Yu.},
     TITLE = {Enumeration of a class of plane weighted trees},
   JOURNAL = {Fundam. Prikl. Mat.},
  FJOURNAL = {Fundamental\cprime naya i Prikladnaya Matematika},
    VOLUME = {18},
      YEAR = {2013},
    NUMBER = {6},
     PAGES = {171--184},
      ISSN = {1560-5159,2076-6203},
   MRCLASS = {05C30 (05C05 05C22)},
  MRNUMBER = {3431863},
       DOI = {10.1007/s10958-015-2503-5},
       URL = {https://doi.org/10.1007/s10958-015-2503-5},
}

@misc{LuSong26,
      title={Enumeration of weighted plane trees by a permutation model}, 
      author={Sicheng Lu and Yi Song},
      year={2026},
      eprint={2601.07544},
      archivePrefix={arXiv},
      primaryClass={math.CO},
      url={https://arxiv.org/abs/2601.07544}, 
}

@article {GorTor1980,
    AUTHOR = {Gordon, M. and Torkington, John A.},
     TITLE = {Enumeration of coloured plane trees with a given type partition},
   JOURNAL = {Discrete Appl. Math.},
  FJOURNAL = {Discrete Applied Mathematics. The Journal of Combinatorial Algorithms, Informatics and Computational Sciences},
    VOLUME = {2},
      YEAR = {1980},
    NUMBER = {3},
     PAGES = {207--223},
      ISSN = {0166-218X,1872-6771},
   MRCLASS = {05C30 (05C05 82A42)},
  MRNUMBER = {588700},
MRREVIEWER = {J.\ W.\ Moon},
       DOI = {10.1016/0166-218X(80)90041-4},
       URL = {https://doi.org/10.1016/0166-218X(80)90041-4},
}

@ARTICLE{SXY22axv,
       author = {{Song}, Jijian and {Xu}, Bin and {Ye}, Yu},
        title = "{A note on the Hurwitz problem and cone spherical metrics}",
      journal = {arXiv e-prints},
         year = 2022,
        month = oct,
          eid = {arXiv:2210.09700},
        pages = {arXiv:2210.09700},
          doi = {10.48550/arXiv.2210.09700},
archivePrefix = {arXiv},
       eprint = {2210.09700},
 primaryClass = {math.GR},
       adsurl = {https://ui.adsabs.harvard.edu/abs/2022arXiv221009700S}
}

@article {Hur91,
    AUTHOR = {Hurwitz, A.},
     TITLE = {Ueber {R}iemann'sche {F}l\"achen mit gegebenen
              {V}erzweigungspunkten},
   JOURNAL = {Math. Ann.},
  FJOURNAL = {Mathematische Annalen},
    VOLUME = {39},
      YEAR = {1891},
    NUMBER = {1},
     PAGES = {1--60},
      ISSN = {0025-5831,1432-1807},
   MRCLASS = {99-04},
  MRNUMBER = {1510692},
       DOI = {10.1007/BF01199469},
       URL = {https://doi.org/10.1007/BF01199469},
}

@article {Zhe06,
    AUTHOR = {Zheng, Hao},
     TITLE = {Realizability of branched coverings of {$S^2$}},
   JOURNAL = {Topology Appl.},
  FJOURNAL = {Topology and its Applications},
    VOLUME = {153},
      YEAR = {2006},
    NUMBER = {12},
     PAGES = {2124--2134},
      ISSN = {0166-8641,1879-3207},
   MRCLASS = {57M12 (05C30)},
  MRNUMBER = {2239076},
MRREVIEWER = {M.\ L.\ Marx},
       DOI = {10.1016/j.topol.2005.08.007},
       URL = {https://doi.org/10.1016/j.topol.2005.08.007},
}

@article {Zog15,
    AUTHOR = {Zograf, Peter},
     TITLE = {Enumeration of {G}rothendieck's dessins and {KP} hierarchy},
   JOURNAL = {Int. Math. Res. Not. IMRN},
  FJOURNAL = {International Mathematics Research Notices. IMRN},
      YEAR = {2015},
    NUMBER = {24},
     PAGES = {13533--13544},
      ISSN = {1073-7928,1687-0247},
   MRCLASS = {14H57 (35Q53 37K10)},
  MRNUMBER = {3436154},
MRREVIEWER = {Valentina\ Nikolaevna\ Davletshina},
       DOI = {10.1093/imrn/rnv077},
       URL = {https://doi.org/10.1093/imrn/rnv077},
}

@article {Wal75,
    AUTHOR = {Walsh, T. R. S.},
     TITLE = {Hypermaps versus bipartite maps},
   JOURNAL = {J. Combinatorial Theory Ser. B},
  FJOURNAL = {Journal of Combinatorial Theory. Series B},
    VOLUME = {18},
      YEAR = {1975},
     PAGES = {155--163},
      ISSN = {0095-8956},
   MRCLASS = {05C10},
  MRNUMBER = {360328},
MRREVIEWER = {Seth\ R.\ Alpert},
       DOI = {10.1016/0095-8956(75)90042-8},
       URL = {https://doi.org/10.1016/0095-8956(75)90042-8},
}

@article {Arq87,
    AUTHOR = {Arqu\`es, Didier},
     TITLE = {Hypercartes point\'ees sur le tore: d\'ecompositions et
              d\'enombrements},
   JOURNAL = {J. Combin. Theory Ser. B},
  FJOURNAL = {Journal of Combinatorial Theory. Series B},
    VOLUME = {43},
      YEAR = {1987},
    NUMBER = {3},
     PAGES = {275--286},
      ISSN = {0095-8956,1096-0902},
   MRCLASS = {05C30 (05C10 05C65)},
  MRNUMBER = {916372},
MRREVIEWER = {E.\ M.\ Palmer},
       DOI = {10.1016/0095-8956(87)90003-7},
       URL = {https://doi.org/10.1016/0095-8956(87)90003-7},
}

@article {Zan95,
    AUTHOR = {Zannier, Umberto},
     TITLE = {On {D}avenport's bound for the degree of {$f^3-g^2$} and
              {R}iemann's existence theorem},
   JOURNAL = {Acta Arith.},
  FJOURNAL = {Acta Arithmetica},
    VOLUME = {71},
      YEAR = {1995},
    NUMBER = {2},
     PAGES = {107--137},
      ISSN = {0065-1036,1730-6264},
   MRCLASS = {11C08 (11D25 11G99)},
  MRNUMBER = {1339121},
MRREVIEWER = {T.\ N.\ Shorey},
       DOI = {10.4064/aa-71-2-107-137},
       URL = {https://doi.org/10.4064/aa-71-2-107-137},
}

@article {HarTut64,
    AUTHOR = {Harary, F. and Tutte, W. T.},
     TITLE = {The number of plane trees with a given partition},
   JOURNAL = {Mathematika},
  FJOURNAL = {Mathematika. A Journal of Pure and Applied Mathematics},
    VOLUME = {11},
      YEAR = {1964},
     PAGES = {99--101},
      ISSN = {0025-5793},
   MRCLASS = {05.45},
  MRNUMBER = {174490},
MRREVIEWER = {Gert\ Sabidussi},
       DOI = {10.1112/S0025579300004307},
       URL = {https://doi.org/10.1112/S0025579300004307},
}

@article {GlJac92,
    AUTHOR = {Goulden, I. P. and Jackson, D. M.},
     TITLE = {The combinatorial relationship between trees, cacti and
              certain connection coefficients for the symmetric group},
   JOURNAL = {European J. Combin.},
  FJOURNAL = {European Journal of Combinatorics},
    VOLUME = {13},
      YEAR = {1992},
    NUMBER = {5},
     PAGES = {357--365},
      ISSN = {0195-6698,1095-9971},
   MRCLASS = {05E15 (05C25 20B30)},
  MRNUMBER = {1181077},
MRREVIEWER = {Andrea\ Brini},
       DOI = {10.1016/S0195-6698(05)80015-0},
       URL = {https://doi.org/10.1016/S0195-6698(05)80015-0},
}

@article {GpSch98,
    AUTHOR = {Goupil, Alain and Schaeffer, Gilles},
     TITLE = {Factoring {$n$}-cycles and counting maps of given genus},
   JOURNAL = {European J. Combin.},
  FJOURNAL = {European Journal of Combinatorics},
    VOLUME = {19},
      YEAR = {1998},
    NUMBER = {7},
     PAGES = {819--834},
      ISSN = {0195-6698,1095-9971},
   MRCLASS = {05C30 (05C10 05C38 57M99)},
  MRNUMBER = {1649966},
MRREVIEWER = {K.\ S.\ Sarkaria},
       DOI = {10.1006/eujc.1998.0215},
       URL = {https://doi.org/10.1006/eujc.1998.0215},
}

@article {PouSch02,
    AUTHOR = {Poulalhon, Dominique and Schaeffer, Gilles},
     TITLE = {Factorizations of large cycles in the symmetric group},
   JOURNAL = {Discrete Math.},
  FJOURNAL = {Discrete Mathematics},
    VOLUME = {254},
      YEAR = {2002},
    NUMBER = {1-3},
     PAGES = {433--458},
      ISSN = {0012-365X,1872-681X},
   MRCLASS = {20B30 (05A15 20C30)},
  MRNUMBER = {1910123},
MRREVIEWER = {Gary\ L.\ Walls},
       DOI = {10.1016/S0012-365X(01)00361-2},
       URL = {https://doi.org/10.1016/S0012-365X(01)00361-2},
}

@article {Boc82,
    AUTHOR = {Boccara, G.},
     TITLE = {Cycles comme produit de deux permutations de classes
              donn\'ees},
   JOURNAL = {Discrete Math.},
  FJOURNAL = {Discrete Mathematics},
    VOLUME = {38},
      YEAR = {1982},
    NUMBER = {2-3},
     PAGES = {129--142},
      ISSN = {0012-365X,1872-681X},
   MRCLASS = {20B35 (05A05)},
  MRNUMBER = {676530},
MRREVIEWER = {J.\ L.\ Brenner},
       DOI = {10.1016/0012-365X(82)90282-5},
       URL = {https://doi.org/10.1016/0012-365X(82)90282-5},
}

@article {CFF13,
    AUTHOR = {Chapuy, Guillaume and F\'eray, Valentin and Fusy, \'Eric},
     TITLE = {A simple model of trees for unicellular maps},
   JOURNAL = {J. Combin. Theory Ser. A},
  FJOURNAL = {Journal of Combinatorial Theory. Series A},
    VOLUME = {120},
      YEAR = {2013},
    NUMBER = {8},
     PAGES = {2064--2092},
      ISSN = {0097-3165,1096-0899},
   MRCLASS = {05C10 (05A19)},
  MRNUMBER = {3102175},
MRREVIEWER = {Anna\ de Mier},
       DOI = {10.1016/j.jcta.2013.08.003},
       URL = {https://doi.org/10.1016/j.jcta.2013.08.003},
}

@article {MedNed10,
    AUTHOR = {Mednykh, Alexander and Nedela, Roman},
     TITLE = {Enumeration of unrooted hypermaps of a given genus},
   JOURNAL = {Discrete Math.},
  FJOURNAL = {Discrete Mathematics},
    VOLUME = {310},
      YEAR = {2010},
    NUMBER = {3},
     PAGES = {518--526},
      ISSN = {0012-365X,1872-681X},
   MRCLASS = {05C30 (05C10 05C25 05C65 20H10 30F20)},
  MRNUMBER = {2564806},
MRREVIEWER = {Valery\ A.\ Liskovets},
       DOI = {10.1016/j.disc.2009.03.033},
       URL = {https://doi.org/10.1016/j.disc.2009.03.033},
}

@article {MedNed06,
    AUTHOR = {Mednykh, Alexander and Nedela, Roman},
     TITLE = {Enumeration of unrooted maps of a given genus},
   JOURNAL = {J. Combin. Theory Ser. B},
  FJOURNAL = {Journal of Combinatorial Theory. Series B},
    VOLUME = {96},
      YEAR = {2006},
    NUMBER = {5},
     PAGES = {706--729},
      ISSN = {0095-8956,1096-0902},
   MRCLASS = {05C30 (05C10)},
  MRNUMBER = {2236507},
MRREVIEWER = {D.\ S.\ Archdeacon},
       DOI = {10.1016/j.jctb.2006.01.005},
       URL = {https://doi.org/10.1016/j.jctb.2006.01.005},
}

@article {Med84,
    AUTHOR = {Mednykh, A. D.},
     TITLE = {Nonequivalent coverings of {R}iemann surfaces with a
              prescribed ramification type},
   JOURNAL = {Sibirsk. Mat. Zh.},
  FJOURNAL = {Akademiya Nauk SSSR. Sibirskoe Otdelenie. Sibirski\u i\
              Matematicheski\u i\ Zhurnal},
    VOLUME = {25},
      YEAR = {1984},
    NUMBER = {4},
     PAGES = {120--142},
      ISSN = {0037-4474},
   MRCLASS = {30F10 (30C25 30F20)},
  MRNUMBER = {754748},
MRREVIEWER = {Boris\ N.\ Apanasov},
}

@article {Har66,
    AUTHOR = {Harvey, W. J.},
     TITLE = {Cyclic groups of automorphisms of a compact {R}iemann surface},
   JOURNAL = {Quart. J. Math. Oxford Ser. (2)},
  FJOURNAL = {The Quarterly Journal of Mathematics. Oxford. Second Series},
    VOLUME = {17},
      YEAR = {1966},
     PAGES = {86--97},
      ISSN = {0033-5606,1464-3847},
   MRCLASS = {30.49},
  MRNUMBER = {201629},
MRREVIEWER = {G.\ G.\ Weill},
       DOI = {10.1093/qmath/17.1.86},
       URL = {https://doi.org/10.1093/qmath/17.1.86},
}

@article {Won71,
    AUTHOR = {Wong, C. K.},
     TITLE = {A uniformization theorem for arbitrary {R}iemann surfaces with
              signature},
   JOURNAL = {Proc. Amer. Math. Soc.},
  FJOURNAL = {Proceedings of the American Mathematical Society},
    VOLUME = {28},
      YEAR = {1971},
     PAGES = {489--495},
      ISSN = {0002-9939,1088-6826},
   MRCLASS = {30.45},
  MRNUMBER = {279303},
MRREVIEWER = {C.\ Earle},
       DOI = {10.2307/2037998},
       URL = {https://doi.org/10.2307/2037998},
}

\bigskip

\begingroup
\footnotesize 
Yi Song

\textsc{School of Mathematical Sciences, University of Science and Technology of China, Hefei, Anhui, People's Republic of China.} 

\textit{Email address}: \texttt{\textcolor[rgb]{0.00,0.00,0.84}{sif4delta0@mail.ustc.edu.cn}}

\endgroup

\newpage

\begin{appendix}

%%%%% section %%%%%
\section{Enumeration of Totally Labeled Weighted Trees} \label{ch:labeled_weighted_tree_enum}

This section presents results on the existence (Theorem \ref{thm:exist_tree}) and enumeration (Theorem \ref{thm:YYK_formula}) of genus 0 quasi-one-face maps (weighted bicolored plane trees).
The genus 0 results also serve as the foundation for solving the existence and enumeration problems for positive genus quasi-one-face maps, as detailed in the proofs of Theorem \ref{thm:CFF_general_passport_formula} and Corollary \ref{cor:exist_quasi-unicellular_map} at the end of Section \ref{ch:rooted_quasi-unicellular_map_enum}.

\bigskip

\begin{theorem}[Boccara \cite{Boc82} and Zannier \cite{Zan95}] \label{thm:exist_tree}
    Given a genus 0 quasi-one-face passport $\ppz = (g, m, n; \Pi_1, \Pi_2)$, $\Pi_i = (S_i, \lambda_i, \wt_i)$.
    Denote
    $$\gcd(\ppz) = \gcd\left(\wt_1(S_1) \sqcup \wt_2(S_2)\right) \ ,$$
    the greatest common divisor of all weights in the weight distribution of the passport.
    
    Then there exists weighted bicolored plane tree with passport $\ppz$ or $\Uc_W(\ppz) \neq \varnothing$, if and only if
    $(v(\ppz) - 1) \cdot \gcd(\ppz) \leq n$.
    \qed
\end{theorem}

\bigskip

\begin{definition}
    Given a totally labeled genus 0 quasi-one-face passport $\ppzl = (g, m, n; \Pi_1, \Pi_2)$, $\Pi_i = (S_i, \lambda_i, \wt_i)$.
    Two non-empty subsets $S_1' \subseteq S_1, S_2' \subseteq S_2$ satisfying the following relation
    \begin{equation}
        n' := \sum_{s \in S_1'} \wt_1(s) = \sum_{s \in S_2'} \wt_2(s) \ ,
    \end{equation}
    define a \textbf{sub-passport} of $\pp_{0, L}$
    $$
    \pp(S_1', S_2') := \big( 0, m', n'; \Pi_1' = (S_1', \one, \left.\wt_1\right|_{S_1'}) \ ,\ \Pi_2' = (S_2', \one, \left.\wt_2\right|_{S_2'}) \big) \ .
    $$
    We allow $S_1' = S_1$ and $S_2' = S_2$, which defines the \textbf{trivial sub-passport} $\pp(S_1, S_2) = \pp_{0, L}$ of $\pp_{0, L}$.
\end{definition}

\begin{definition}
    A \textbf{$k$-partition} $\pttp = \{\pp(S_{1, i}, S_{2, i})\}_{i=1}^k$ of a totally labeled passport $\pp_{0, L}$ is a $k$-element set of sub-passports of $\pp_{0, L}$ such that their ``disjoint union'' is the trivial sub-passport of $\pp_{0, L}$
    \begin{equation}
        \bigsqcup_{i=1}^k S_{1, i} = S_1 \ , \ \bigsqcup_{i=1}^k S_{2, i} = S_2 \ .
    \end{equation}
    The number of elements of a partition is the length of the partition $\abs{\pttp} := k$.
    The set of all partitions is denoted by $\pttp \in \Ptt(\pp_{0, L})$.
    We allow the \textbf{trivial partition} $\ptt{e} := \{\pp_{0, L}\}$ which contains only the trivial sub-passport.

    The $X$ function of a partition is analogous to its factorial $X(\pttp) := \prod_{i=1}^k (v(\pp_i) - 1)!$.
\end{definition}

\begin{theorem}[Kochetkov \cite{YYK15}] \label{thm:YYK_formula}
    The number of totally labeled weighted bicolored plane trees corresponding to a genus $0$ totally labeled quasi-one-face passport $\pp_{0, L} = (0, m, n; \Pi_1, \Pi_2)$ is
    \begin{equation}
        \abs{\Mc(\pp_{0, L})} = \sum_{\pttp \in \Ptt(\pp_{0, L})} (-1)^{\abs{\pttp} - 1} m^{\abs{\pttp} - 2} X(\pttp) \ .
    \end{equation}
    \qed
\end{theorem}

\begin{proposition} \label{prop:weighted_tree_no_isom}
    Given a genus $0$ totally labeled quasi-one-face passport $\ppzl = (0, m, n; \Pi_1, \Pi_2)$.
    Since every totally labeled weighted tree $T \in \Mc(\pp_{0, L})$ has only trivial automorphism $\Aut(T) \cong \{\id\}$, we have
    \begin{equation}
        \abs{\Uc_{R, W}(\ppzl)} = m \abs{\Uc_W(\ppzl)} = m \abs{\Mc(\ppzl)} \ ,\ \abs{\Mc_R(\ppzl)} = n \abs{\Mc(\ppzl)} \ .
    \end{equation}
\end{proposition}

\begin{proof}
Fix a totally labeled weighted tree $T = (S^2; E, U_k; \Lab_k, \wt_E) \in \Uc_W(\ppzl)$.
Assume that the index sets of weight distributions in passport $\ppzl$ are $[u_k(\ppzl)], \\ k=1,2$.
Denote the vertices $v^{(k)}_i := \Lab_k^{-1}(i)$.
Since there is at most one edge between any two vertices in a tree, we can encode the edges as
\begin{equation*}
    \Nf : E \hookrightarrow [u_1(\ppzl)] \times [u_2(\ppzl)] \ ,\ 
    e = \{v^{(1)}_i, v^{(2)}_j\} \mapsto (i, j) \ .
\end{equation*}
mapping an edge $e$ connecting vertices $v^{(1)}_i, v^{(2)}_j$ to the pair of positive integers $(i, j)$.
By sorting the image set $\Nf(E)$ in lexicographic order, we obtain an permutation $(e_l)_{l=1}^n$ of the edges, which gives a totally numbered weighted tree $(T, (e_l)_{l=1}^n)$. 
It has only the trivial automorphism.
Consequently, the original totally labeled weighted tree $T$ also has only the trivial automorphism.

From formulas \eqref{eq:rooted_and_nonrooted} and \ref{prop:QU_cong_UW}, we obtain
\begin{equation*}
    \abs{\Uc_{R, W}(\ppzl)} = m \sum_{M \in \Uc_W(\ppzl)} \frac1{\abs{\Aut(M)}} = m \abs{\Uc_W(\ppzl)} = m \abs{\Mc(\ppzl)} \ . 
\end{equation*}
\begin{equation*}
    \abs{\Mc_R(\ppzl)} = n \sum_{M \in \Mc(\ppzl)} \frac1{\abs{\Aut(M)}} = n \abs{\Mc(\ppzl)} \ . \qedhere
\end{equation*}

\end{proof}

\bigskip
%%%%% section %%%%%
\section{Enumeration of Order-Preserving Epimorphisms} \label{ch:epi_enum}

This section lists the results of Mednykh-Nedela \cite{MedNed06} on the enumeration of order-preserving epimorphisms from Fuchsian groups $\Gamma(\sigma)$ to finite cyclic groups $\ZZ_l$.
Some of the number-theoretic functions used are described below.

\bigskip

Euler's totient function $\varphi(n)$: the number of positive integers that are less than $n$ and coprime to $n$.

M\"{o}bius function $\mu(n)$: for a positive integer $n$ with prime factorization $n = p_1^{e_1} \cdots p_r^{e_r}$, $\mu(n)$ is defined as
\begin{itemize}
    \item if $n = 1$, $\mu(1) = 1$;
    \item if all prime factors appear to the first power $e_1 = \cdots = e_r = 1$, then $\mu(n) = (-1)^r$ depending on the parity of the number of prime factors;
    \item if there is any prime factor with exponent $e_i > 1$, $\mu(n) = 0$.
\end{itemize}

The following functions can also be defined.

\bigskip

von Sterneck function
\begin{equation}
    \eta(x, n) := \frac{\varphi(n) \cdot \mu\big(n \big/ \gcd(x, n)\big)}{\varphi\big(n \big/ \gcd(x, n)\big)} = \sum_{\substack{ 1 \leq k \leq n \\ \gcd(k, n) = 1 }} \exp\left( \frac{2\pi i k x}{n} \right) \ .
\end{equation}

Orbicyclic arithmetic function
\begin{equation}
    E(t_1, \cdots, t_r) = \frac1{T} \sum_{k=1}^T \prod_{i=1}^r \eta(k, t_i) \quad ,\quad \text{where}\ T = \lcm(t_1, \ldots, t_r) \ .
\end{equation}

Jordan multiplicative function
\begin{equation}
    \varphi_k(n) = \sum_{d \,|\, n} \mu\left( \frac{n}{d} \right) d^k \ .
\end{equation}

\begin{proposition}[Mednykh-Nedela \cite{MedNed06}] \label{prop:Epi_count}
Given an orbifold signature $\sigma = (g; t_1, \cdots, t_r)$ and a positive integer $l$, let $T = \lcm(t_1, \cdots, t_r)$. Then
\begin{equation}
    \abs{\Epi_0 (\Gamma(\sigma), \ZZ_l)} = T^{2g} \cdot \varphi_{2g} (l\ /\ T) \cdot E(t_1, \cdots, t_r) \ .
\end{equation}
\qed
\end{proposition}

\bigskip
%%%%% section %%%%%
\section{Examples for the Main Theorems} \label{ch:example}

This section lists three examples of applying Theorem \ref{thm:CFF_general_passport_formula} and Theorem \ref{thm:unrooted_QU} to computations.

\begin{example} \label{eg:ppqu_844}
Take $\ppqu = (1, 4, 8; 8 \ ,\ 4^2)$, i.e., $g = 1, m = 4, n = 8$ and $\pi_1 = (8), \pi_2 = (4^2)$ are weight distributions.
Then $\Fill(\ppqu)$ has weight distributions $\bmwt_1 = (8), \bmwt_2 = (4, 4)$.

First, use Theorem \ref{thm:CFF_general_passport_formula} to compute the number of rooted maps.

When the cyclic data is $\Lambda_1 = ((1) ,(0, 0))$, the passports satisfying the condition $\pp_1 \xrightarrow{\Lambda_1} \Fill(\ppqu)$ are
\begin{enumerate}
    \item $\pp_{1,1} = (0, 4, 8; (6,1,1) ,(4, 4))$ (or change $(6,1,1)$ to $(1,6,1), (1,1,6)$), $\abs{\Mc(\pp_{1,1})} = 6$.
    \item $\pp_{1,2} = (0, 4, 8; (5,2,1) , (4, 4))$ (similarly, after permuting $(5,2,1)$, there are $6$ passports), $\abs{\Mc(\pp_{1,2})} = 6$.
    \item $\pp_{1,3} = (0, 4, 8; (4,3,1)  , (4, 4))$ (there are $6$), $\abs{\Mc(\pp_{1,3})} = 2$.
    \item $\pp_{1,4} = (0, 4, 8; (4,2,2)  , (4, 4))$ (there are $3$), $\abs{\Mc(\pp_{1,4})} = 2$.
    \item $\pp_{1,5} = (0, 4, 8; (3,3,2) , (4, 4))$ (there are $3$), $\abs{\Mc(\pp_{1,5})} = 6$.
\end{enumerate}

\bigskip

When the cyclic data is $\Lambda_2 = ((0) , (1, 0))$, the passports satisfying the condition $\pp_2 \xrightarrow{\Lambda_2} \Fill(\ppqu)$ are
$\pp_{2} = (0, 4, 8; (8) , (2, 1, 1 ;4))$.
Here $(2, 1, 1 ;4)$ means that in the two-dimensional array $\bmwt_{T, 2}$, the first row is $(2, 1, 1)$ and the second row is $(4)$.
After permuting the first row $(2, 1, 1)$, we obtain $3$ similar passports, all with $\abs{\Mc(\pp_2)} = 6$.

When the cyclic data is $\Lambda_3 = ((0) , (0, 1))$, similar to the $\Lambda_2$ case, there are $3$ passports and $6$ trees.

Based on the above results, we have
\begin{equation*}
\begin{split}
     &\abs{\Mc_R(\ppqu)} \\
    =&\ \frac{5}{2^{2 \times 2} \times 2!} \times 3^{-1} \big((6 \times 3 + 6 \times 6 + 2 \times 6 + 2 \times 3 + 6 \times 3) + (6 \times 3) + (6 \times 3) \big) \\
    =&\ 42 \ .
\end{split}
\end{equation*}

\bigskip

Next, use Theorem \ref{thm:unrooted_QU} to compute the number of unrooted maps.
First consider the positive integer $l$, the orbifold symbol $\sigma$, and the passport $\ppsi$ that satisfy the condition $l \ |\ m=4, 2 = \sigma \times l, \ppqu = \sigma \times_l \ppsi$.
Here we write $\ppqu =\\ (1, 8; 8 \ ,\ 4^2 \ ,\ 4\ 1^4)$.

\begin{enumerate}
    \item When $l = 1$, naturally we have $\sigma_1 = (2; \varnothing)$ and $\pp_{\sigma_1} = \ppqu$.
    Here $\abs{\Mc_R(\ppqu)} = 42$ and $\abs{\Epi_0(\Gamma(\sigma_1), \ZZ_1)} = 1$.

    \item When $l = 2$, $\sigma_2 = (0; 2^4)$, we can find $\pp_{\sigma_2} = (0, 4; {_2 4} \ ,\ {_2 2^2} \ ,\ {_2 2} \ {_1 1^2})$.
    Here $\big| \Mc_R(\pp_{\sigma_2}) \big| = 2$, $\abs{\Epi_0(\Gamma(\sigma_2), \ZZ_2)} = 1$.

    \item When $l = 4$, $\sigma_3 = (0; 4^2\ 2)$, we can find $\pp_{\sigma_3} = (0, 2; {_4 2} \ ,\ {_2 2} \ ,\ {_4 1} \ {_1 1})$.
    Here $\big| \Mc_R(\pp_{\sigma_3}) \big| = 2$, $\abs{\Epi_0(\Gamma(\sigma_2), \ZZ_4)} = 2$.
\end{enumerate}

Thus
\begin{equation*}
    \abs{\Mc(\ppqu)} = \frac18 (42 \times 1 + 2 \times 1 + 2 \times 2) = 6 \ .
\end{equation*}

\bigskip

Figure \ref{fig:114514} lists all $6$ maps in $\Uc_R(\ppqu) \cong \Mc(\ppqu)$.
\end{example}

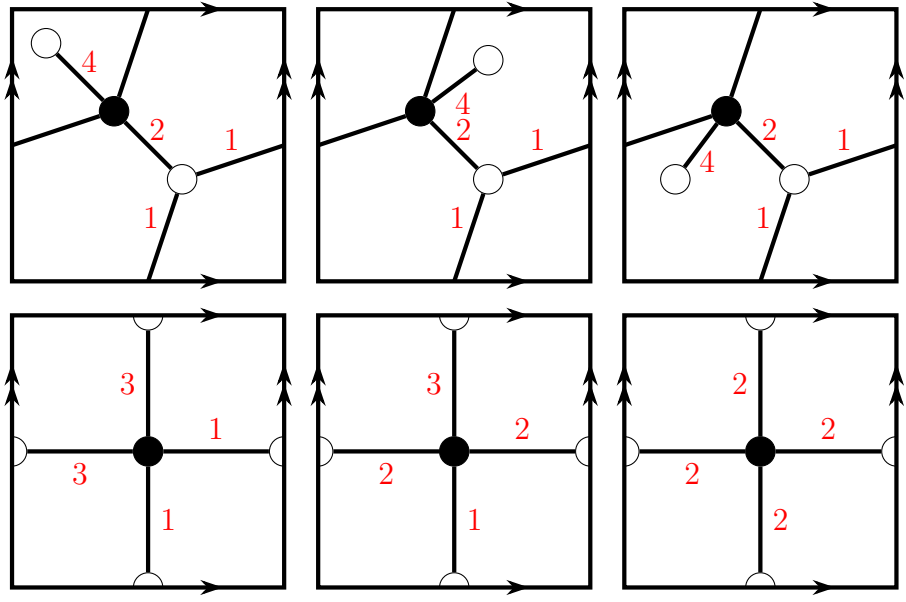
\begin{figure}
\centering
\begin{tikzpicture}[scale=0.9]
    % 定义正方形边长
    \def\side{4}
    
    %%%%% map1 %%%%%
    \begin{scope}
    % 绘制正方形边框
    \draw[ultra thick] (0,0) rectangle (\side,\side);
    
    % 在上下两边添加向右的箭头（位于边的中间）
    \draw[thick, -{Stealth[length=3mm, width=2mm]}] 
        (3 - 0.2,0) -- ++(0.3,0);
    \draw[thick, -{Stealth[length=3mm, width=2mm]}] 
        (3 - 0.2,\side) -- ++(0.3,0);
    
    % 在左右两边添加向上的双箭头（位于边的中间，箭头方向相同）
    \draw[thick, -{Stealth[length=3mm, width=2mm]}] 
        (0,3 - 0.3) -- ++(0,0.3);
    \draw[thick, -{Stealth[length=3mm, width=2mm]}] 
        (0,3) -- ++(0,0.3);
    \draw[thick, -{Stealth[length=3mm, width=2mm]}] 
        (\side,3- 0.3) -- ++(0,0.3);
    \draw[thick, -{Stealth[length=3mm, width=2mm]}] 
        (\side,3) -- ++(0,0.3);

    % 绘制一般的黑点白点和连线
    \node[blk] (B) at (1.5, 2.5) {};
    \node[wht] (W) at (2.5, 1.5) {};
    \draw (B) edge[ultra thick] node[color = red, pos=0.7, above=0.7] {2} (W) ;
    \draw (B) edge[ultra thick] (0, \side/2) ;
    \draw (B) edge[ultra thick] (\side/2, \side) ;
    \draw (W) edge[ultra thick] node[color = red, pos=0.4, above=0.2] {1} (\side, \side/2);
    \draw (W) edge[ultra thick] node[color = red, pos=0.9, above=12] {1} (\side/2, 0);
    
    %绘制另一白点的连线
    \node[wht] (W1) at (0.5, 3.5) {};
    \draw (B) edge[ultra thick] node[color = red, pos=0.3, above=0.7] {4} (W1);
    \end{scope}

    %%%%% map2 %%%%%
    \begin{scope}[shift={(4.5, 0)}]
    % 绘制正方形边框
    \draw[ultra thick] (0,0) rectangle (\side,\side);
    
    % 在上下两边添加向右的箭头（位于边的中间）
    \draw[thick, -{Stealth[length=3mm, width=2mm]}] 
        (3 - 0.2,0) -- ++(0.3,0);
    \draw[thick, -{Stealth[length=3mm, width=2mm]}] 
        (3 - 0.2,\side) -- ++(0.3,0);
    
    % 在左右两边添加向上的双箭头（位于边的中间，箭头方向相同）
    \draw[thick, -{Stealth[length=3mm, width=2mm]}] 
        (0,3 - 0.3) -- ++(0,0.3);
    \draw[thick, -{Stealth[length=3mm, width=2mm]}] 
        (0,3) -- ++(0,0.3);
    \draw[thick, -{Stealth[length=3mm, width=2mm]}] 
        (\side,3- 0.3) -- ++(0,0.3);
    \draw[thick, -{Stealth[length=3mm, width=2mm]}] 
        (\side,3) -- ++(0,0.3);

    % 绘制一般的黑点白点和连线
    \node[blk] (B) at (1.5, 2.5) {};
    \node[wht] (W) at (2.5, 1.5) {};
    \draw (B) edge[ultra thick] node[color = red, pos=0.7, above=0.7] {2} (W) ;
    \draw (B) edge[ultra thick] (0, \side/2) ;
    \draw (B) edge[ultra thick] (\side/2, \side) ;
    \draw (W) edge[ultra thick] node[color = red, pos=0.4, above=0.2] {1} (\side, \side/2);
    \draw (W) edge[ultra thick] node[color = red, pos=0.9, above=12] {1} (\side/2, 0);
    
    %绘制另一白点的连线
    \node[wht] (W1) at (2.5, 3.25) {};
    \draw (B) edge[ultra thick] node[color = red, pos=0.7, below=0.5] {4} (W1);
    \end{scope}

    %%%%% map3 %%%%%
    \begin{scope}[shift={(9, 0)}]
    % 绘制正方形边框
    \draw[ultra thick] (0,0) rectangle (\side,\side);
    
    % 在上下两边添加向右的箭头（位于边的中间）
    \draw[thick, -{Stealth[length=3mm, width=2mm]}] 
        (3 - 0.2,0) -- ++(0.3,0);
    \draw[thick, -{Stealth[length=3mm, width=2mm]}] 
        (3 - 0.2,\side) -- ++(0.3,0);
    
    % 在左右两边添加向上的双箭头（位于边的中间，箭头方向相同）
    \draw[thick, -{Stealth[length=3mm, width=2mm]}] 
        (0,3 - 0.3) -- ++(0,0.3);
    \draw[thick, -{Stealth[length=3mm, width=2mm]}] 
        (0,3) -- ++(0,0.3);
    \draw[thick, -{Stealth[length=3mm, width=2mm]}] 
        (\side,3- 0.3) -- ++(0,0.3);
    \draw[thick, -{Stealth[length=3mm, width=2mm]}] 
        (\side,3) -- ++(0,0.3);

    % 绘制一般的黑点白点和连线
    \node[blk] (B) at (1.5, 2.5) {};
    \node[wht] (W) at (2.5, 1.5) {};
    \draw (B) edge[ultra thick] node[color = red, pos=0.7, above=0.7] {2} (W) ;
    \draw (B) edge[ultra thick] (0, \side/2) ;
    \draw (B) edge[ultra thick] (\side/2, \side) ;
    \draw (W) edge[ultra thick] node[color = red, pos=0.4, above=0.2] {1} (\side, \side/2);
    \draw (W) edge[ultra thick] node[color = red, pos=0.9, above=12] {1} (\side/2, 0);
    
    %绘制另一白点的连线
    \node[wht] (W1) at (0.75, 1.5) {};
    \draw (B) edge[ultra thick] node[color = red, pos=0.3, below=1.0] {4} (W1);
    \end{scope}

    %%%%% map4 %%%%%
    \begin{scope}[shift={(0, -4.5)}]

    % 在上下两边添加向右的箭头（位于边的中间）
    \draw[thick, -{Stealth[length=3mm, width=2mm]}] 
        (3 - 0.2,0) -- ++(0.3,0);
    \draw[thick, -{Stealth[length=3mm, width=2mm]}] 
        (3 - 0.2,\side) -- ++(0.3,0);
    
    % 在左右两边添加向上的双箭头（位于边的中间，箭头方向相同）
    \draw[thick, -{Stealth[length=3mm, width=2mm]}] 
        (0,3 - 0.3) -- ++(0,0.3);
    \draw[thick, -{Stealth[length=3mm, width=2mm]}] 
        (0,3) -- ++(0,0.3);
    \draw[thick, -{Stealth[length=3mm, width=2mm]}] 
        (\side,3- 0.3) -- ++(0,0.3);
    \draw[thick, -{Stealth[length=3mm, width=2mm]}] 
        (\side,3) -- ++(0,0.3);
    % 先绘制边处顶点
    \begin{scope}
    \clip (0,0) rectangle (\side,\side);
    \node[wht] (W11) at (\side/2, 0) {};
    \node[wht] (W12) at (\side/2, \side) {};
    \node[wht] (W21) at (\side, \side/2) {};
    \node[wht] (W22) at (0, \side/2) {};
    \end{scope}
    % 再绘制正方形边框
    \draw[ultra thick] (0,0) rectangle (\side,\side);

    % 绘制面心的黑点和连线
    \node[blk] (B) at (2, 2) {};
    \draw (B) edge[ultra thick] node[color = red, auto] {1} (W11);
    \draw (B) edge[ultra thick] node[color = red, auto] {3} (W12);
    \draw (B) edge[ultra thick] node[color = red, auto] {1} (W21);
    \draw (B) edge[ultra thick] node[color = red, auto] {3} (W22);
    \end{scope}

    %%%%% map5 %%%%%
    \begin{scope}[shift={(4.5, -4.5)}]

    % 在上下两边添加向右的箭头（位于边的中间）
    \draw[thick, -{Stealth[length=3mm, width=2mm]}] 
        (3 - 0.2,0) -- ++(0.3,0);
    \draw[thick, -{Stealth[length=3mm, width=2mm]}] 
        (3 - 0.2,\side) -- ++(0.3,0);
    
    % 在左右两边添加向上的双箭头（位于边的中间，箭头方向相同）
    \draw[thick, -{Stealth[length=3mm, width=2mm]}] 
        (0,3 - 0.3) -- ++(0,0.3);
    \draw[thick, -{Stealth[length=3mm, width=2mm]}] 
        (0,3) -- ++(0,0.3);
    \draw[thick, -{Stealth[length=3mm, width=2mm]}] 
        (\side,3- 0.3) -- ++(0,0.3);
    \draw[thick, -{Stealth[length=3mm, width=2mm]}] 
        (\side,3) -- ++(0,0.3);
    % 先绘制边处顶点
    \begin{scope}
    \clip (0,0) rectangle (\side,\side);
    \node[wht] (W11) at (\side/2, 0) {};
    \node[wht] (W12) at (\side/2, \side) {};
    \node[wht] (W21) at (\side, \side/2) {};
    \node[wht] (W22) at (0, \side/2) {};
    \end{scope}
    % 再绘制正方形边框
    \draw[ultra thick] (0,0) rectangle (\side,\side);

    % 绘制面心的黑点和连线
    \node[blk] (B) at (2, 2) {};
    \draw (B) edge[ultra thick] node[color = red, auto] {1} (W11);
    \draw (B) edge[ultra thick] node[color = red, auto] {3} (W12);
    \draw (B) edge[ultra thick] node[color = red, auto] {2} (W21);
    \draw (B) edge[ultra thick] node[color = red, auto] {2} (W22);
    \end{scope}

    %%%%% map6 %%%%%
    \begin{scope}[shift={(9, -4.5)}]

    % 在上下两边添加向右的箭头（位于边的中间）
    \draw[thick, -{Stealth[length=3mm, width=2mm]}] 
        (3 - 0.2,0) -- ++(0.3,0);
    \draw[thick, -{Stealth[length=3mm, width=2mm]}] 
        (3 - 0.2,\side) -- ++(0.3,0);
    
    % 在左右两边添加向上的双箭头（位于边的中间，箭头方向相同）
    \draw[thick, -{Stealth[length=3mm, width=2mm]}] 
        (0,3 - 0.3) -- ++(0,0.3);
    \draw[thick, -{Stealth[length=3mm, width=2mm]}] 
        (0,3) -- ++(0,0.3);
    \draw[thick, -{Stealth[length=3mm, width=2mm]}] 
        (\side,3- 0.3) -- ++(0,0.3);
    \draw[thick, -{Stealth[length=3mm, width=2mm]}] 
        (\side,3) -- ++(0,0.3);
    % 先绘制边处顶点
    \begin{scope}
    \clip (0,0) rectangle (\side,\side);
    \node[wht] (W11) at (\side/2, 0) {};
    \node[wht] (W12) at (\side/2, \side) {};
    \node[wht] (W21) at (\side, \side/2) {};
    \node[wht] (W22) at (0, \side/2) {};
    \end{scope}
    % 再绘制正方形边框
    \draw[ultra thick] (0,0) rectangle (\side,\side);

    % 绘制面心的黑点和连线
    \node[blk] (B) at (2, 2) {};
    \draw (B) edge[ultra thick] node[color = red, auto] {2} (W11);
    \draw (B) edge[ultra thick] node[color = red, auto] {2} (W12);
    \draw (B) edge[ultra thick] node[color = red, auto] {2} (W21);
    \draw (B) edge[ultra thick] node[color = red, auto] {2} (W22);
    \end{scope}

\end{tikzpicture}

\caption{$6$ maps with passport $\ppqu = (1, 4, 8; 8\ ,\ 4^2)$} \label{fig:114514}

\end{figure}

In Figure \ref{fig:114514}, only the map at the bottom right has $4$-fold symmetry, thus $\abs{\Mc_R(\ppqu)} = 8 \times (5 + \frac14) = 42$.

\bigskip

\begin{example} \label{eg:ppqu_55}
Take the quasi-one-face passport $\ppqu = (2, 5, 5; 5\ ,\ 5)$.

First, use Theorem \ref{thm:CFF_general_passport_formula} to compute the number of rooted maps.
\begin{enumerate}
    \item When the cyclic data is $\Lambda_1 = ((2) ,(0))$, the only passport satisfying the condition $\pp_1 \xrightarrow{\Lambda_1} \ppqu$ is $\pp_1 = (0, 5, 5; (1,1,1,1,1) , (5))$.
    And $\abs{\Mc(\pp_1)} = 24$.
    
    \item When the cyclic data is $\Lambda_2 = ((1),(1))$,
    since the $3$-partitions of $5$ are $\{3, 1, 1\}$ and $\{2, 2, 1\}$, and both have $3$ permutations, 
    %e.g., $(3,1,1), (1,3,1), (1,1,3)$, 
    the total number of ordered $3$-partitions of $5$ is $6$.
    Thus the passports satisfying the condition $\pp_2 \xrightarrow{\Lambda_2} \ppqu$ can be taken as
    $\pp_2 = (0, 5, 5; (3, 1, 1) , (3, 1, 1))$ wiht $6 \times 6 = 36$ in total.
    And the number of weighted trees is $\abs{\Mc(\pp_2)} = 4$.

    \item When the cyclic data is $\Lambda_3 = ((0) , (2))$, $\pp_3 = (0, 5, 5; (5) , (1,1,1,1,1))$, $\abs{\Mc(\pp_3)} = 24$.
\end{enumerate}

Based on the above results, we have
\begin{equation*}
    \abs{\Mc_R(\ppqu)} = \frac{5}{2^{2 \times 2}} (5^{-1} \times 24 + (3 \times 3)^{-1} \times 36 \times 4 + 5^{-1} \times 24) = 8 \ .
\end{equation*}

\bigskip

Next, use Theorem \ref{thm:unrooted_QU} to compute the number of unrooted maps.
Consider $l$, $\sigma$, $\ppsi$ that satisfy condition $l \ |\ 5, 2 = \sigma \times l, \ppqu = \sigma \times_l \ppsi$.
Here we write $\ppqu = (2, 5; 5 \ ,\ 5 \ ,\ 5)$
\begin{enumerate}
    \item When $l = 1$, $\sigma_1 = (2; \varnothing)$ and $\pp_{\sigma_1} = \ppqu$.
    Here $\abs{\Mc_R(\ppqu)} = 8$ and $\abs{\Epi_0(\Gamma(\sigma_1), \ZZ_1)} = 1$.

    \item When $l = 5$, only $\sigma_2 = (0; 5^3)$ satisfies the condition. Then
    $\pp_{\sigma_2} = (0, 1; {_5 1} \ ,\ {_5 1} \ ,\ {_5 1})$ by Notation \ref{notat:sigma_labeled_passport}.
    We have $\big| \Mc_R(\pp_{\sigma_2}) \big| = 1$ and 
    $\abs{\Epi_0 (\Gamma(\sigma_2), \ZZ_5)} = 12$ by Proposition \ref{prop:Epi_count}.
\end{enumerate}

Thus
\begin{equation*}
    \abs{\Mc(\ppqu)} = \frac15 (8 \times 1 + 1 \times 12) = 4 \ .
\end{equation*}

\bigskip

Table \ref{tab:xy_example} list the rooted algebraic maps to verify these two results $\abs{\Mc_R(\ppqu)} = 8, \abs{\Mc(\ppqu)} = 4$.

\textup{
\setlength{\tabcolsep}{20pt}
\begin{table}[ht]
  \centering
  \begin{tabular}{cc c}
    \toprule
    No.    &       $x$       &       $y$       \\
    \midrule
    1       & $(1,3,2,5,4)$   & $(1,5,3,2,4)$   \\
    2       & $(1,3,5,2,4)$   & $(1,5,4,3,2)$   \\
    3       & $(1,3,5,4,2)$   & $(1,3,2,5,4)$   \\
    4       & $(1,4,2,5,3)$   & $(1,4,2,5,3)$   \\
    5       & $(1,4,3,5,2)$   & $(1,3,5,4,2)$   \\
    6       & $(1,5,2,4,3)$   & $(1,4,3,5,2)$   \\
    7       & $(1,5,3,2,4)$   & $(1,5,2,4,3)$   \\
    8       & $(1,5,4,3,2)$   & $(1,3,5,2,4)$   \\
    \bottomrule
  \end{tabular}
  \caption{8 rooted algebraic maps of passport $\ppqu = (2, 5, 5; 5\ ,\ 5)$}
  \label{tab:xy_example}
\end{table}
}
%For rooted maps, we fix the edge set $E = [5]$ and the face permutation $yx = (1,2,3,4,5)$, so that $y(e) = x^{-1}(e) + 1$.
%We only need $y(e) \neq e \iff x(e) \neq e+1$ and $x$ to be a cycle to ensure that both $x$ and $y$ are cycles, thus conforming to the passport $\ppqu$. If $y$ were not a cycle, it could only have the cycle type $(3 \ 2)$, which would contradict the fact that $c(x) +c(y) + c(yx) - n = 1 + 2 + 1 - 5 = -1$ is even.

%In summary, enumeration yields $8$ pairs of $x, y$ (see Table \ref{tab:xy_example}), verifying that the total number of rooted maps $\abs{\Mc_R(\ppqu)} = 8$.

%\bigskip

Consider the rooted maps in Table \ref{tab:xy_example}. 
For the No. 1, 6, 3, 7, 5 algebraic maps, their permutations $x, y$ differ by conjugation of $\phi = (1,2,3,4,5)$ with adjacent maps, so they are isomorphic as unrooted maps. 
However, the No. 2, 4, 8 maps are self-conjugate, then are not isomorphic to the other maps. 
We can see that there are only $4$ non-isomorphic unrooted maps in Table \ref{tab:xy_example}, verifying that $\abs{\Mc(\ppqu)} = 4$.
\end{example}

\bigskip

\begin{example} \label{eg:ppqu_66}
Take the quasi-one-face passport $\ppqu = (2, 5, 6; 6\ ,\ 6)$.

First, use Theorem \ref{thm:CFF_general_passport_formula} to compute the number of rooted maps.

When the cyclic data is $\Lambda = ((2) , (0))$, the passports are $\pp = (0, 5, 6; (2,1,1,1,1) ,\\ (6))$, etc., giving $5$ in total (by permuting $(2,1,1,1,1)$).
And $\abs{\Mc(\pp)} = 24$.
The case $\Lambda = ((0), (2))$ is similar.

When the cyclic data is $\Lambda = ((1) , (1))$,
the $3$-partitions of $6$ are $\{4, 1, 1\}$, $\{3, 2, 1\}$ and $\{2, 2, 2\}$.
And the number of passports and trees are as follows

\begin{enumerate}
    \item Passports $\pp = (0,  5, 6; (4,1,1) , (4,1,1))$, $9$ in total, $4$ trees;
    \item Passports $\pp = (0,  5, 6; (4,1,1) , (3,2,1))$, $36$ in total, $8$ trees;
    \item Passports $\pp = (0,  5, 6; (4,1,1) , (2,2,2))$, $6$ in total, $12$ trees;
    \item Passports $\pp = (0,  5, 6; (3,2,1) , (3,2,1))$, $36$ in total, $7$ trees;
    \item Passports $\pp = (0,  5, 6; (3,2,1) , (2,2,2))$, $12$ in total, $6$ trees;
    \item Passports $\pp = (0,  5, 6; (2,2,2) , (2,2,2))$, $1$ in total, $0$ trees;
\end{enumerate}

Based on the above results, we have
\begin{equation*}\begin{split}
    \abs{\Mc_R(\ppqu)} =& \frac{6}{2^{2 \times 2}} \bigg( 5^{-1} \times 24 \times 2 + (3 \times 3)^{-1} \times \\ 
    &(4 \times 9 + 8 \times 36 + 12 \times 6 + 7 \times 36 + 6 \times 12 + 0 \times 1) \bigg) = 48 \ .
\end{split}\end{equation*}

Next, use Theorem \ref{thm:unrooted_QU} to compute the number of unrooted maps.
Since there is no non-trivial $l > 1$ satisfying $l \ |\ 5$ (by Lemma \ref{lem:QU_divided_also_QU}) and $l \ |\ 6$, the passport $\ppqu = (2, 6; 6 \ ,\ 6 \ ,\ 5\ 1)$ cannot be expressed as a non-trivial $\sigma \times_l \ppsi$. This also shows that the maps in $\Mc(\ppqu)$ have only trivial automorphisms, so
\begin{equation*}
    \abs{\Mc(\ppqu)} = \frac16 \abs{\Mc_R(\ppqu)} = 8 \ .
\end{equation*}

\bigskip

We use the conclusion of Example \ref{eg:ppqu_55} to verify $\abs{\Mc(\ppqu)} = 8$.
We already know that $\ppqu' = (2, 5, 5; 5 \ ,\ 5)$ has $\abs{\Mc_R(\ppqu')} = 8$. 
Now show that $\Mc(\ppqu) \cong \Uc_W(\ppqu) \cong \Mc_R(\ppqu')$.

It is known that a map $M$ in $\Uc_W(\ppqu)$ has $m = 5$ edges, but the total weight is $n = 6$,
so $M$ has $1$ edge of weight $2$ and $4$ edges of weight $1$.
If we take the edge of weight $2$ as the root edge and erase all weights, it corresponds to a rooted map $(M', e)$ in $\Mc_R(\ppqu')$ having $5$ edges and all of weight $1$.
This proves $\Uc_W(\ppqu) \cong \Mc_R(\ppqu')$.
\end{example}

\end{appendix}

\end{document}